\documentclass{amsart}

\usepackage{amssymb}
\usepackage{paralist}
\usepackage[only,llbracket,rrbracket]{stmaryrd}
\usepackage{graphicx}
\usepackage{vmargin}

\setpapersize{A4}
\setmarginsrb{2.7cm}{2.7cm}{2.7cm}{2.7cm}{10pt}{20pt}{10pt}{20pt}

\theoremstyle{plain}
\newtheorem{theorem}{Theorem}
\newtheorem{corollary}[theorem]{Corollary}
\newtheorem{lemma}[theorem]{Lemma}
\newtheorem{facts}[theorem]{Facts}

\theoremstyle{remark}
\newtheorem{remark}[theorem]{Remark}
\newtheorem{remarks}[theorem]{Remarks}

\theoremstyle{definition}
\newtheorem{example}[theorem]{Example}
\newtheorem{examples}[theorem]{Examples}
\newtheorem{defn}[theorem]{Definition}
\newtheorem{defns}[theorem]{Definitions}

\newcommand{\Q}{{\mathbb{Q}}}
\newcommand{\N}{{\mathbb{N}}}
\newcommand{\R}{{\mathbb{R}}}
\newcommand{\Z}{{\mathbb{Z}}}

\newcommand{\cC}{{\mathcal{C}}}
\newcommand{\cD}{{\mathcal{D}}}
\newcommand{\cF}{{\mathcal{F}}}
\newcommand{\cI}{{\mathcal{I}}}
\newcommand{\IM}{{\mathcal{IM}}}
\newcommand{\cJ}{{\mathcal{J}}}
\newcommand{\cM}{{\mathcal{M}}}
\newcommand{\cN}{{\mathcal{N}}}
\newcommand{\cO}{{\mathcal{O}}}
\newcommand{\cP}{{\mathcal{P}}}
\newcommand{\cR}{{\mathcal{R}}}
\newcommand{\cW}{{\mathcal{W}}}
\newcommand{\cX}{{\mathcal{X}}}

\newcommand{\bal}{{\boldsymbol\alpha}}
\newcommand{\bbeta}{{\boldsymbol\beta}}
\newcommand{\bgamma}{{\boldsymbol\gamma}}
\newcommand{\ba}{{\mathbf a}}
\newcommand{\bA}{{\mathbf A}}
\newcommand{\bC}{{\mathbf C}}
\newcommand{\be}{{\mathbf e}}
\newcommand{\bm}{{\mathbf m}}
\newcommand{\bn}{{\mathbf n}}
\newcommand{\bN}{{\mathbf N}}
\newcommand{\bp}{{\mathbf p}}
\newcommand{\br}{{\mathbf r}}
\newcommand{\bv}{{\mathbf v}}

\newcommand{\CI}{\operatorname{CI}}
\newcommand{\DF}{\operatorname{DF}}
\newcommand{\df}{\operatorname{df}}
\newcommand{\E}{\operatorname{E}}
\newcommand{\EP}{\operatorname{EP}}
\newcommand{\FE}{\E'}
\newcommand{\In}{\operatorname{In}}
\newcommand{\Rat}{\operatorname{Rat}}
\newcommand{\Reg}{\operatorname{Reg}}
\newcommand{\rhe}{\operatorname{rhe}}
\newcommand{\Ve}{\operatorname{Vert}}

\newcommand{\Word}[2]{#1\left\llbracket#2\right\rrbracket}

\newcommand{\closure}[1]{\operatorname{Cl}\left(#1\right)}
\newcommand{\LL}[2]{L_{#1}^{(#2)}}
\newcommand{\VV}[2]{\bal_{#1}^{(#2)}}

\graphicspath{{./}{figures/}}
\begin{document}

\title{On digit frequencies in $\beta$-expansions}
\date{June 2014}
\author{Philip Boyland}
\address{Department of
    Mathematics\\ University of Florida\\ 372 Little Hall\\ Gainesville\\
    FL 32611-8105, USA}
\email{boyland@ufl.edu}
\author{Andr\'e de Carvalho}
\address{Departamento de
    Matem\'atica Aplicada\\ IME-USP\\ Rua Do Mat\~ao 1010\\ Cidade
    Universit\'aria\\ 05508-090 S\~ao Paulo SP\\ Brazil}
\email{andre@ime.usp.br}
\author{Toby Hall}
\address{Department of Mathematical Sciences\\ University of
    Liverpool\\ Liverpool L69 7ZL, UK}
\email{tobyhall@liv.ac.uk}

\thanks{ The authors are grateful for the support of FAPESP grants
  2010/09667-0 and \mbox{2011/17581-0}. This research has also been supported
  in part by EU Marie-Curie IRSES Brazilian-European partnership in
  Dynamical Systems (FP7-PEOPLE-2012-IRSES 318999 BREUDS)}

\subjclass[2010]{Primary 11A63, 
Secondary 37B10, 
68R15
}

\begin{abstract}
We study the sets $\DF(\beta)$ of digit frequencies of
$\beta$-expansions of numbers in~$[0,1]$. We show that $\DF(\beta)$ is
a compact convex set with countably many extreme points which varies
continuously with~$\beta$; that there is a full measure collection of
non-trivial closed intervals on each of which $\DF(\beta)$ mode locks
to a constant polytope with rational vertices; and that the generic
digit frequency set has infinitely many extreme points, accumulating
on a single non-rational extreme point whose components are rationally
independent. 
\end{abstract}  

\maketitle


\section{Introduction}
\label{sec:intro}
\subsection{$\beta$-expansions}

Let~$\beta>1$ be a real number and write~$k = \lceil \beta\rceil$, the
smallest integer which is not less than~$\beta$.
A {\em $\beta$-expansion}~\cite{Renyi} of a number $x\in[0,1]$ is any
representation of~$x$ of the form
\[x=\sum_{r=0}^\infty w_r\beta^{-(r+1)},\]
in which the sequence~$w=(w_r)_{r\ge 0}$ of digits belongs to~$\Sigma_k =
\{0,1,\ldots,k-1\}^{\N}$. 

In general, a number~$x$ can have many distinct
$\beta$-expansions~\cite{Erdos,Sidorov}. However there is a canonical
choice, known as the {\em greedy $\beta$-expansion}, or simply as {\em
  the} \mbox{$\beta$-expansion}, for which the sequence $w\in\Sigma_k$ of
digits is lexicographically greatest. By analogy with the usual
algorithm for determining expansions to integer bases, it is found by
choosing each digit in turn to be as large as possible. To be precise,
except in the trivial case where~$\beta$ is an integer and~$x=1$,
if $f_\beta\colon[0,1]\to[0,1]$ is defined by $f_\beta(x) = \beta
x\bmod 1$, then the sequence $d_\beta(x)\in\Sigma_k$ of digits of the
greedy $\beta$-expansion of~$x$ is given by $d_\beta(x)_r = \lfloor
\beta f_\beta^r(x)\rfloor$, the integer part of
$\beta f_\beta^r(x)$. Equivalently, $d_\beta(x)$ is the itinerary of~$x$
under $f_\beta$ with respect to the intervals $I_j$ ($0\le j\le k-1$)
defined by $I_j = [j/\beta, (j+1)/\beta)$ for $0\le j < k-1$, and
  \mbox{$I_{k-1} = [(k-1)/\beta, 1]$}: that is, $d_\beta(x)_r = j$ if
  and only if $f_\beta^r(x)\in I_j$.

\subsection{The digit frequency set}
Let \[\Delta = \left\{\bal\in\R_{\ge0}^k\,:\,\sum_{i=0}^{k-1}\alpha_i = 1\right\}\] be
the standard $(k-1)$-simplex. Given~$\bal\in\Delta$, we say that a
number~$x\in[0,1]$ {\em has $\beta$-digit frequency~$\bal$}, and write
$\delta_\beta(x) = \bal$, if $\lim_{r\to\infty} N_{i,r}(d_\beta(x))/r =
\alpha_i$ for each~$i$, where $N_{i,r}(d_\beta(x))$ denotes the number
of $i^\text{s}$ in the first~$r$ entries
of~$d_\beta(x)$. Parry~\cite{Parry} observed that since the digit
frequency can be written as a Birkhoff sum, it exists and is constant
for almost every $x\in[0,1]$ with respect to any ergodic invariant
measure, such as the measure of maximal entropy which he himself
defined.

In this paper we study the sets
\[\DF(\beta) = \{\delta_\beta(x)\,:\,x \in[0,1] 
\text{ has well-defined $\beta$-digit frequency}\}\]
of {\em all} digit frequencies of $\beta$-expansions of numbers
in~$[0,1]$. For example, Figure~\ref{fig:short}
depicts~$\DF(\beta)$ for $\beta = 2.1901$ (and indeed for all~$\beta$
in a neighbourhood of this value), projected into the
$(\alpha_0,\alpha_2)$-plane, and shown within the
$2$-simplex~$\Delta$. It is a pentagon, with vertices~$(0,1,0)$,
$(1,0,0)$, $(3/4,0,1/4)$, $(5/8,1/8,2/8)$, and $(4/9,3/9,2/9)$ (see
Examples~\ref{ex:ep}a), \ref{ex:explicit-beta}, and
\ref{ex:calc-bn-w}). So, for example, a $\beta$-expansion with this
$\beta$ can have at most $1/4$ of its digits equal to~2; and if it has
this many $2^\text{s}$, then at most $1/8$ of its digits can be equal to~1.

\begin{figure}[htbp]
\begin{center}
\includegraphics[width=0.4\textwidth]{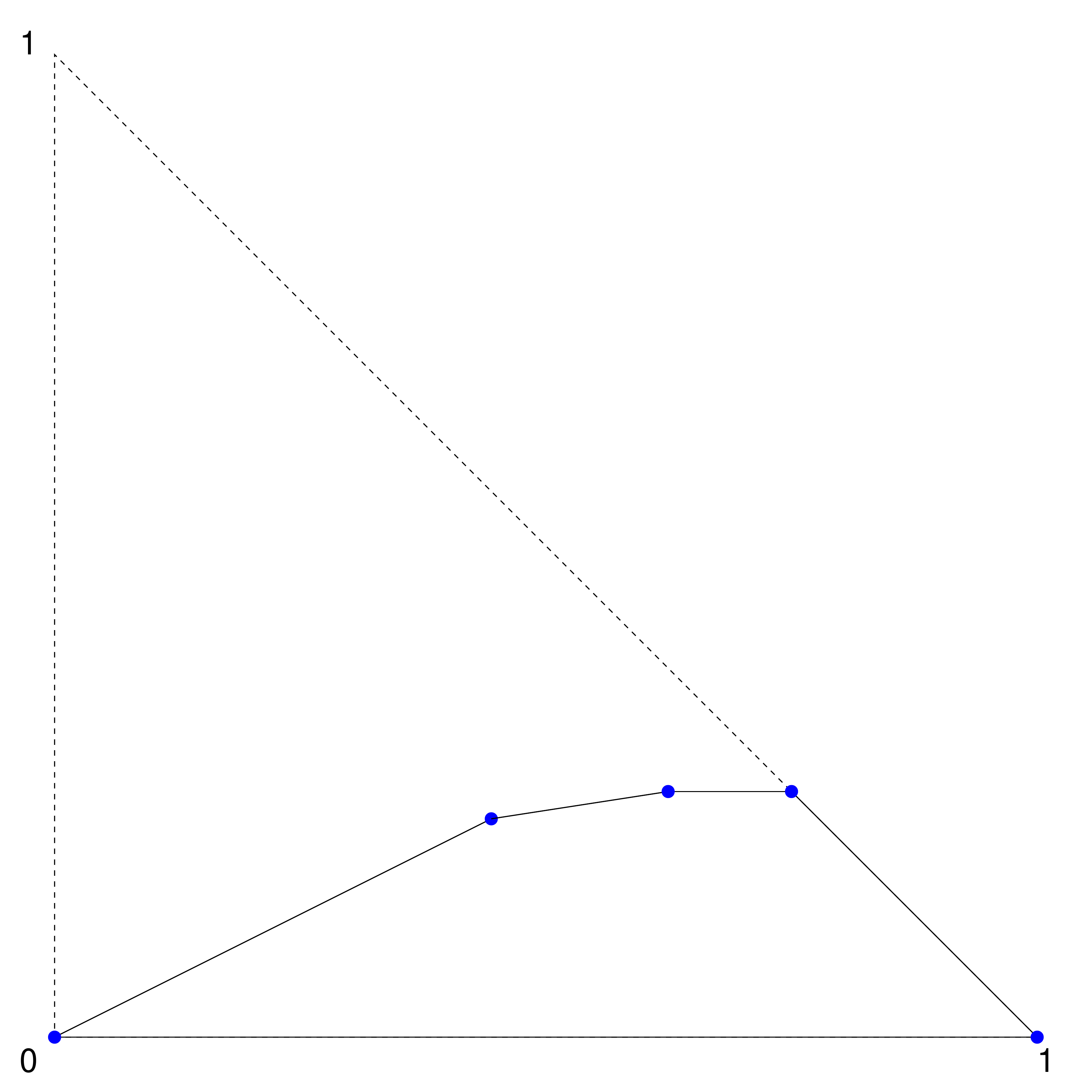}
\caption{$\DF(\beta)$ for~$\beta$ in a neighbourhood of $2.1901$
  (so~$k=3$), projected into the $(\alpha_0,\alpha_2)$-plane}
\label{fig:short}
\end{center}
\end{figure}

\subsection{Bifurcations of the digit frequency set}
When $1<\beta< 2$ (i.e., when $k=2$), the digit frequency set is a
compact interval contained in the line $\alpha_0+\alpha_1=1$,
\[\DF(\beta) = \{(1-\alpha_1, \alpha_1)\,:\,\alpha_1\in
     [0,\rhe(\beta)]\}.\] As $\beta$ increases through $(1,2)$, the
     right hand endpoint $\rhe(\beta)$ varies as a devil's staircase
     (Figure~\ref{fig:k=2}). That is, for each rational $m/n\in(0,1)$
     there is a non-trivial interval $I_{m/n}\subset (1,2)$ with
     $\rhe(\beta) = m/n$ for all $\beta\in I_{m/n}$
     (``mode-locking''). On the other hand, each irrational in~$(0,1)$
     is equal to $\rhe(\beta)$ for only one~$\beta$; such
     ``irrational'' values of~$\beta$ are buried points of a Cantor
     set whose complementary gaps are the interiors of the
     mode-locking intervals. Bifurcations in the digit frequency set
     take place at those values of~$\beta$ for which the itinerary
     of~$1$ under~$f_\beta$ is a Sturmian sequence $s_\alpha$ for some
     $\alpha\in(0,1)$: this is the smallest value of $\beta$ for which
     the frequency $(1-\alpha, \alpha)$ belongs to $\DF(\beta)$. This
     situation is well understood, and can be derived from the theory
     of rotation intervals of bimodal degree one circle maps.

When $\beta\in(k-1,k)$ for $k\ge 3$ the situation, while somewhat
analogous, is more complicated. Now $\DF(\beta)$ behaves as a convex
set-valued devil's staircase, mode-locking to a polytope with rational
vertices on each of an infinite collection of closed subintervals of
$(k-1,k)$. ``Irrational'' behaviour again occurs when $\beta$ is a
buried point of the Cantor set whose complementary gaps are the
interiors of the mode-locking intervals. For such values of~$\beta$,
the digit frequency set $\DF(\beta)$ has a countably infinite set of
rational vertices, which limits on a finite number of non-rational
extreme points, of which there are at least~$1$ and at most~$k-1$ (see
Figures~\ref{fig:cubes} and~\ref{fig:squares}). $\beta$ is said to be
{\em regular} when there is only~$1$ non-rational extreme point (or when
there are none, i.e. when $\DF(\beta)$ is a polytope); and is said to
be {\em exceptional} otherwise.

As in the case~$k=2$, bifurcations of the digit frequency set occur
when the itinerary of~$1$ under~$f_\beta$ passes through elements of a
particular set of sequences: these are the {\em lexicographic infimax
  sequences} of~\cite{lex}, which are described in
Section~\ref{sec:infimax}. In the regular case the infimax sequence
has a well-defined digit frequency, which becomes an element of the
digit frequency set. In the exceptional case, on the other hand, the
convex hull of the set of non-rational extreme points is a ``prime''
collection of frequencies which is either contained in or disjoint
from $\DF(\beta)$ for every~$\beta$ (see Remark~\ref{rmk:order}).

The topologically generic behaviour, with respect to the
parameter~$\beta$, is therefore that $\DF(\beta)$ is a polytope with
rational vertices. 
We shall also show
(Theorem~\ref{thm:polytope-typical}) that $\DF(\beta)$ is a polytope
with rational vertices for Lebesgue almost every~$\beta$ in
$(k-1,k)$. This contrasts, however, with the generic
behaviour in the collection of all digit frequency sets with the
Hausdorff topology. This space is homeomorphic to an interval, and we
show (Theorem~\ref{thm:typical-regular}) that in this interval,
the generic digit frequency set is of non-rational regular
type, with its single limiting extreme point having components that are
independent over the rationals.

\subsection{Summary of results}
The following list summarises the main results of the paper, which are
contained in Theorems~\ref{thm:beta-exp-props},
\ref{thm:polytope-typical}, and~\ref{thm:typical-regular}. The
statements are for $\beta\in(k-1,k)$, where $k\ge 3$: the simpler
situation when $\beta\in(1,2)$ is discussed in
Example~\ref{ex:k=2}. The same statements hold for $\beta\in[k-1,k]$
(except that $\DF(\beta)$ has dimension $k-2$ if $\beta=k-1$),
but it avoids technical issues to restrict to non-integer
values of $\beta$.
\begin{itemize}
\item $\DF(\beta)$ is a compact convex set of dimension~$k-1$.
\item The function $\beta\mapsto\DF(\beta)$ is increasing, and is
  continuous with respect to the Hausdorff topology on the set of
  non-empty compact subsets of~$\Delta$. 
\item $\DF(\beta)$ has countably many extreme points, and there is an
  algorithm which lists them.
\item All but at most $k-1$ extreme points of~$\DF(\beta)$ are
  rational. There exist~$\beta$ for which the set of extreme points
  accumulates on $k-1$ non-rational points.
\item There are infinitely many disjoint closed intervals, whose union
  has full Lebesgue measure in~$(k-1,k)$, on each of which
  $\DF(\beta)$ mode locks to a constant polytope with rational
  vertices.
\item Digit frequency sets which are not polytopes are realised by
  only one value of~$\beta$.
\item The set $\cD := \{\DF(\beta)\,:\,\beta\in(k-1,k)\}$  with
  the Hausdorff topology is homeomorphic to an interval. There is a
  dense $G_\delta$ subset of $\cD$ consisting of digit frequency sets
  having a single non-rational extreme point, whose components are
  rationally independent.
\end{itemize}

\subsection{Outline of the paper}
Let~$Z_\beta=\{d_\beta(x)\,:\,x\in[0,1]\}$ be the set of all digit
sequences of greedy $\beta$-expansions, or equivalently the set of all
itineraries of orbits of $f_\beta$. The set~$Z_\beta$ is determined by
the ``kneading sequence'' of $f_\beta$, the itinerary of the rightmost
point~$1$, once a minor correction has been made to account for the
ambiguity of coding at the endpoints of the intervals $I_j$ -- this is
the same ambiguity which arises in expansions to integer
bases. To describe this correction, define $w_\beta\in\Sigma_k$ by
\[
w_\beta = 
\begin{cases}
\overline{d_1\ldots d_{r-1}(d_r-1)} & \text{ if }d_\beta(1) =
d_1\ldots d_{r-1}d_r\,\overline{0} \text{ for some~$r\ge1$ with }d_r>0,\\
d_\beta(1) & \text{ otherwise,}
\end{cases}
\]
where the overbar denotes infinite repetition. Then (see for example~\cite{ACW}
Theorem~7.2.9)
\[Z_\beta = \{v\in\Sigma_k\,:\,\sigma^r(v) < w_\beta \text{ for all }
  r\in\N\} \cup \{d_\beta(1)\},\] 
where~$<$ is the lexicographic order on~$\Sigma_k$.  

The problem of determining $\DF(\beta)$ can therefore be
rephrased as follows: for which $\bal\in\Delta$ is there some
$v\in\Sigma_k$ with digit frequency~$\bal$, whose entire
$\sigma$-orbit is less than~$w_\beta$? Techniques for answering this
type of question were developed in~\cite{lex}. The results of that
paper will be summarised and extended in Section~\ref{sec:infimax},
and applied in Section~\ref{sec:beta-shift} to the closely related
problem of describing digit frequency sets of {\em symbolic
  $\beta$-shifts}
\[X(w) = \{v\in\Sigma_k\,:\, \sigma^r(v)\le w \text{ for all
}r\in\N\},\] where $w\in\Sigma_k$. In Section~\ref{sec:beta-expansion}
we interpret these results in terms of digit frequency sets of
$\beta$-expansions; describe typical digit frequency sets from both
the measure-theoretic and topological points of view; investigate the
smoothness of digit frequency sets at non-rational extreme points;
discuss the set of accumulation points of the sequences
$(N_{i,r}(d_\beta(1)))_{r\ge 0}$; and present some examples
illustrating how the digit frequency set~$\DF(\beta)$ can be
calculated (or approximated in the non-polytope case) in practice for
a specific value of~$\beta$.

\subsection{The subsequential approach}
Digit frequencies could alternatively be defined subsequentially. Let
\[\delta_\beta'(x) = \{\bal\in\Delta\,:\, \lim_{s\to\infty} N_{i,
  r_s}(d_\beta(x))/r_s = \alpha_i \text{ for some $r_s\to\infty$ and
  each~$i$}\}\] for each $x\in[0,1]$, and set \[\DF'(\beta) =
\bigcup_{x\in[0,1]} \delta_\beta'(x).\] A priori
$\DF'(\beta)$ is a bigger set than $\DF(\beta)$, but it turns out
(Remark~\ref{rmk:subsequential}) that the two are equal for
all $\beta>1$.

\subsection{An example}
\label{sec:markov-example}
When $d_\beta(1)$ is preperiodic, a bare hands calculation of
$\DF(\beta)$ can be carried out using Markov partition techniques
(compare~\cite{Fried,Ziemian}). For example, Figure~\ref{fig:beta-ex}
shows a Markov partition~$(J_1,J_2,J_3,J_4,J_5)$ for $f_\beta$ when
$d_\beta(1) = 2\,1\,2\,1\,\overline{0}$, so that
\mbox{$\beta^4-2\beta^3-\beta^2-2\beta-1 = 0$} ($\beta\simeq 2.7$). We
see the digit $0$ every time we visit $J_1$, the digit $1$ every time
we visit $J_2$ or $J_3$, and the digit $2$ every time we visit $J_4$
or~$J_5$.

The associated Markov transition graph has~10 {\em minimal loops}
(loops which visit each interval at most once). These are
$\overline{1}$, $\overline{1\,2}$, $\overline{1\,2\,4}$,
$\overline{1\,3\,5\,2}$, $\overline{1\,3\,5\,2\,4}$,
$\overline{1\,4}$, $\overline{1\,5\,2}$, $\overline{1\,5\,2\,4}$,
$\overline{2}$, and $\overline{2\,3\,5}$, with corresponding
digit frequencies $(1,0,0)$, $(1/2,1/2,0)$, $(1/3,1/3,1/3)$,
$(1/4,1/2,1/4)$, $(1/5,2/5,2/5)$, $(1/2,0,1/2)$, $(1/3,1/3,1/3)$,
$(1/4,1/4,1/2)$, $(0,1,0)$, and $(0,2/3,1/3)$ respectively.  The digit
frequency set is obtained by taking the convex hull of these
frequencies: it is a pentagon with vertices $(1,0,0)$, $(1/2,0,1/2)$,
$(1/4,1/4,1/2)$, $(0,1,0)$, and $(0,2/3,1/3)$. See
Example~\ref{ex:markov-again}.

This observation by itself is sufficient, using Theorem~3.4
of~\cite{Ziemian}, to establish that $\DF(\beta)$ is a convex polytope
with rational vertices whenever the $f_\beta$-orbit of~$1$ is finite.
(A straightforward concatenation argument shows that, in the case
where $w_\beta\not=d_\beta(1)$, the ``missing'' digit frequency of
$w_\beta$ is realised as $d_\beta(x)$ for some $x$: this digit
frequency corresponds to a loop~$L$ in the Markov graph, which can be
concatenated with any other intersecting loop~$M$ in the pattern
$LMLLMLLLM\ldots$ to provide the itinerary of a suitable point~$x$.)
However, quite different techniques are needed to address the general
case.

This calculation, and several of the results presented in this
paper, are reminiscent of rotation sets of torus homeomorphisms. The
connection between the two problems will be made explicit in the
authors' forthcoming paper {\em ``New rotation sets in a family of
  torus homeomorphisms''}.

\begin{figure}[htbp]
\begin{center}
\includegraphics[width=0.4\textwidth]{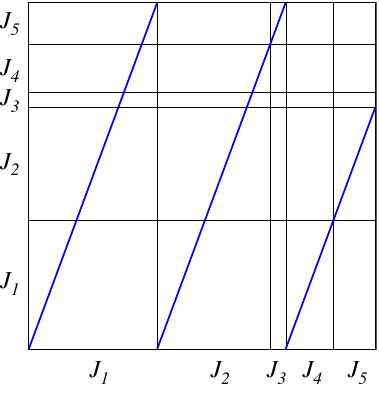}
\caption{A Markov partition for~$f_\beta$ when $d_\beta(1) = 2\,1\,2\,1\,\overline{0}$}
\label{fig:beta-ex}
\end{center}
\end{figure}

We note also the connection of this work with the {\em digit
  distribution problem} first studied by Besicovitch~\cite{Bes} and
Eggleston~\cite{Eggleston}. For expansions to {\em integer} bases~$N$
and elements $\bal$ of the $(N-1)$-simplex~$\Delta$, they considered the
properties (for example, the Hausdorff dimension) of the
sets \[H_{\bal} = \{x\in [0,1]\,:\,\delta_N(x) = \bal\}.\] The results
presented here enable one to determine, for arbitrary bases~$\beta$,
the values of $\bal$ for which $H_{\bal}$ is non-empty.


\section{Infimax sequences}
\label{sec:infimax}

\subsection{Notation and summary of results from \protect\cite{lex}}
Let~$k\ge 2$ be an integer, and \mbox{$\Sigma_k =
\{0,1,\ldots,k-1\}^\N$}, the set of sequences $w = (w_r)_{r\ge 0}$ with
entries in $\{0,1,\ldots,k-1\}$, endowed with the lexicographic order
and the product topology (we consider the natural numbers~$\N$ to
include~$0$). The suffix $k$ will generally be suppressed, both
on~$\Sigma_k$ and on other $k$-dependent objects. We refer to
elements of the alphabet $\{0,1,\ldots,k-1\}$ as {\em digits}, since
elements of~$\Sigma$ will be interpreted as digit sequences of
$\beta$-expansions. 

Denote by $\cW$ the set of non-trivial finite words $W$ over the
alphabet $\{0,1,\ldots,k-1\}$, ordered lexicographically with the
convention that any proper initial subword of $W$ is greater
than~$W$. We write $|W|\ge 1$ for the length of a word~$W$.

If $V,W\in\cW$, we write $VW$ for the concatenation of $V$ and $W$;
$W^n$ for the $n$-fold repetition of $W$ (where $n\ge 1$);
$\overline{W} = WWWW\ldots$ for the element of~$\Sigma$ given by
infinite repetition of~$W$; and $V\overline{W}$ for the element
$VWWWW\ldots$ of $\Sigma$. An element of~$\Sigma$ of the form
$\overline{W}$ is said to be {\em periodic}.

If $w\in\Sigma$ and $r\ge 1$, we write $\Word{w}{r} = w_0w_1\ldots
w_{r-1}$ for the word formed by the first~$r$ digits of $w$.

The shift map $\sigma\colon\Sigma\to\Sigma$ is defined by $\sigma(w)_r
= w_{r+1}$. An element $w$ of~$\Sigma$ is said to be {\em maximal} if
it is the maximum element of its $\sigma$-orbit: that is, if
$\sigma^r(w)\le w$ for all $r\ge 0$. We write $\cM$ for the set of
maximal elements of~$\Sigma$.

As in the introduction, let~$\Delta$ denote the standard
$(k-1)$-simplex
\[\Delta =
\left\{\bal\in\R^k_{\ge0}\,:\,\sum_{i=0}^{k-1}\alpha_i=1\right\}\]
with the Euclidean metric~$d$. (In~\cite{lex}, $\Delta$ was endowed
with the maximum metric to ease some of the calculations, but this is
not necessary here.) $B_\epsilon(\bal)$ denotes the open
$\epsilon$-ball about $\bal\in\Delta$. Write $\cC(\Delta)$ for the
space of non-empty compact subsets of~$\Delta$, with the Hausdorff
metric~$d_H$.

Given $\bal\in\Delta$, let $\cR(\bal)\subset\Sigma$ be the set of
sequences with digit frequency~$\bal$, 
\[\cR(\bal) = \{w\in\Sigma\,:\,\lim_{r\to\infty} N_{i,r}(w)/r =
\alpha_i \text{ for each }i\},\]
and $\cM(\bal) = \cM\cap\cR(\bal)$, the set of maximal sequences with
digit frequency~$\bal$.

We begin with a brief summary of necessary results from~\cite{lex}. We
work here over the alphabet $\{0,1,\ldots,k-1\}$, as is appropriate
for digit sequences of $\beta$-expansions, rather than over the
alphabet $\{1,2,\ldots,k\}$ used in~\cite{lex}. A second notational
change is that we write
\[\Delta' = \{\bal\in\Delta\,:\,\alpha_{k-1} \not=0\}\] 
for the set of elements of the standard simplex whose final coordinate
is non-zero: in~\cite{lex} this set, which was the main object of
study, was denoted~$\Delta$, and the set here called~$\Delta$ was
denoted~$\overline{\Delta}$.

Let~$\Delta_n \subset\Delta$ be defined for~$n\ge 0$ by
\[\Delta_n = \left\{\bal\in\Delta'\,:\, \lfloor
\alpha_0/\alpha_{k-1}\rfloor = n\right\},\]
so that the $\Delta_n$ partition $\Delta'$. Define $J\colon
\Delta'\to\N$ by $J(\bal) = \lfloor \alpha_0/\alpha_{k-1}\rfloor$, so
that $\bal\in \Delta_{J(\bal)}$ for each $\bal\in\Delta'$.

We define a multi-dimensional continued fraction map $K\colon
\Delta'\to\Delta'$ by setting \mbox{$K(\bal) = K_{J(\bal)}(\bal)$},
where~$K_n\colon \Delta_n\to\Delta'$ is the projective homeomorphism
given by
\[
K_n(\bal) = \left(\frac{\alpha_1}{1-\alpha_0},\,\,
\frac{\alpha_2}{1-\alpha_0},\,\,
\ldots,\,\,\frac{\alpha_{k-2}}{1-\alpha_0},\,\, \frac{\alpha_0 -
  n\alpha_{k-1}}{1-\alpha_0},\,\,
\frac{(n+1)\alpha_{k-1}-\alpha_0}{1-\alpha_0} \right).
\]
Let~$\Phi\colon\Delta'\to\N^\N$ be the itinerary map of~$K$ with
respect to the partition~$\{\Delta_n\}$. That is, for
each~$\bal\in\Delta'$, the sequence $\Phi(\bal)\in\N^\N$ is defined by
\[\Phi(\bal)_r = J(K^r(\bal)) \qquad (r\in\N).\]
We order $\N^\N$ {\em reverse lexicographically}: if $\bm$ and
$\bn$ are distinct elements of $\N^\N$, then $\bm < \bn$ if and only
if $m_r > n_r$, where~$r$ is the smallest index with $m_r\not=n_r$.

For each~$n\in\N$, let~$\Lambda_n\colon\Sigma\to\Sigma$ and
$\Lambda_n\colon\cW\to\cW$ be the
substitutions defined by
\begin{equation}
\label{eq:lambda}
\Lambda_n\colon \qquad
\left\{
\begin{array}{lll}
  i & \mapsto & (i+1) \qquad\qquad \text{ if }0\le i\le k-3\\
  (k-2) & \mapsto & (k-1)\, 0^{n+1} \\
  (k-1) & \mapsto & (k-1)\, 0^n. 
\end{array}
\right.
\end{equation}
These substitutions are strictly order preserving and satisfy
$\Lambda_n(\cM)\subset\cM$. 

Given $\bn\in\N^\N$, define substitutions $\Lambda_{\bn,r}$ for
each~$r\in\N$ by
\[\Lambda_{\bn,r} = \Lambda_{n_0}\circ\Lambda_{n_1}\circ\cdots
\circ\Lambda_{n_r},\]
and let $S\colon\N^\N\to\cM\subset\Sigma$ be given by
\[S(\bn) = \lim_{r\to\infty}\Lambda_{\bn,r}(\overline{k-1}),\]
the limit existing since $\Lambda_{n_{r+1}}(k-1)$ begins with the
digit~$k-1$, so that $\Lambda_{\bn,r}(k-1)$ is an initial subword of
$\Lambda_{\bn,r+1}(k-1)$ for all~$r$. Finally, let $\cI =
S\circ\Phi\colon \Delta'\to\cM$.

The following results from~\cite{lex} will be used here:
\begin{facts}\mbox{}
\label{facts:infifacts}
\begin{enumerate}[a)]
\item Let $\bal\in\Delta'$. Then $\cI(\bal)$ is the infimum of
  $\cM(\bal)$, the so-called {\em $\bal$-infimax sequence}.
\item Let $\bal\in\Delta'$ and $w\in\cR(\bal)$. Then $\cI(\bal) \le
  \sup_{r\ge 0}\sigma^r(w)$.
\item The itinerary $\Phi(\bal)$ of $\bal$ is of the form
  $n_0\,n_1\,\ldots\, n_r\,\overline{0}$ if and only if $\bal\in\Q^k$.
In this case $\cI(\bal) = \overline{\Lambda_{\bn,r}(k-1)}$, and the
repeating block $B_\bal=\Lambda_{\bn,r}(k-1)$ cannot be written in the
form $B_\bal = W^n$ with $W\in\cW$ and $n>1$.
\item $\Phi\colon\Delta'\to\N^\N$ is lower semi-continuous and
  surjective. It is not injective except when $k=2$: the preimage
  $\Phi^{-1}(\bn)$ of a point $\bn\in\N^\N$ is a $d$-simplex, where
  $0\le d = d(\bn)\le k-2$.
\item $S\colon \N^\N\to\cM$ is a continuous order-preserving
  bijection onto its image~$\IM\subset\cM$, the set of infimax
  sequences.
\item $\cI(\bal)\in\cM(\bal)$ (and hence $\cI(\bal) = \min \cM(\bal)$)
  if and only if $\Phi^{-1}(\Phi(\bal)) = \{\bal\}$. This is always
  the case when $\bal\in\Q^k$. If $\Phi^{-1}(\Phi(\bal))\not=\{\bal\}$
  then $\cI(\bal)$ does not have well-defined digit frequency.
\item If $\cR'(\bal)$ denotes the set of sequences with subsequential
  digit frequency $\bal$, i.e.
\[\cR'(\bal) = \{w\in\Sigma\,:\,\lim_{s\to\infty}N_{i,r_s}(w)/r_s = \alpha_i \text{ for
  some }r_s\to\infty\text{ and each }i\},\]
and $\cM'(\bal) = \cM\cap\cR'(\bal)$, then a) and b) hold in the
primed versions: that is, $\cI(\bal) = \inf\cM'(\bal)$, and $\cI(\bal) \le
\sup_{r\ge 0}\sigma^r(w)$ for all $w\in\cR'(\bal)$.
\item The set of itineraries $\bn$ for which $\Phi^{-1}(\bn)$ is a
  single point contains the dense $G_\delta$ subset~$\cO$ of $\N^\N$
  consisting of those sequences which contain infinitely many distinct
  subwords $1^{2k-3}$. (For $k=3$ this is a result of Bruin and
  Troubetzkoy~\cite{BT}.) 
\item An element $\bal$ of $\Delta'$ has the property that its orbit
  $(K^r(\bal))_{r\ge 0}$ is disjoint from the faces of $\Delta$ if and
  only if its itinerary $\bn=\Phi(\bal)$ has the following property:
  for every $r\ge 0$ there is some $s\ge 0$ with $n_{r+s(k-1)}\not=0$.
\end{enumerate}
\end{facts}

\begin{defns}[Rational type, regular, exceptional]
Let $\bal\in\Delta'$ with itinerary $\bn=\Phi(\bal)$. We say that
$\bal$ and $\bn$ are of {\em rational type} if $\bal\in\Q^k$: that is,
if $\bn = n_0\,n_1\,\ldots\, n_r\,\overline{0}$ for some
$n_0,n_1,\ldots,n_r$. We say that $\bal$ and $\bn$
are {\em regular} if $\Phi^{-1}(\bn) = \{\bal\}$, and that they are
{\em exceptional} otherwise. In the latter case, we refer to the
non-trivial simplex $\Phi^{-1}(\bn)\subset\Delta'$ as an {\em
  exceptional set}. 
\end{defns}

It can be shown~\cite{lex} that, for $k\ge 3$, an element $\bn$ of $\N^\N$
is regular when it grows slowly enough and has only finitely many
zero entries; and that it is exceptional when it grows
rapidly enough. 

\begin{example}
Let $k=3$ and $\bal = (7/16, 5/16, 4/16) \in \Delta'$.  We have
\[(7/16,5/16,4/16) \,\,\stackrel{K_1}{\longrightarrow}\,\, (5/9,3/9,1/9)
\,\,\stackrel{K_5}{\longrightarrow}\,\, (3/4,0,1/4)
\,\,\stackrel{K_3}{\longrightarrow}\,\, (0,0,1),\]
and $K_0(0,0,1) = (0,0,1)$. Therefore $\bal$ has itinerary $\Phi(\bal)
= 1\,5\,3\,\overline{0}$. The $\bal$-infimax sequence is
\[\cI(\bal) = S(1\,5\,3\,\overline{0}) =
\Lambda_1(\Lambda_5(\Lambda_3(\overline{2}))) = \overline{
2\,0\,1\,1\,1\,1\,1\,2\,0\,0\,2\,0\,0\,2\,0\,0}.\] 
This is the smallest maximal sequence with digit frequency~$\bal$.
\end{example}

\subsection{Convergence to the exceptional set}
\label{sec:converge}
In this section we establish information about exceptional sets which
goes beyond that contained in~\cite{lex}. The results are technical,
and their proofs could be omitted on first reading.

We fix an element $\bn$ of $\N^\N$, and begin by describing, as
in~\cite{lex}, a decreasing sequence $(A_{\bn,r})_{r\ge 0}$ of
simplices whose intersection is $\Phi^{-1}(\bn)$.

The homeomorphism $K_n\colon \Delta_n\to\Delta'$ has
inverse given by
\begin{equation}
\label{eq:KnI}
K_n^{-1}(\bal) = \left(\frac{(n+1)\,\alpha_{k-2} + n\,\alpha_{k-1}}{D},
\frac{\alpha_0}{D}, \frac{\alpha_1}{D}, \ldots,\frac{\alpha_{k-3}}{D},
\frac{\alpha_{k-2}+\alpha_{k-1}}{D} 
\right),
\end{equation}
where $D= (n+1)\,\alpha_{k-2} + n\,\alpha_{k-1}+ 1$. The
homeomorphism $K_n^{-1}\colon \Delta' \to \Delta_n$ extends by the
same formula to a homeomorphism
$K_n^{-1}\colon\Delta\to\closure{\Delta_n} \subset\Delta$. 

\begin{defns}[$\Upsilon_{\bn,r}$,\, $A_{\bn,r}$,\, $\cF$,\, $\cF_{\bn,r}$]
For each
$r\in\N$, we define an embedding
\[
\Upsilon_{\bn,r} = K_{n_0}^{-1}\circ K_{n_1}^{-1} \circ \cdots \circ
K_{n_r}^{-1} \colon \Delta\to\Delta.
\]

Let $\cF = \Delta\setminus\Delta'$ denote the face $\alpha_{k-1}=0$
of~$\Delta$, and write
\[A_{\bn,r} = \Upsilon_{\bn,r}(\Delta) \qquad\text{and}\qquad
\cF_{\bn,r} = \Upsilon_{\bn,r}(\cF).\]
\end{defns}
Since each $K_n^{-1}$ is projective, $A_{\bn,r}$ is a $(k-1)$-simplex
and $\cF_{\bn,r}$ is a $(k-2)$-simplex for all~$r$. Moreover, since
$K_{n_r}^{-1}(\Delta) \subset\Delta$, the sequence $(A_{\bn,r})_{r\ge
  0}$ is decreasing.

Now 
\[\Upsilon_{\bn,r}(\Delta') = A_{\bn,r} \setminus \cF_{\bn,r} =
\{\bal\in\Delta'\,:\,\Word{\Phi(\bal)}{r+1} = \Word{\bn}{r+1}\},\] the set of
frequencies whose itineraries agree with~$\bn$ on their first $r+1$
entries.  On the other hand, 
\begin{equation}
\label{eq:cF-itin}
\bal\in\cF_{\bn,r} \cap\Delta' \implies \Word{\Phi(\bal)}{r-i+1} = n_0n_1\ldots
n_{r-i-1}\,(n_{r-i}+1)\text{ for some $i$ with }0\le i\le k-2.
\end{equation}
In particular, $\Phi(\bal) < \bn$ for all $\bal\in\cF_{\bn,r}$ when
$r\ge k-2$, since $\cF_{\bn,r} \subset\Delta'$ for $r\ge k-2$.

Let~$\bal\in\Delta'$. If $\Phi(\bal)=\bn$ then $\bal\in A_{\bn,r}$ for
all~$r$. On the other hand, if $\Phi(\bal)\not=\bn$, let $r\in\N$ be
such that $\Phi(\bal)_r\not=n_r$: then $\bal\not\in\cF_{\bn,r+k}$
by~(\ref{eq:cF-itin}), and hence $\bal\not\in A_{\bn,r+k}$. We
therefore have
\begin{itemize}
\item every element of~$\Delta'$ whose itinerary starts
  $n_0\,\ldots\,n_r$ lies in $A_{\bn,r}$, and
\item if $r\ge k$, then every element of $A_{\bn,r}$ has itinerary
  starting $n_0\,\ldots\,n_{r-k}$.
\end{itemize}
In particular,
\[\Phi^{-1}(\bn) = \bigcap_{r\ge 0} A_{\bn,r},\]
and the decreasing sequence $(A_{\bn,r})$ converges Hausdorff to
$\Phi^{-1}(\bn)$.  Moreover, since $(A_{\bn,r})$ is a decreasing
sequence of simplices, it follows by a theorem of
Borovikov~\cite{Borovikov} that $\Phi^{-1}(\bn)$ is also a simplex. By
Facts~\ref{facts:infifacts}c), this simplex cannot have interior
in~$\Delta$ (since then there would be both rational and non-rational
points having itinerary~$\bn$), and so it has dimension at most~$k-2$.

The following lemma plays a key r\^ole in the proofs of
Lemma~\ref{lem:describe-lfs} and Theorem~\ref{thm:continuous}, two of
the central results of the paper.

\begin{lemma}
\label{lem:face-converge}
Let~$\bn\in\N^\N$. Then for every~$\epsilon>0$, there are infinitely
many~$r$ with $d_H(\cF_{\bn,r},\, \Phi^{-1}(\bn))<\epsilon$.
\end{lemma}
\begin{proof} 
If $\bn$ is regular then $\Phi^{-1}(\bn)$ is a point and $\cF_{\bn,r}
\subset A_{\bn,r} \to \Phi^{-1}(\bn)$, so the result is immediate. We
can therefore assume that~$\bn$ is exceptional. In particular, $\bal$
is not of rational type, and hence $n_r\not=0$ for infinitely many~$r$.

Let $\bv_1,\ldots,
\bv_{n}$ ($2\le n\le k-1$) be the vertices of the simplex
$\Phi^{-1}(\bn)$. Write $\be_0,\ldots \be_{k-1}$ for the vertices
of~$\Delta$ (so that the $i^\text{th}$ component of $\be_i$ is $1$),
and let $\bal_r^{(i)} = \Upsilon_{\bn,r}(\be_i)$. Therefore the
vertices of $\cF_{\bn,r}$ are $\bal_r^{(i)}$ for $0\le i\le k-2$, and
$A_{\bn,r}$ has the additional vertex $\bal_r^{(k-1)}$. Notice, by
comparison of~(\ref{eq:lambda}) and~(\ref{eq:KnI}), that
$\bal_r^{(i)}$ is the digit frequency of the word
$\Lambda_{\bn,r}(i)$. We write $L_r^{(i)}$ for the length of
$\Lambda_{\bn,r}(i)$, so that $L_r^{(i)}\bal_r^{(i)}$ is an integer
vector whose entries give the number of each digit in
$\Lambda_{\bn,r}(i)$. 

Since $A_{\bn,r}\to\Phi^{-1}(\bn)$ as $r\to\infty$ we have
$\cF_{\bn,r} \subset A_{\bn,r} \subset
B_\epsilon(\Phi^{-1}(\bn))$ for all sufficiently large~$r$. It remains
to prove that $\Phi^{-1}(\bn) \subset
B_\epsilon(\cF_{\bn,r})$ for infinitely many~$r$, and for
this it is
enough to show that, for infinitely many~$r$, every
vertex of~$\Phi^{-1}(\bn)$ is approximated by a vertex of
$\cF_{\bn,r}$: that is,
\[
\forall\epsilon>0,\,\,\,\forall R,\,\,\,\exists r\ge R,\,\,\, \forall m\le n,\,\,\,
\exists i\le k-2,\,\,\, \bal_r^{(i)} \in B_\epsilon(\bv_m).
\]

Suppose for a contradiction that there exist $\epsilon>0$ and $R$ such
that for all $r\ge R$ there is some~$m$ for which $B_\epsilon(\bv_m)$
doesn't contain $\bal_r^{(i)}$ for any $i\le k-2$. Decrease~$\epsilon$
if necessary so that the distance between any two vertices of
$\Phi^{-1}(\bn)$ is at least $2\epsilon$; and increase~$R$ if
necessary so that $d_H(A_{\bn,r},\Phi^{-1}(\bn))<\epsilon/5$ for all
$r\ge R$. In particular this means that, for all $r\ge R$, every
$B_{\epsilon/5}(\bv_\ell)$ contains some vertex~$\bal_r^{(i)}$ of
$A_{\bn,r}$.

Pick $r\ge R$ with $n_{r+1}\not=0$. Since there is some~$m$ for which
none of the $\bal_r^{(i)}$ with $i\le k-2$ lie in $B_\epsilon(\bv_m)$,
we must have $\bal_r^{(k-1)}\in B_{\epsilon/5}(\bv_m)$.  

Now~(\ref{eq:lambda}) gives $\bal_{r+1}^{(i)} = \bal_r^{(i+1)}$ for $0\le i\le k-3$, and 
\begin{eqnarray}
\label{eq:alpha-eqs-1}
\bal_{r+1}^{(k-2)} &=& \frac{(n_{r+1}+1)L_r^{(0)}\bal_r^{(0)} +
  L_r^{(k-1)}\bal_r^{(k-1)}}{(n_{r+1}+1)L_r^{(0)} + L_r^{(k-1)}},
\qquad \text{and}\\
\label{eq:alpha-eqs-2}
\bal_{r+1}^{(k-1)} &=& \frac{n_{r+1}L_r^{(0)}\bal_r^{(0)} +
  L_r^{(k-1)}\bal_r^{(k-1)}}{n_{r+1}L_r^{(0)}+L_r^{(k-1)}}.
\end{eqnarray}
Since the only vertices of~$A_{\bn,r+1}$ which are not also
vertices of~$A_{\bn,r}$ are $\bal_{r+1}^{(k-2)}$ and
$\bal_{r+1}^{(k-1)}$, one of these must lie in
$B_{\epsilon/5}(\bv_m)$. However, both lie along the line segment
  joining $\bal_r^{(0)}$ to $\bal_r^{(k-1)}$, and $\bal_{r+1}^{(k-1)}$
  is the closer of the two to $\bal_r^{(k-1)}$. Therefore
  $\bal_{r+1}^{(k-1)}\in B_{\epsilon/5}(\bv_m)$. 

Let~$d_1<d_2$ be the distances from $\bal_r^{(k-1)}$ to
$\bal_{r+1}^{(k-1)}$ and $\bal_{r+1}^{(k-2)}$ respectively. Then
\[
\frac{d_2}{d_1} = \frac{n_{r+1}+1}{n_{r+1}} \, \frac{n_{r+1}L_r^{(0)}
  + L_r^{(k-1)}} {(n_{r+1}+1)L_r^{(0)} + L_r^{(k-1)}} < 2
\]
since $n_{r+1}\ge 1$ by choice of $r$, and so $(n_{r+1}+1)/n_{r+1} \le
2$. However, since both $\bal_r^{(k-1)}$ and $\bal_{r+1}^{(k-1)}$ lie
in $B_{\epsilon/5}(\bv_m)$ we have $d_1 < 2\epsilon/5$, and hence $d_2
< 4\epsilon/5$. Therefore
\[
d(\bal_{r+1}^{(k-2)}, \bv_m) \le d_2 + d(\bal_r^{(k-1)}, \bv_m) <
\frac{4\epsilon}{5} + \frac{\epsilon}{5} = \epsilon.
\]
This is the required contradiction. For since both
$\bal_{r+1}^{(k-2)}$ and $\bal_{r+1}^{(k-1)}$ are within~$\epsilon$ of
$\bv_m$, every other vertex~$\bv_\ell$ of $\Phi^{-1}(\bn)$ must have
$d(\bv_\ell,\bal_{r+1}^{(i)}) < \epsilon/5$ for some $i<k-2$.
\end{proof}

As a consequence, the itinerary map $\Phi$ has no local minima:

\begin{corollary}
\label{cor:ball-less}
Let~$\bal\in\Delta'$. Then for all~$\epsilon>0$ there is some
$\bbeta\in B_\epsilon(\bal)$ with $\Phi(\bbeta) < \Phi(\bal)$.
\end{corollary}
\begin{proof}
Write $\bn = \Phi(\bal)$. Let~$r \ge k-2$ be such that
$d_H(\cF_{\bn,r}, \Phi^{-1}(\bn)) < \epsilon$, so that there is a
point~$\bbeta$ of $\cF_{\bn,r}$ within distance~$\epsilon$ of
$\bal\in\Phi^{-1}(\bn)$. Then $\Phi(\bbeta) < \bn = \Phi(\bal)$
by~(\ref{eq:cF-itin}).
\end{proof}

\subsection{Concatenations of repeating blocks}
We will need the following straightforward result about
concatentations of repeating blocks of rational infimaxes.

\begin{lemma}
\label{lem:concatenate}
Let $\bal\in\Delta'$, and $(\bal_i)$ be a sequence of
rational elements of $\Delta'$ with the property that
$\cI(\bal_i)<\cI(\bal)$ for all $i$. Write $B_i$ for the repeating
block of $\cI(\bal_i)$. Then the sequence
\[
w = B_0\,B_1\,B_2\,\ldots \in\Sigma
\]
satisfies $\sigma^r(w)\le\cI(\bal)$ for all $r\ge 0$.
\end{lemma}
\begin{proof}
We will show that, for each~$i$,
\begin{enumerate}[a)]
\item $B_i$ is strictly smaller than the initial length $|B_i|$
  subword of $\cI(\bal)$; and
\item every proper final subword of $B_i$ is strictly smaller than
  $B_i$,
\end{enumerate}
from which the result follows. Recall that, by convention, any word is
smaller than any of its proper initial subwords, so that the
lexicographic order on the set~$\cW$ of words is total.

\medskip

For~a), since $\overline{B_i}=\cI(\bal_i)<\cI(\bal)$, it is enough to
show that $B_i$ is not an initial subword of $\cI(\bal)$. 

Let $\Phi(\bal_i) = n_0\ldots n_r \overline{0}$. Since
$\cI(\bal_i)<\cI(\bal)$, there is some $j\le r$ such that $\Phi(\bal)
= n_0\ldots n_{j-1}m_j\ldots$, where $m_j<n_j$. Then
$\Lambda_{n_j}(k-1) = (k-1)\,0^{m_j}0\ldots$, and hence $B_i$ has an
initial subword of the form $P_i =
\Lambda_{\bn,j-1}((k-1)\,0^{m_j}0)$. 

Similarly, $\cI(\bal)$ has an initial subword
$P=\Lambda_{\bn,j-1}((k-1)\,0^{m_j}\,s\,W)$, where $s>0$ and~$W$ is a
word long enough to ensure that $|P|>|P'|$.

Since $\Lambda_{\bn,j-1}\colon\cW\to\cW$ is strictly order-preserving,
it follows that $P'<P$. Therefore $P'$ is not an initial subword of
$\cI(\bal)$, and so neither is $B_i$, as required.

\medskip

For~b), suppose for a contradiction that $B_i$ has a proper final
subword $W$ with $W>B_i$. Since
$\cI(\bal_i)=\overline{B_i}$ is a maximal sequence, $W$ must also be an
initial subword of $B_i$. Therefore there are words $U$ and $V$ of the
same length with $B_i = WU = VW$. Then $\cI(\bal_i) = \overline{WU} =
\overline{VW} > \overline{WV}$, with the inequality coming from the
maximality of $\cI(\bal_i)$ and
Facts~\ref{facts:infifacts}c). Therefore $U>V$, so that
$\sigma^{|W|}(\cI(\bal_i))=\overline{UW}>\overline{VW} = \cI(\bal_i)$, contradicting the maximality of
$\cI(\bal_i)$.

\end{proof}

\subsection{Compactification of the space of itineraries}
In this section we describe a compactification of~$\N^\N$ to a Cantor
set~$\cN$, and extend the map $S\colon\N^\N\to\cM$, which
associates an infimax sequence with each itinerary, over $\cN$ as an
order-preserving homeomorphism onto its image. The set~$\cN$ and its
topology and order are modelled by the function~$g$ whose graph is depicted in
Figure~\ref{fig:compactify-model}. In this figure $(L_n)_{n\ge 0}$ is
a sequence of mutually disjoint closed subintervals
of~$[0,1]$, and $g$ maps each $L_n$ affinely and increasingly
onto~$[0,1]$ and satisfies $g(0)=0$. The biggest set
\[C=\bigcap_{r=0}^\infty g^{-r}\left(\{0\} \cup \bigcup_{n=0}^\infty
L_n\right)\]
on which all of the iterates of~$g$ are defined is a Cantor set, and
the points of~$C$ correspond bijectively via their itineraries with
the union of $\N^\N$ and the set of finite words in~$\N$ terminated
with the symbol~$\infty$, which is regarded as the address of
$0$. The gaps in this Cantor set have points whose itineraries end
$\overline{0}$ at their left hand ends, and points whose itineraries are
terminated with~$\infty$ at their right hand ends.

\begin{figure}[htbp]
\begin{center}
\includegraphics[width=0.4\textwidth]{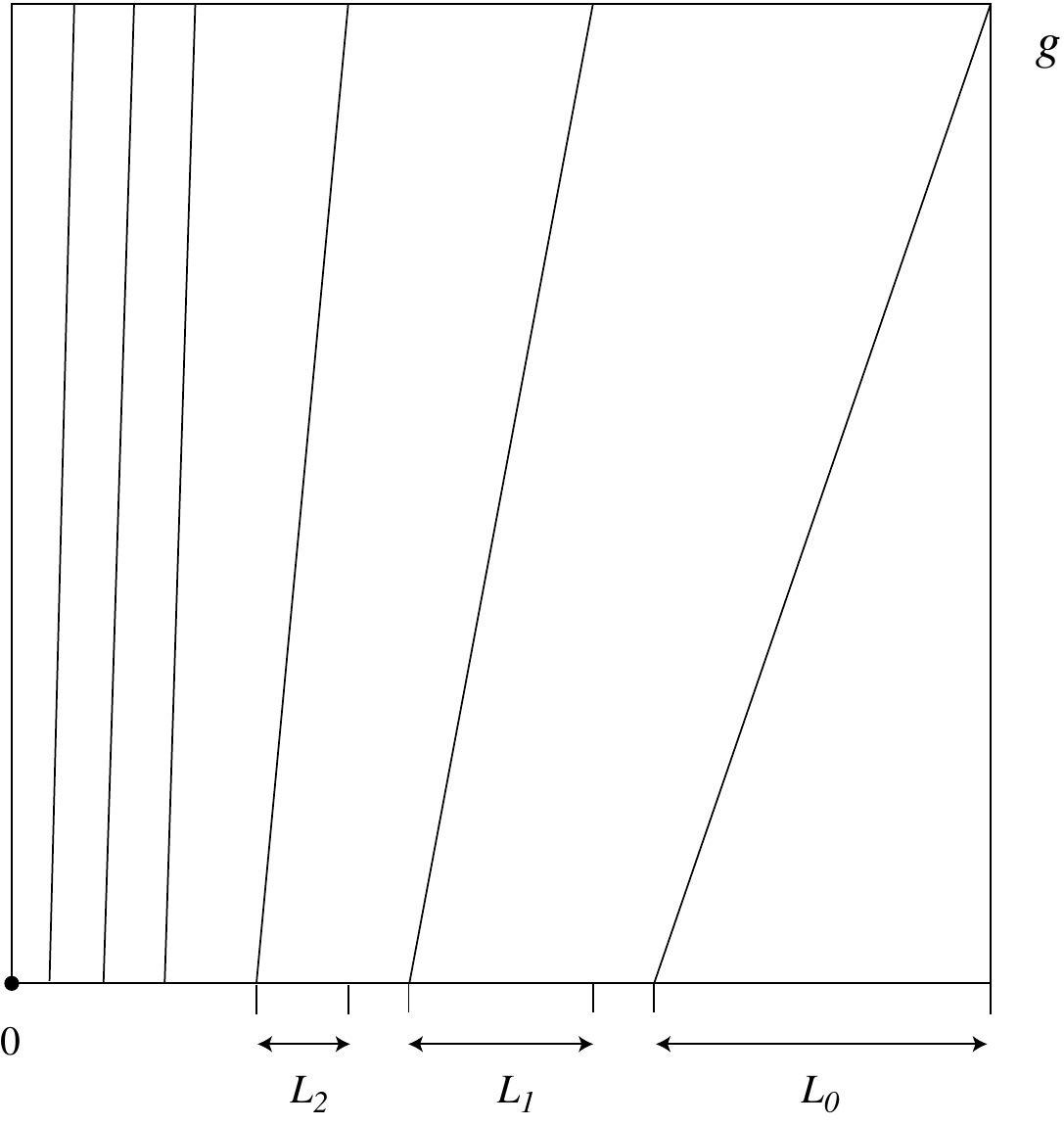}
\caption{Model for the topology and order on~$\cN$}
\label{fig:compactify-model}
\end{center}
\end{figure}

\begin{defns}[$\cN$, finite type]
We write 
\[\cN = \N^\N \cup \{W\infty\,:\,W\text{ is a (possibly trivial) word over the alphabet }\N\},\] 
Elements of~$\cN$ of the form~$W\infty$ are said to be of {\em
  finite type}.
\end{defns}

We extend the reverse lexicographic ordering of $\N^\N$ to~$\cN$. Because
finite type elements of $\cN$ are terminated by~$\infty$, every pair
$\bm,\bn$ of distinct elements of~$\cN$ first disagree at some
index~$r$ at which both $m_r$ and $n_r$ are either natural numbers
or~$\infty$: we say that $\bm < \bn$ if and only if either
$m_r=\infty$, or $m_r$ and $n_r$ are both natural numbers with
$m_r>n_r$.

We regard two elements of~$\cN$ as being close if either they agree up
to a large index, or if they have large entries up to the point where
they first disagree. To define a metric $d$ on~$\cN$ reflecting this we
write, for each pair $\bm\not=\bn$ of elements of~$\cN$,
\begin{eqnarray*}
R(\bm,\bn) &=& \min\{r\,:\,m_r\not=n_r\},\\
X(\bm,\bn) &=& r + \min\left(\sum_{s\le r}m_s,\,\, \sum_{s\le r}n_s\right)
\qquad \text{ where $r=R(\bm,\bn)$, and}\\
d(\bm,\bn) &=& 2^{-X(\bm,\bn)}.
\end{eqnarray*}
Notice that $\min(\sum_{s\le r}m_s, \sum_{s\le r}n_s) = \sum_{s<r}m_s
+ \min(m_r,n_r) < \infty$.  It is easily verified that $X(\bm,\bp) \ge
\min(X(\bm,\bn), X(\bn,\bp))$ for all $\bm,\bn,\bp\in\cN$, so that $d$
satisfies the triangle inequality. We will use the following property
of~$d$: if $\bn\in\cN$ and $n_{r}\in\N$ (i.e. $\bn$ is not a finite
type element of length~$r+1$ or less), then
\[d(\bm,\bn) < 2^{-(r+\sum_{s\le r}n_s)} \quad \implies \quad
\text{$\bm$ has initial subword $n_0n_1\ldots n_r$}.\] 

Observe that the metric is compatible with the
order on~$\cN$, in the sense that
\[
 \bm\le\bn\le\bp  \quad \implies \quad d(\bm,\bn) \le d(\bm,\bp)
 \quad\text{ and }\quad d(\bn,\bp) \le d(\bm,\bp).
\]

\medskip\medskip

\begin{lemma}
\label{lem:compact}
($\cN,d)$ is compact.
\end{lemma}
\begin{proof} 
Since every sequence in~$\cN$ has a monotonic subsequence, it suffices
to show that every monotonic sequence converges.

Consider first an increasing sequence~$(\bn^{(r)})$, and assume
without loss of generality that it is not eventually
constant. Recalling that the order on~$\cN$ is reverse lexicographic, it
is straightforward to show inductively that for every $s\in\N$ the
sequence $\left(n_s^{(r)}\right)_{r\ge 0}$ is eventually defined and
takes some constant value~$m_s$. Then $\bn^{(r)}\to\bm$ as
$r\to\infty$.

Now let~$(\bn^{(r)})$ be a decreasing sequence which is not eventually
constant. If the sequence $\left(n_s^{(r)}\right)_{r\ge 0}$ is bounded
above by some~$K_s$ for all~$s\in\N$, then the sequence converges by
the same argument as in the increasing case. So suppose this is not
the case, and let~$S\in\N$ be least such that
$\left(n_S^{(r)}\right)_{r\ge 0}$ is not bounded above. As in the
increasing case, there are natural numbers $m_s$ for $s<S$ and a
natural number~$R$ such that $n_s^{(r)} = m_s$ for all $r>R$ and
$s<S$. Then the sequence $(\bn^{(r)})$ converges to $\bm =
m_0m_1\ldots m_{S-1}\infty$, since $d(\bn^{(r)}, \bm) = 2^{-\left(S +
  n_S^{(r)} + \sum_{s<S}m_s\right)}$ for all $r>R$, and $n_S^{(r)} \to
\infty$ as $r\to\infty$.
\end{proof}

\begin{defns}[$S\colon\cN\to\cM$,\,\, $\cJ$]
Extend the function $S\colon\N^\N\to \cM$ to~$\cN$ in the
natural way, using $\Lambda_\infty(k-1) = (k-1)\, \overline{0}$:
\[S(n_0\,n_1\,\ldots\, n_r\,\infty) = \Lambda_{n_0}\circ \Lambda_{n_1} \circ
\cdots \circ \Lambda_{n_r}\left((k-1)\,\overline{0}\right).\]
Write $\cJ = S(\cN)\subset\cM$ for the image of $S$, the union of the set of
infimax sequences with the countable set of sequences just defined.
\end{defns}

\medskip\medskip

\begin{lemma}
\label{lem:S-op-homeo}
$S\colon\cN\to\cJ$ is an order-preserving homeomorphism.
\end{lemma}
\begin{proof} 
To show that~$S$ is order-preserving, let $\bm,\bn\in\cN$ with
$\bm<\bn$, and write $r=R(\bm,\bn)$ so that $m_r>n_r$ (perhaps
$m_r=\infty$). Then there is some digit~$\ell\in\{0,\ldots,k-1\}$ such
that $S(\bn)$ has initial subword
\[\Lambda_{\bn,r-1}(\Lambda_{n_r}((k-1)\ell)) =
\Lambda_{\bn,r-1}((k-1)\,0^{n_r}\Lambda_n(\ell)).\] Now $S(\bm)$ has
initial subword $\Lambda_{\bn,r-1}((k-1)\,0^{n_r+1})$ because
\mbox{$m_r\ge n_r+1$}, and $(k-1)\,0^{n_r}\Lambda_n(\ell) >
(k-1)\,0^{n_r+1}$ since $\Lambda_n(\ell)$ starts with a digit other
than~$0$. Since the $\Lambda_n$ are strictly order-preserving, it
follows that $S(\bm)<S(\bn)$ as required.

\medskip

In particular $S$ is a bijection, and in view of the compactness
of~$\cN$ it only remains to show that~$S$ is
continuous. Let~$\bn\in\cN$. To show continuity of~$S$ at~$\bn$, we
distinguish three cases.

\medskip 

\begin{enumerate}[a)]
\item Suppose that $\bn = n_0\,n_1\,\ldots\, n_r\,\infty$ is of finite
  type, so that \mbox{$S(\bn) =
    \Lambda_{\bn,r}((k-1)\,\overline{0})$}. For each
  $N\in\N$, if $d(\bn,\bm) < 2^{-\left( r + \sum_{s\le r}n_s +
    N\right)}$ then $\bm$ has initial subword $n_0\,n_1\,\ldots\,
  n_r\, M$ for some $M>N$, so $S(\bm)$ has initial subword
  \mbox{$\Lambda_{\bn,r}((k-1)\,0^N)$}. Therefore $S(\bm)\to S(\bn)$ as
  $\bm\to\bn$.
\item Suppose that $\bn\in\N^\N$ is not of rational type, so that
  there are arbitrarily large~$r$ with $n_r\not=0$ and hence the length of
  the word $\Lambda_{\bn,r}(k-1)$ goes to~$\infty$ as
  $r\to\infty$. Given~$r\in\N$, if $d(\bn,\bm) < 2^{-\left(r+\sum_{s\le
      r}n_s\right)}$ then $m_s=n_s$ for all $s\le r$: therefore
  $\Lambda_{\bm,r}(k-1) = \Lambda_{\bn,r}(k-1)$ for all $\bm$
  sufficiently close to~$\bn$, so that $S(\bm)\to S(\bn)$ as
  $\bm\to\bn$.
\item Finally, suppose that $\bn = n_0\,n_1\,\ldots\,
  n_r\,\overline{0}$ is of rational type, so that we have $S(\bn) =
  \Lambda_{\bn,r}(\overline{k-1})$. It can be shown (see for
  example the proof of Lemma~4 of~\cite{lex}) that for every $R\ge 0$,
  the word $\Lambda_{\bn,r}(\Lambda_0^R((k-1)\,0))$ has initial
  subword $\Lambda_{\bn,r}((k-1)^{1+\lfloor R/(k-1)\rfloor})$. Now if
  $\bm\not=\bn$ and $d(\bn,\bm) < 2^{-\left(r+R+\sum_{s\le
      r}n_s\right)}$, then $\bm $ has initial subword $n_0n_1\ldots
  n_r 0^{R+T}m$ for some $T\ge 0$ and $m>0$ (perhaps $m=\infty$); and
  hence $S(\bm)$ has initial subword
  $\Lambda_{\bn,r}(\Lambda_0^{R+T}((k-1)\,0))$, agreeing with $S(\bn)$
  on a subword of length at least $1 + \lfloor
  R/(k-1)\rfloor$. Therefore \mbox{$S(\bm)\to S(\bn)$} as $\bm\to\bn$.
\end{enumerate}
\end{proof}

\begin{lemma}
\label{lem:cantor}
$(\cN,d)$ is a Cantor set.
\end{lemma}
\begin{proof} 
It is a compact metric space which is homeomorphic to a subset
of~$\Sigma$ and so is totally disconnected. It therefore only
remains to show that every $\bn\in\cN$ is the limit of a sequence
$\bm^{(r)}$ of other elements of~$\cN$.

If $\bn\in\N^\N$ then we can take $m^{(r)}_r = n_r+1$ and $m^{(r)}_s =
n_s$ for all $s\not=r$. On the other hand, if $\bn = n_0\,\ldots\,
n_{R-1}\,\infty$ is of finite type, then we can take $m^{(r)}_s = n_s$
for all~$r$ if $0\le s < R$, and $m^{(r)}_s = r$ for all~$r$ if $s\ge
R$.
\end{proof}

The final lemma in this section shows that, as suggested by
Figure~\ref{fig:compactify-model}, pairs of consecutive elements
of~$\cN$ consist of one element of rational type and one of finite
type. 
\begin{lemma}
\label{lem:gaps}
Let $n_0, n_1, \ldots, n_R\in\N$ for some
$R\ge 0$. Then $n_0\,\ldots\,n_{R-1}\,(n_R+1)\,\overline{0}
\, < \, n_0\,\ldots\,n_{R-1}\,n_R\,\infty$ are consecutive elements
of~$\cN$.

On the other hand, every element of~$\cN$ which is not of rational
(respectively finite) type is the limit of a strictly decreasing
(respectively strictly increasing) sequence in~$\cN$.

\end{lemma}

\begin{proof} 
It is clear that $n_0\,\ldots\,n_{R-1}\,(n_R+1)\,\overline{0}$ and
$n_0\,\ldots\,n_{R-1}\,n_R\,\infty$ are consecutive, the former being
the largest element of~$\cN$ starting $n_0\,\ldots\,n_{R-1}\,(n_R+1)$ and the
latter the smallest element starting $n_0\,\ldots\,n_{R-1}\,n_R$.

If $\bn$ is not of finite type then, as in the proof of
Lemma~\ref{lem:cantor}, there is a sequence $\bm^{(r)}\to\bn$ with
$\bm^{(r)}<\bn$ for all~$r$. 

Suppose then that $\bn$ is not of rational type. If it is of finite
type then, as in the proof of Lemma~\ref{lem:cantor}, there is a
sequence $\bm^{(r)}\to\bn$ with $\bm^{(r)}>\bn$ for all~$r$. On the
other hand, if $\bn\in\N^\N$ then there is an increasing sequence
$i_r\to\infty$ with $n_{i_r}>0$ for each~$r$, and we can define a
sequence $(\bm^{(r)})$ converging to~$\bn$ from above by taking
$m_s^{(r)} = n_s$ for $s\not=i_r$ and $m_{i_r}^{(r)} = n_{i_r}-1$.
\end{proof}

\begin{defn}[Rational-finite pair]
For every~$R\ge 0$, and every finite sequence of natural numbers
$n_0,\ldots,n_{R-1},n_R$, the consecutive pair of elements
$n_0\,\ldots\,n_{R-1}\,(n_R+1)\,\overline{0}$ and
$n_0\,\ldots\,n_{R-1}\, n_R\,\infty$ of~$\cN$ is called a {\em
  rational-finite pair}.
\end{defn}

Thus every element of~$\cN$ of rational type apart from $\overline{0}$,
and every element of finite type apart from~$\infty$, belongs to a
rational-finite pair.

\medskip

It will be convenient to extend $\Phi\colon\Delta'\to\N^\N$ to a
(still lower semi-continuous) function $\Phi\colon\Delta\to\cN$ by
setting $\Phi(\bal) = \infty$ if $\bal\in\cF = \Delta\setminus
\Delta'$; and hence to extend $\cI\colon\Delta'\to\cM$ to a
function $\cI\colon\Delta\to\cM$ by $\cI = S\circ\Phi$. This
extension has no dynamical significance (in particular, $\cI(\bal) =
(k-1)\overline{0}$ when $\bal\in\Delta\setminus\Delta'$ is not an
$\bal$-infimax), but will make the statements of some results cleaner.


\section{Digit frequency sets of symbolic $\beta$-shifts}
\label{sec:beta-shift}
\subsection{Preliminaries}
Let~$w\in\Sigma$. The {\em symbolic $\beta$-shift} associated to~$w$
is the subshift \mbox{$\sigma\colon X(w)\to X(w)$}, where
\[X(w) = \{v\in\Sigma\,:\, \sigma^r(v) \le w \text{ for all
}r\in\N\}.\] Since the supremum of any shift-invariant subset of
$\Sigma$ is a maximal sequence, and since moreover $X(w) = X(\sup
X(w))$, it suffices to consider the case where~$w\in\cM$, which we
henceforth assume. We shall also assume that $w$ contains (and hence
starts with) the digit~$k-1$, since otherwise we could decrease the
value of~$k$. 

\begin{defns}[$\cM^*$, $\DF(w)$ for $w\in\cM^*$]
Denote by $\cM^*$ the set of elements of~$\cM$ which
start with the digit $k-1$, and write
\[
\DF(w) = 
\{
   \bal\in\Delta\,:\,X(w) \cap \cR(\bal) \not=\emptyset
\}
 \subset \Delta
\]
for each $w\in\cM^*$, the set of digit frequencies of
elements of $X(w)$.
\end{defns}

The following lemma is the fundamental result which connects digit
frequency sets to the infimaxes of Section~\ref{sec:infimax}.

\begin{lemma}
\label{lem:describe-lfs}
Let $w\in\cM^*$. Then
\[\DF(w) = \{\bal\in\Delta\,:\, \cI(\bal) \le w\}.\]
\end{lemma}
\begin{proof}
Let $\bal\in\Delta$. First note that if $\alpha_{k-1}=0$ then
$\cI(\bal) = (k-1)\,\overline{0} \le w$ for all $w\in\cM^*$, and
\mbox{$\bal\in\DF(w)$} since every sequence which doesn't contain the
digit~$k-1$ belongs to~$X(w)$. We can therefore assume that
$\alpha_{k-1}\not=0$, i.e. that $\bal\in\Delta'$.

\medskip

If $\cI(\bal) < w$ then, since $\cI(\bal) = \inf\cM(\bal)$
(Facts~\ref{facts:infifacts}a)), there is some $v\in\cM(\bal)$ with
$v<w$, and hence $v\in X(w)$. Therefore $\bal\in\DF(w)$.

\medskip

If $\cI(\bal)>w$ then every $v\in\cR(\bal)$ satisfies
$w<\cI(\bal)\le\sup_{r\ge 0}\sigma^r(v)$ by
Facts~\ref{facts:infifacts}b), so that $v\not\in X(w)$. Therefore
$X(w)\cap\cR(\bal) = \emptyset$, and hence $\bal\not\in\DF(w)$.

\medskip

Suppose then that $\cI(\bal) = w$. We shall construct an element
$v$ of $X(w)\cap\cR(\bal)$. We can assume that $\bal$ is exceptional,
and in particular is not rational, since otherwise we can take $v=w$
by Facts~\ref{facts:infifacts}f). Write $\bn=\Phi(\bal)$.

By Lemma~\ref{lem:face-converge}, $\bal$ can be approximated
arbitrarily closely by rational elements of $\DF(w)$: for the simplex
$\Phi^{-1}(\bn)$, which contains~$\bal$, can be approximated arbitrarily closely by
the simplices $\cF_{\bn,r}$, which have rational vertices. Moreover, any
element~$\bbeta$ of $\cF_{\bn,r}$ has itinerary $\Phi(\bbeta)<\bn$
by~(\ref{eq:cF-itin}), so that $\cI(\bbeta) < \cI(\bal) = w$, and
hence $\bbeta\in\DF(w)$ by the first part of the proof.

The idea of the construction is to pick a sequence $(\bal_n)$ of
rational elements of $\DF(w)$ with $\bal_n\to\bal$, and a sequence of
positive integers $(M_n)$, such that, writing $B_n$ for the repeating
block of $\cI(\bal_n)$, the sequence
\begin{equation}
\label{eq:v}
v = B_1^{M_1}\,B_2^{M_2}\,B_3^{M_3}\, \ldots
\end{equation}
lies in $\cR(\bal)$. Since it also lies in $X(w) = X(\cI(\bal))$ by
Lemma~\ref{lem:concatenate}, this will establish that $\bal\in\DF(w)$
as required.

Pick any strictly decreasing sequence $\epsilon_n\to0$, and choose the sequence
$(\bal_n)$ so that $d(\bal_n, \bal) < \epsilon_n/2$. Write $\bal_n =
\bp_n/q_n$, where $\bp_n$ is an integer vector and $q_n =
|B_n|$. Choose the positive integers~$M_n$ inductively to satisfy $M_1
= 1$ and, for $n>1$,
\begin{equation}
\label{eq:chooseM}
\frac
{\max\left(q_{n+1},
\sum_{\ell=1}^{n-1} M_\ell q_\ell
\right)}
{\sum_{\ell=1}^n M_\ell q_\ell} \,<\,
\frac
{\epsilon_n}{\sqrt{k}}.
\end{equation}
Define $v\in\Sigma$ by~(\ref{eq:v}). We specify each initial subword
of $v$ using a triple of integers $(n,i,j)$ with $n\ge 0$, $0\le i <
M_{n+1}$, and $0\le j< q_{n+1}$: the subword $V(n,i,j)$ of $v$ has
length $j + iq_{n+1} + \sum_{\ell=1}^n M_\ell q_\ell$: that is,
\[
V(n,i,j) = B_1^{M_1}\, \ldots\,B_{n}^{M_n}\,B_{n+1}^i \Word{B_{n+1}}{j}.
\]
We shall show that, for every such triple, the digit frequency vector
$\df(V(n,i,j))$ of $V(n,i,j)$ is within distance $5\epsilon_n$ of
$\bal$, which will establish that $v\in\cR(\bal)$ as required.

To simplify notation, write $\ba_n = \sum_{\ell=1}^n M_\ell\bp_\ell$
and $b_n = \sum_{\ell=1}^n M_\ell q_\ell$, so that
$\ba_n/b_n=\df(V(n,0,0))$. We will use the following simple estimate:
if $\bA/B,\bC/D\in\Delta_k$, where $\bA,\bC\in\N^k$ and $B,D\in\N$,
then $(\bA+\bC)/(B+D) - \bA/B = D(\bC/D-\bA/B)/(B+D)$, and hence
\begin{equation}
\label{eq:delta-est}
d\left(
\frac{\bA}{B}, \frac{\bA+\bC}{B+D}
\right)
\le \sqrt{k}\,\frac{D}{B+D}.
\end{equation}
Now
\begin{align*}
d(\bal, \df(V(n, i, j))) &\le
 d(\bal, \bal_n) \,+\, d(\bal_n, \df(V(n,0, 0)))  \\
&\,+\, d(\df(V(n, 0,0)), \df(V(n, i, 0))) \,+\,
 d(\df(V(n, i, 0)), \df(V(n, i, j))),
\end{align*}
and we estimate each term.

For the first, $d(\bal, \bal_n) <\epsilon_n/2 < \epsilon_n$ by choice of the sequence
$(\bal_n)$.

For the second,
\begin{equation}
\label{eq:second}
d(\bal_n, \df(V(n,0, 0))) = d(\bp_n/q_n, \ba_n/b_n) = d\left(
\frac{M_n\bp_n}{M_n q_n},
\frac{\ba_{n-1}+ M_n\bp_n}{b_{n-1}+M_nq_n}
\right)
\le \sqrt{k}\,\frac{b_{n-1}}{b_{n-1}+M_nq_n}
<\epsilon_n
\end{equation}
using (\ref{eq:delta-est}) and~(\ref{eq:chooseM}).

For the third, we have
\[
\df(V(n,i,0)) = \left(
\frac{b_n}{b_n+iq_{n+1}}
\right)\,\frac{\ba_n}{b_n} +
\left(
\frac{iq_{n+1}}{b_n + iq_{n+1}}
\right)\,\frac{\bp_{n+1}}{q_{n+1}},
\]
which lies on the line segment with endpoints $\ba_n/b_n$ and
$\bal_{n+1}$. Since $\df(V(n,0,0))=\ba_n/b_n$, the third term is bounded above by
$d(\ba_n/b_n, \bal_n) + d(\bal_n, \bal_{n+1})$, which is less than
$2\epsilon_n$ by~(\ref{eq:second}) and choice of the sequence
$(\bal_n)$.

Finally, writing $\df(\Word{B_{n+1}}{j}) = \br_{n,j}/j$, the fourth
term is
\[
d(\df(V(n,i,0)),\df(V(n,i,j))) = d\left(
\frac{\ba_n+i\bp_{n+1}}
{b_n+iq_{n+1}},
\frac{\ba_n+i\bp_{n+1}+\br_{n,j}}
{b_n+iq_{n+1}+j}
\right)
< \sqrt{k}\,\frac{j}{b_n+iq_{n+1}} \le \sqrt{k}\,\frac{j}{b_n} < \epsilon_n
\]
by~(\ref{eq:delta-est}) and~(\ref{eq:chooseM}), since $j<q_{n+1}$. This completes the proof.
\end{proof}

\begin{corollary}
\label{cor:dfs-compact}
$\DF(w)$ is compact for all $w\in\cM^*$.
\end{corollary}
\begin{proof}
Immediate from Lemma~\ref{lem:describe-lfs} and the lower
semi-continuity of $\cI = S\circ\Phi$ (Facts~\ref{facts:infifacts}d)
and e)).
\end{proof}

\begin{remark}
\label{rmk:subsequential}
If we were to define the digit frequency set subsequentially, by
\[\DF'(w) =
\{\bal\in\Delta\,:\,X(w)\cap\cR'(\bal)\not=\emptyset\},\]
then the proof of Lemma~\ref{lem:describe-lfs} goes through, using the ``primed''
versions of Facts~\ref{facts:infifacts}a) and b) (see
Facts~\ref{facts:infifacts}g)), to show that $\DF'(w) =
\{\bal\in\Delta\,:\,\cI(\bal)\le w\}$. That is, the digit frequency
set isn't sensitive to whether or not it is defined
subsequentially. We shall see (Lemma~\ref{lem:expansion-translate})
that the digit frequency set $\DF(\beta)$, where $\beta\in(k-1,k)$,
can be written as $\DF(w)$ for an appropriate $w\in\cM^*$, so that it
too can be defined subsequentially without affecting its value.
\end{remark}

A first consequence of Lemma~\ref{lem:describe-lfs} is that we can
restrict attention to $w\in\cJ$ rather than all $w\in\cM^*$, as
expressed by the next lemma.

\begin{lemma}
\label{lem:w-in-J}
Let $w\in\cM^*$, and let $s = \max\,\{t\in\cJ\,:\,t\le
w\}$. Then \mbox{$\DF(s) = \DF(w)$}. 
\end{lemma}
\begin{proof}
Note that the maximum exists by Lemmas~\ref{lem:compact}
and~\ref{lem:S-op-homeo}. Moreover, for each $\bal\in\Delta$ we have,
by Lemma~\ref{lem:describe-lfs},

\[\bal\in\DF(w) \iff \cI(\bal) \le w \iff
\cI(\bal)\in\{t\in\cJ\,:\,t\le w\}\\ \iff \cI(\bal) \le s \iff
\bal\in\DF(s).\]
\end{proof}

\begin{remark}
Lemma~\ref{lem:w-in-J} means that, as $w$ increases, the set
$\DF(w)$ undergoes bifurcations only as $w$ passes through elements
of~$\cJ$. In particular, $\DF(w)$ mode locks as $w$ passes through
each complementary gap of the Cantor set~$\cJ$. In fact, if
$w\not\in\cJ$ then $s = \max\{t\in\cJ\,:\,t\le w\}$ is of rational
type by Lemmas~\ref{lem:S-op-homeo} and~\ref{lem:gaps}, and we will
see (Theorem~\ref{thm:extreme-points}) that $\DF(s)$ is a
polytope with rational vertices in this case.
\end{remark}

In view of the fact that $S\colon\cN\to\cJ$ is an order-preserving
homeomorphism (Lemma~\ref{lem:S-op-homeo}), we can equally well study
the sets $\DF(w)$, where $w\in\cJ$, in terms of itineraries,
motivating the following definition.
\begin{defn}[$\DF(\bn)$ for $\bn\in\cN$]
For each~$\bn\in\cN$, define
\[\DF(\bn) := \DF(S(\bn)).\]
\end{defn}
Then Lemma~\ref{lem:describe-lfs} reads
\begin{equation}
\label{eq:DFn}
\DF(\bn) = \{\bal\in\Delta\,:\,\Phi(\bal) \le \bn\}.
\end{equation}
For $\bal\in\DF(\bn) \iff \bal\in\DF(S(\bn)) \iff \cI(\bal)\le
S(\bn) \iff \Phi(\bal)\le\bn$, since $\cI = S\circ \Phi$ and $S$ is an
order-preserving homeomorphism.

\begin{remark}
We have now defined digit frequency sets for three types of objects:
numbers \mbox{$\beta\in(k-1,k)$}, digit sequences $w\in\cM^*$, and
itineraries $\bn\in\cN$. In each case the collection of digit frequency
sets is the same, with the exception that $\Delta$ and $\cF$ can be
realised as digit frequency sets of elements of $\cM^*$ (namely
$\overline{k-1}$ and $(k-1)\,\overline{0}$) and of elements of $\cN$
(namely $\overline{0}$ and $\infty$), but not of elements of $(k-1,k)$
-- although they are, of course, the digit frequency sets of $k$ and
of $k-1$ respectively. To show that the collections are otherwise the
same, it is enough to observe (see Definitions~\ref{defn:beta_n}
and~\ref{defn:n_w} and Lemma~\ref{lem:expansion-translate}) that there
are digit frequency set preserving maps
\begin{eqnarray*}
\beta\mapsto w_\beta \,&\colon& \,
(k-1,k)\to\cM^*\setminus\{\overline{k-1},\,(k-1)\,\overline{0}\}, \\
w\mapsto \max\,\{\bn\in\cN\,:\,S(\bn)\le w\} \,&\colon& \,
\cM^*\setminus\{\overline{k-1},\,(k-1)\,\overline{0}\} \to
\cN\setminus\{\overline{0},\,\infty\}, \text{ and}\\
\bn\mapsto \beta(\bn) \,&\colon& \,
\cN\setminus\{\overline{0},\,\infty\} \to (k-1,k).
\end{eqnarray*}
\end{remark}

\medskip

Using (\ref{eq:DFn}) and Lemma~\ref{lem:gaps}, together with the fact
that elements of~$\cN$ of finite type are not in the image of~$\Phi$,
we have
\begin{equation}
\label{eq:infinity-LF}
\DF(n_0\,\ldots\, n_{R-1}\, (n_{R}+1)\, \overline{0}) =
\DF(n_0\,\ldots\, n_{R-1}\,n_R\,\infty)
\end{equation}
for all $n_0,\,\ldots,\, n_R$: the digit frequency set of an element
of finite type other than~$\infty$ is the same as that of its rational
pair. In particular, when describing the different possible structures
of $\DF(\bn)$ as $\bn$ varies in~$\cN$, we can restrict to the case $\bn\in\N^\N$.

\subsection{Convexity of the digit frequency set}

\begin{defn}[$T_n$]
For each $n\ge 0$, let $T_n$ be the simplex
\[
T_n = \{\bal\in\Delta\,:\,(n+1)\,\alpha_{k-1} \le \alpha_0\}\subset\Delta,
\]
the union of the simplices $\Delta_m$ for $m>n$ together with the face
$\alpha_{k-1}=0$; or, equivalently, the set of $\bal$ with
$\Phi(\bal)_0 > n$ (perhaps $\Phi(\bal)_0 = \infty$). 
\end{defn}
The following lemma describes the structure of digit frequency sets
recursively.
\begin{lemma}
\label{lem:lfs-recursive}
Let~$\bn\in\N^\N$. Then
\[
\DF(\bn) = T_{n_0} \cup K_{n_0}^{-1}(\DF(\sigma(\bn))),
\]
where $\sigma\colon\N^\N\to\N^\N$ is the shift map.
\end{lemma}
\begin{proof}
$T_{n_0}\subset\DF(\bn)$, since $\bal\in T_{n_0}\implies\Phi(\bal)_0 >
  n_0 \implies \Phi(\bal) < \bn$. 

If $\bal\in\DF(\bn)\setminus T_{n_0}$ then $\Phi(\bal)_0 = n_0$, and
for all $\bal$ with $\Phi(\bal)_0 = n_0$ we have
\[
\bal\in\DF(\bn) \iff \Phi(\bal)\le\bn \iff \sigma(\Phi(\bal)) \le
\sigma(\bn) \iff \Phi(K_{n_0}(\bal)) \le \sigma(\bn) \iff
K_{n_0}(\bal)\in\DF(\sigma(\bn)).\]
\end{proof}

We now define useful sequences of subsets and supersets of~$\DF(\bn)$.

\begin{defn}[$L_{\bn,r}$, $U_{\bn,r}$]
For each $\bn\in\N^\N$ and $r\in\N$, write
\[L_{\bn,r} = \bigcup_{s=0}^r K_{n_0}^{-1}\circ \cdots \circ
K_{n_{s-1}}^{-1}(T_{n_s}) \qquad\text{and}\qquad U_{\bn,r} = L_{\bn,r}
\cup A_{\bn,r}.\]
\end{defn}

\begin{lemma}
\label{lem:bounds}
Let $\bn\in\N^\N$. Then $L_{\bn,r} \subset \DF(\bn) \subset U_{\bn,r}$
and $L_{\bn,r} \cap A_{\bn,r} = \cF_{\bn,r}$ for
all~$r\in\N$. Moreover \mbox{$\DF(\bn) = \bigcap_{r\ge 0}
  U_{\bn,r}$}.
\end{lemma}

\begin{proof} 
$K_{n_0}^{-1}\circ \cdots\circ K_{n_{s-1}}^{-1}(T_{n_s}) =
  K_{n_0}^{-1}\circ \cdots\circ K_{n_{s-1}}^{-1}(\cF) \cup
  K_{n_0}^{-1}\circ \cdots\circ K_{n_{s-1}}^{-1}(T_{n_s}\setminus\cF)$
  is the union of $\cF_{\bn,s-1}$ (which is contained in $\DF(\bn)$
  using~(\ref{eq:cF-itin}) and that $\cF\subset\DF(\bn)$) and of
  $\{\bal\in\Delta\,:\, \Word{\Phi(\bal)}{s} = \Word{\bn}{s} \text{
    and } \Phi(\bal)_s > n_s\}$ (which is contained in $\DF(\bn)$
  by~(\ref{eq:DFn})). This establishes the lower bound. 

Moreover,
  since $L_{\bn,r}$ contains $\{\bal\in\Delta\,:\,
  \Word{\Phi(\bal)}{s} = \Word{\bn}{s} \text{ and } \Phi(\bal)_s >
  n_s\}$ for each $s$ with $0\le s\le r$, we have that
\[
  L_{\bn,r} \,\supset\, \{\bal\in\Delta\,:\,\Word{\Phi(\bal)}{r+1} <
  \Word{\bn}{r+1}\}.
\]
Since $A_{\bn,r} = \{\bal\in\Delta\,:\,\Word{\Phi(\bal)}{r+1} =
  \Word{\bn}{r+1}\} \cup \cF_{\bn,r}$ it follows that $U_{\bn,r}$
  contains the set \mbox{$\{\bal\in\Delta\,:\, \Word{\Phi(\bal)}{r+1}
    \le \Word{\bn}{r+1}\}$}, which contains~$\DF(\bn)$, establishing
  the upper bound. Moreover if $\Word{\Phi(\bal)}{r+1} =
  \Word{\bn}{r+1}$ then $\bal\not\in L_{\bn,r}$, so that
  $L_{\bn,r}\cap A_{\bn,r} \subset \cF_{\bn,r}$; and points $\bal$ of
  $\cF_{\bn,r}$ either lie in $\cF\subset T_{n_0}\subset L_{\bn,r}$ or
  have $\Word{\Phi(\bal)}{r+1} < \Word{\bn}{r+1}$
  by~(\ref{eq:cF-itin}), so that $\cF_{\bn,r} \subset L_{\bn,r}\cap
  A_{\bn,r}$ as required.

To show that $\DF(\bn) = \bigcap_{r\ge 0} U_{\bn,r}$, observe that any
$\bal\in\bigcap_{r\ge 0} U_{\bn,r}$ either lies in $L_{\bn,r}$ for
some~$r$, and hence in $\DF(\bn)$; or lies in $\bigcap_{r\ge 0}
A_{\bn,r} = \Phi^{-1}(\bn)$, so that $\Phi(\bal) = \bn$ and
$\bal\in\DF(\bn)$. 
\end{proof}

\begin{theorem}
\label{thm:convex}
For all~$\bn\in\cN$, the digit frequency set $\DF(\bn)$ is convex.
\end{theorem}

\begin{proof}
By Lemma~\ref{lem:bounds} we need only
show that $U_{\bn,r}$ is convex for all $\bn\in\N^\N$ and
all~$r\in\N$, which we do by induction on~$r$. In the base case~$r=0$,
we have $U_{\bn,0} = T_{n_0} \cup K_{n_0}^{-1}(\Delta) = T_{n_0-1}$,
which is convex (we write $T_{-1} = \Delta$ for convenience).

For $r>0$, we have $L_{\bn,r} = T_{n_0} \cup
K_{n_0}^{-1}(L_{\sigma(\bn), r-1})$ and $A_{\bn,r} =
  K_{n_0}^{-1}(A_{\sigma(\bn), r-1})$, so that
\[U_{\bn,r} = T_{n_0} \cup K_{n_0}^{-1}(U_{\sigma(\bn), r-1}).\]
Now $U_{\sigma(\bn), r-1}$ is convex by the inductive hypothesis, and
hence so is $K_{n_0}^{-1}(U_{\sigma(\bn), r-1})$. Moreover,
\[\closure{\Delta_{n_0}}\supset K_{n_0}^{-1}(U_{\sigma(\bn),r-1}) \supset
K_{n_0}^{-1}(\DF(\sigma(\bn))) \supset K_{n_0}^{-1}(\cF) =
\cF_{\bn,0},\] and $\cF_{\bn,0}$ is also a face of~$T_{n_0}$. That is, the
simplex~$T_{n_0-1}$ is the union of the two simplices $T_{n_0}$ and
$\closure{\Delta_{n_0}}$ whose intersection is the common
face~$\cF_{\bn,0}$; and $U_{\bn,r}$ is the union of $T_{n_0}$ and the
convex subset $K_{n_0}^{-1}(U_{\sigma(\bn),r-1})$ of
$\closure{\Delta_{n_0}}$ which contains $\cF_{\bn,0}$. It follows that
$U_{\bn,r}$ is convex as required, since any line segment joining a
point of $T_{n_0}$ to a point of $K_{n_0}^{-1}(U_{\sigma(\bn),r-1})$
passes through $\cF_{\bn,0}$, and is therefore the join of a segment
in $T_{n_0}$ and a segment in $K_{n_0}^{-1}(U_{\sigma(\bn),r-1})$.
\end{proof}

\begin{remark}
\label{rmk:no-local-max}
As a consequence, we obtain the following property of the itinerary
map~$\Phi$: if $\ell\subset\Delta$ is a line segment with endpoints
$\bal$ and $\bbeta$, and $\bgamma\in\ell$, then $\Phi(\bgamma) \le
\max(\Phi(\bal), \Phi(\bbeta))$. For if not then, since
$\max(\Phi(\bal),\Phi(\bbeta))$ and $\Phi(\bgamma)$ are not the
elements of a rational-finite pair (the only finite type element in
the image of~$\Phi$ is $\infty$), we can choose $\bn\in\N^\N$ with
\mbox{$\max(\Phi(\bal),\Phi(\bbeta)) < \bn < \Phi(\bgamma)$}. Then
$\bal,\bbeta\in \DF(\bn)$ but $\bgamma\not\in\DF(\bn)$, contradicting
the convexity of $\DF(\bn)$.
\end{remark}

\subsection{Extreme points of the digit frequency set}

By Theorem~\ref{thm:convex}, $\DF(\bn)$ is determined by the set
 of its extreme points, which we now describe. Observe
first that since $\cF\subset\DF(\bn)$, the vertices
$\be_0,\be_1,\ldots,\be_{k-2}$ of~$\cF$ are extreme points of
$\DF(\bn)$ for all $\bn$. 

\begin{defns}[$\EP(\bn)$, $\E(\bn)$, non-trivial extreme points]
Let $\EP(\bn)$ denote the set of extreme points of $\DF(\bn)$, and
\[\E(\bn) = \EP(\bn) \setminus \{\be_0,\be_1,\ldots,\be_{k-2}\}\]
the set of {\em non-trivial} extreme points.
\end{defns}

The next lemma translates the inductive expression for $\DF(\bn)$
given by Lemma~\ref{lem:lfs-recursive} into an analogous one for
$\E(\bn)$. Notice that in the case~$k=2$ the condition ``$n_r=0$ for
$1\le r\le k-2$'' is always true.

\begin{lemma}
\label{lem:extreme-points-recursive}
Let~$\bn\in\N^\N$. Then
\[
\E(\bn) = 
\begin{cases}
K_{n_0}^{-1}(\E(\sigma(\bn))) & \text{ if }n_r=0 \text{ for }1\le
r \le k-2,\\
K_{n_0}^{-1}(\E(\sigma(\bn)) \,\cup\, \{\be_{k-2}\}) & \text{ otherwise.}
\end{cases}
\]
\end{lemma}

\begin{proof} 
By Lemma~\ref{lem:lfs-recursive}, $\DF(\bn)$ is the union of the
simplex $T_{n_0}$ and the image of $\DF(\sigma(\bn))$ under the
projective homeomorphism~$K_{n_0}^{-1}$. Since the intersection of
these sets is exactly~$\cF_{\bn,0}$, which has vertices
$\be_1,\ldots,\be_{k-2}$ and $K_{n_0}^{-1}(\be_{k-2})$, the non-trivial
extreme points of $\DF(\bn)$ are the images under $K_{n_0}^{-1}$ of
the non-trivial extreme points of $\DF(\sigma(\bn))$, together perhaps
with the point $K_{n_0}^{-1}(\be_{k-2})$. 

Now if $K_{n_0}^{-1}(\be_{k-2}) = ((n_0+1)/(n_0+2), 0, \ldots, 0,
1/(n_0+2))$ is on a line segment joining two other points of
$\DF(\bn)$, then these points must have coordinates $(1-x,0,\ldots,0,
x)$ and $(1-y,0,\ldots,0, y)$ with $x < 1/(n_0+2) < y$. Since
\mbox{$\be_0 = (1,0,\ldots,0,0)\in \DF(\bn)$}, it follows that
$K_{n_0}^{-1}(\be_{k-2})$ is {\em not} an extreme point of $\DF(\bn)$
if and only if $K_{n_0}^{-1}(\DF(\sigma(\bn)))$ contains a point of
the form $(1-y,0,\ldots,0,y)$ for some $y$ with $1/(n_0+2) < y \le
1/(n_0+1)$ (the latter inequality coming from the fact that
\mbox{$(1-y,0,\ldots,0,y)$} can only be in $K_{n_0}^{-1}(\Delta)$ if
$(1-y)/y \ge n_0$).

Since $K_{n_0}(1-y,0,\ldots,0,y) = \left(0,\ldots,0, \frac{1}{y} - (n_0+1),
(n_0+2)- \frac{1}{y}\right)$ it follows that
\begin{eqnarray*}
K_{n_0}^{-1}(\be_{k-2})\not\in\E(\bn) &\iff& 
\exists z\in (0,1],\,\,(0,0,\ldots,0,1-z,z) \in\DF(\sigma(\bn))\\
&\iff&
\exists z\in (0,1],\,\, \Phi(0,0,\ldots,0,1-z,z) \le \sigma(\bn).
\end{eqnarray*}
Now $\Phi(0,0,\ldots,0,1-z,z)_r = 0$ for $0\le r\le k-3$, so that
\[
\exists z\in (0,1],\,\, \Phi(0,0,\ldots,0,1-z,z) \le \sigma(\bn)
  \implies n_r= 0 \text{ for }1\le r\le k-2.
\]
 Conversely, if $n_r=0$ for $1\le r\le k-2$, then let $z=1/(n_{k-1}+2)
 \in(0,1]$. We have
\[K^{k-2}(0,0,\ldots,0,1-z,z) =
(1-z,0,0,\ldots,0,z) \in\Delta_{\lfloor(1-z)/z\rfloor} =
\Delta_{n_{k-1}+1},\]
so that $\Phi(0,0,\ldots,0,1-z,z)\le \sigma(\bn)$.

Therefore $K_{n_0}^{-1}(\be_{k-2})\not\in\E(\bn)$ if and only if $n_r=0$
for $1\le r\le k-2$, as required.
\end{proof}

The following theorem describes the set of non-trivial extreme points
of $\DF(\bn)$ non-inductively: this set consists of the points
$\Upsilon_{\bn,s}(\be_{k-2})$ for those indices~$s$ which are not
followed by $k-2$ zeroes in the itinerary~$\bn$, together with a
subset of the vertices of $\Phi^{-1}(\bn)$. In the regular case, this
subset is exactly the one-point set $\Phi^{-1}(\bn)$; while in the
exceptional case, since $\Phi^{-1}(\bn)$ is a $j$-simplex for some
$j\le k-2$, it contains at most $k-1$ points.

\begin{defn}[$\FE(\bn)$]
For each $\bn\in\N^\N$, write
\[\FE(\bn) = \{K_{n_0}^{-1} \circ \cdots \circ K_{n_s}^{-1}(\be_{k-2})
\,:\, s\ge 0,\,\, n_{s+t}\not=0 \text{ for some }1\le t\le k-2\}.\]
\end{defn}

\begin{remark}
\label{rmk:unique-description}
Each point of $\FE(\bn)$ is listed once only: that
is, whenever $s<t$ we have $K_{n_0}^{-1} \circ \cdots \circ
K_{n_s}^{-1}(\be_{k-2}) \not= K_{n_0}^{-1} \circ \cdots \circ
K_{n_t}^{-1}(\be_{k-2})$. For otherwise we would have
  $K_{n_{s+1}}^{-1}\circ \cdots \circ K_{n_t}^{-1}(\be_{k-2}) =
  \be_{k-2}$; but it follows immediately from~(\ref{eq:KnI}) that
  $(K_{n_{s+1}}^{-1}\circ \cdots \circ
  K_{n_t}^{-1}(\be_{k-2}))_{k-1}>0$, which is a contradiction.
\end{remark}

\begin{theorem}
\label{thm:extreme-points}
Let~$\bn\in\N^\N$.  Then
\[\FE(\bn) \,\,\subset\,\, \E(\bn)\,\, \subset\,\, \left(\FE(\bn)  \cup
\Ve\Phi^{-1}(\bn)\right).\]
Moreover, 
\begin{enumerate}[a)]
\item $\E(\bn)$ is finite (i.e. $\DF(\bn)$ is a polytope) if and only
  if $k=2$ or $\bn$ is of rational type.
\item If $\bn$ is regular, so that $\Phi^{-1}(\bn)=\{\bal\}$ for some
  $\bal$, then $\E(\bn) = \FE(\bn) \,\cup\,\{\bal\}$, and $\bal$ is an
  accumulation point of $\E(\bn)$ if $\E(\bn)$ is infinite.
\item The points of $\FE(\bn)$ are rational, while those of
  $\Ve\Phi^{-1}(\bn)$ are non-rational unless $\bn$ is of rational type.
\end{enumerate}
\end{theorem}

\begin{remarks}\mbox{}
\label{rmk:extreme-pt-thm}
\begin{enumerate}[a)]
\item When~$k=2$, $\FE(\bn)$ is always empty and $\bn$ is always
  regular, so there is one non-trivial extreme point, the point of
  $\Phi^{-1}(\bn)$, together with the trivial extreme point~$(1,0)$. 
\item Let $n_r = 2^{2^{3r}}$. Then for each $k\ge
  3$, the set $\E(\bn)$ has $k-1$ accumulation points. For it is shown
  in~\cite{lex} that $\Phi^{-1}(\bn)$ is a simplex of
  dimension~$k-2$. The easy part of the proof of Borovikov's theorem
  (see~\cite{Winkler} Section~3.3 for a proof in English) shows that
  each of the $k-1$ vertices of $\Phi^{-1}(\bn)$ is an accumulation
  point of vertices of the simplices $A_{\bn,r}$ -- and each of these
  vertices belongs to $\E(\bn)$, since $n_r>0$ for all~$r$.
\item The authors do not know whether or not $\E(\bn)\not=\FE(\bn)
  \cup \Ve\Phi^{-1}(\bn)$ is possible in the case when $\bn$ is
  exceptional; nor whether or not it is possible for the points of
  $\E(\bn)$ to accumulate at a point of $\Phi^{-1}(\bn)$ other than a
  vertex (they can't accumulate at a point not in $\Phi^{-1}(\bn)$,
  since $K_{n_0}^{-1} \circ \cdots \circ K_{n_s}^{-1}(\be_{k-2}) \in
  A_{\bn,s} \to \Phi^{-1}(\bn)$ as $s\to\infty$).
\end{enumerate}
\end{remarks}

\begin{proof} 
A straightforward induction on~$R$ using
Lemma~\ref{lem:extreme-points-recursive} gives that, for each~$R\ge1$,

\begin{eqnarray*}
\E(\bn) &=& \{K_{n_0}^{-1} \circ \cdots \circ K_{n_s}^{-1}(\be_{k-2})
\,:\, 0\le s < R,\, n_{s+t}\not=0 \text{ for some }1\le t\le k-2\} \\
&& \qquad \cup \quad K_{n_0}^{-1}\circ \cdots \circ
K_{n_{R-1}}^{-1}(\E(\sigma^R(\bn)). 
\end{eqnarray*}
Now $K_{n_0}^{-1}\circ \cdots \circ
K_{n_{R-1}}^{-1}(\E(\sigma^R(\bn))\subset A_{\bn,R-1}$, and
$\bigcap_{r\ge 0} A_{\bn,r} = \Phi^{-1}(\bn)$. Therefore 
\[
\E(\bn) \supset\FE(\bn) =  \{K_{n_0}^{-1} \circ \cdots \circ K_{n_s}^{-1}(\be_{k-2})
\,:\, s\ge 0,\, n_{s+t}\not=0 \text{ for some }1\le t\le k-2\},
\]
and any elements of $\E(\bn)$ not in $\FE(\bn)$ are contained in
$\Phi^{-1}(\bn)$. Since $\Phi^{-1}(\bn)$ is a simplex which is
contained in $\DF(\bn)$, any remaining extreme points of $\DF(\bn)$
must be vertices of~$\Phi^{-1}(\bn)$. 

\medskip

We now prove the remaining statements of the theorem.
\begin{enumerate}[a)]
\item By Remark~\ref{rmk:unique-description}, $\E(\bn)$ is finite if
  and only if there are only finitely many~$s\ge 0$ with the property
  that $n_{s+t}\not=0$ for some $1\le t\le k-2$. This is always the
  case when~$k=2$, but for $k\ge 3$ is true if and only if $n_s = 0$
  for all sufficiently large~$s$; i.e., if and only if $\bn$ is of
  rational type.
\item By Remark~\ref{rmk:no-local-max}, if $\bn$ is regular then it is
  not possible for the unique point of $\Phi^{-1}(\bn)$ (which has
  itinerary~$\bn$) to belong to a line segment whose endpoints are in
  $\DF(\bn) \setminus \Phi^{-1}(\bn)$ (and so have itinerary less
  than~$\bn$). This point is therefore an extreme point of $\DF(\bn)$
  as required. Moreover, $\Phi^{-1}(\bn)$ is necessarily an
  accumulation point of~$\E(\bn)$ when $\E(\bn)$ is infinite, since
  $K_{n_0}^{-1}\circ\cdots\circ K_{n_s}^{-1}(\be_{k-2})\in A_{\bn,s}
  \to \Phi^{-1}(\bn)$ as $s\to\infty$.
\item The points of $\FE(\bn)$ are clearly rational. No point with
  itinerary~$\bn$ can be rational unless $\bn$ is of rational type.

\end{enumerate}

\end{proof}

\begin{corollary}
\label{cor:almost-injective}
Let~$\bm,\bn\in\cN$ with $\bm\not=\bn$. Then $\DF(\bm) = \DF(\bn)$ if and
only if $\bm$ and~$\bn$ form a rational-finite pair.
\end{corollary}
\begin{proof}
If $\bm$ and $\bn$ form a rational-finite pair then
$\DF(\bm)=\DF(\bn)$ by~(\ref{eq:infinity-LF}).  Otherwise there are
both rational and non-rational elements of~$\cN$ between $\bm$ and
$\bn$. Since $\DF$ is increasing by definition, it follows from
Theorem~\ref{thm:extreme-points}a) (or, in the case~$k=2$, from
Theorem~\ref{thm:extreme-points}c)) that $\DF(\bm)\not=\DF(\bn)$ as
required.
\end{proof}

\begin{examples} 
\label{ex:ep}
\begin{enumerate}[a)]
\item Let $\bn = 2\,1\,0\,1\,\overline{0}$. In the case $k=3$, the
  digit frequency set $\DF(\bn)$ has two trivial extreme points,
  $(1,0,0)$ and $(0,1,0)$, together with three non-trivial extreme
  points
\begin{eqnarray*}
K_2^{-1}(0,1,0) &=& (3/4,\,0,\,1/4),\\
K_2^{-1}K_1^{-1}K_0^{-1}(0,1,0) &=& (5/8,\,1/8,\,2/8), \quad\text{ and}\\
\Phi^{-1}(\bn) = K_2^{-1}K_1^{-1}K_0^{-1}K_1^{-1}(0,0,1) &=& (4/9,\,3/9,\,2/9).
\end{eqnarray*}
Hence $\DF(\bn)$ is a pentagon (Figure~\ref{fig:short}). The point
$K_2^{-1}K_1^{-1}(0,1,0) = (2/5,\,2/5,\,1/5)$ is not an extreme point
since~$n_2=0$ --- and indeed it can be checked that it lies on the line
segment joining the extreme points $(0,1,0)$ and $(4/9,\,3/9,\,2/9)$.

\medskip

The same itinerary~$\bn$ in the case~$k=4$ gives
$|\EP(\bn)| = 7$: in addition to the three trivial extreme points
$(1,0,0,0)$, $(0,1,0,0)$, and $(0,0,1,0)$, we have
\begin{eqnarray*}
K_2^{-1}(0,0,1,0) &=& (3/4,\,0,\,0,\,1/4),\\
K_2^{-1}K_1^{-1}(0,0,1,0) &=& (2/5,\,2/5,\,0,\,1/5),\\
K_2^{-1}K_1^{-1}K_0^{-1}(0,0,1,0) &=& (2/5,\,1/5,\,1/5,\,1/5), \quad\text{
  and}\\
\Phi^{-1}(\bn) = K_2^{-1}K_1^{-1}K_0^{-1}K_1^{-1}(0,0,0,1) &=&
(5/8,\,1/8,\, 0,\, 2/8).\\
\end{eqnarray*}
Thus $\DF(\bn)$ is a polyhedron with 7 vertices. In this case we do include
the point $K_2^{-1}K_1^{-1}(0,0,1,0)$, since it is not true that both
$n_2=0$ and $n_3=0$.

\item When $k=3$ and $\bn$ is exceptional, the boundary of $\DF(\bn)$
  contains the exceptional interval, both ends of which are
  accumulations of extreme points. Figure~\ref{fig:cubes}, drawn in
  the $(\alpha_0,\alpha_2)$-plane, illustrates
  the situation: it depicts $\DF(\bn)$, and its extreme points, in the
  case where $n_r = r^3$ for $0\le r\le 25$, and $n_r=0$ for $r>
  25$. Notice that, if~$\bN\in\N^\N$ is given by $N_r=r^3$ for
  all~$r$, then the first $25$ extreme points in $\FE(\bN)$ provided
  by Theorem~\ref{thm:extreme-points} are exactly the points of $\FE(\bn)$.

It is not known whether or not~$\bN$ is exceptional, although
experimental evidence suggests that it is. However, itineraries which
are known to be exceptional, such as $n_r=2^{2^{3r}}$, grow too
quickly for it to be feasible to produce plots similar to
Figure~\ref{fig:cubes}.

\begin{figure}[htbp]
\begin{center}
\includegraphics[width=0.5\textwidth]{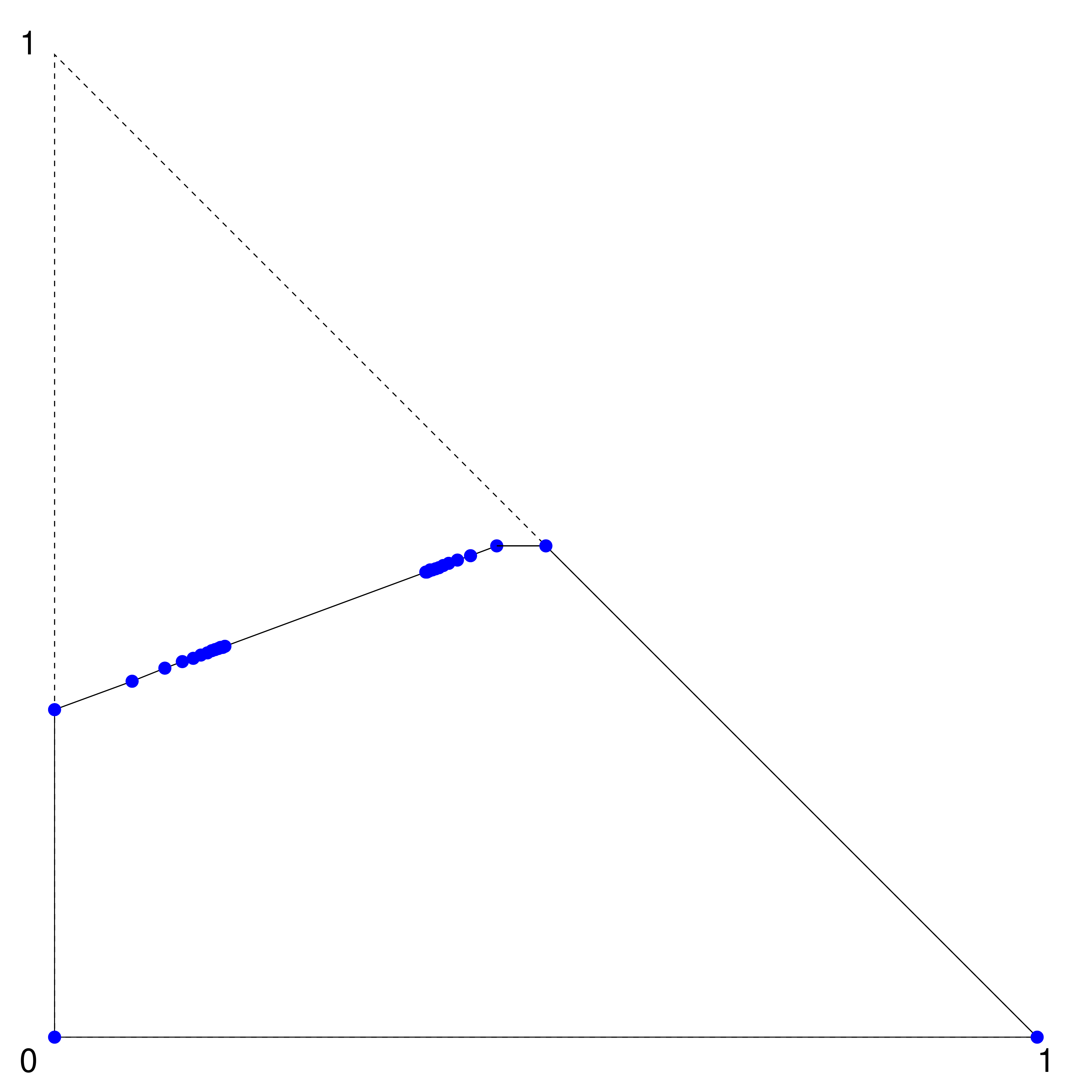}
\caption{An approximation to $\DF(\bn)$ for exceptional~$\bn$ in the case~$k=3$}
\label{fig:cubes}
\end{center}
\end{figure}

Constrast Figure~\ref{fig:cubes} with Figure~\ref{fig:squares}, which
shows an approximation to $\DF(\bn)$ when $n_r=r^2$, which is known to
be regular. Here the sequence of rational extreme points limits on the
unique (non-rational) $\bal$ with $\Phi(\bal) = \bn$.

\begin{figure}[htbp]
\begin{center}
\includegraphics[width=0.5\textwidth]{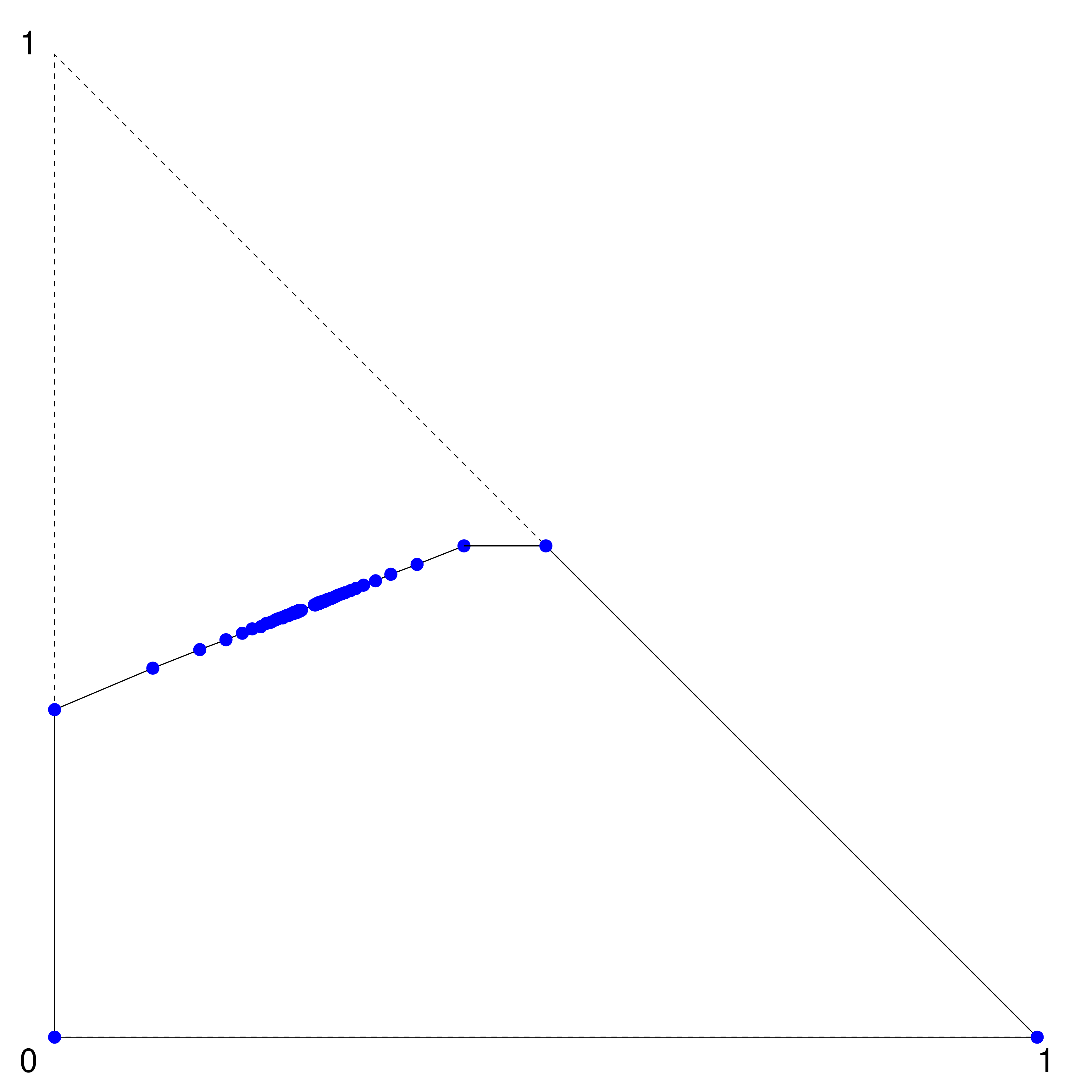}
\caption{An approximation to $\DF(\bn)$ when~$k=3$ in the regular case
  $n_r=r^2$}
\label{fig:squares}
\end{center}
\end{figure}

\item Let~$k=3$, and pick $\bal\in\Delta'$ with $\Phi(\bal) = \bn$. In
  this example we consider the way in which the rational extreme
  points $K_{n_0}^{-1}\circ\cdots\circ K_{n_s}^{-1}(0,1,0)$ (with
  $n_{s+1}\not=0$) are ordered around the boundary of $\DF(\bn)$. Each
  such point either lies on the {\em even} segment of
  $\partial\DF(\bn)$, which joins $(1,0,0)$ to~$\bal$ and does not
  contain $(0,1,0)$; or on the {\em odd} segment, which joins
  $(0,1,0)$ to $\bal$ and does not contain $(1,0,0)$. We claim that
  the extreme points $K_{n_0}^{-1}\circ\cdots\circ K_{n_s}^{-1}(0,1,0)$
  with~$s$ even (respectively odd) are contained in the even
  (respectively odd) segment, and move monotonically along the segment
  towards~$\bal$ as $s$ increases. In other words, in pictures such as
  those of Figures~\ref{fig:short},~\ref{fig:cubes},
  and~\ref{fig:squares}, the points with $s$ even move counter-clockwise
  around the boundary starting at the bottom right vertex~$(1,0,0)$,
  while those with $s$ odd move clockwise starting at the
  bottom left vertex~$(0,1,0)$.

The claim can be proved in the case where $\bal$ is rational, with
$\Phi(\bal) = n_0\,n_1\,\ldots\,n_r\,\overline{0}$, by induction
on~$r$, using Lemma~\ref{lem:lfs-recursive} and the fact that each
$K_n^{-1}$ is orientation-reversing. The result then follows in the
non-rational case since the extreme points $K_{n_0}^{-1}\circ\cdots\circ
K_{n_s}^{-1}(0,1,0)$ on the boundary of $\DF(\bn)$ are the same points
as those on the boundary of $\DF(\bm)$, where $\bm =
n_0\,n_1\,\ldots\,n_{s+1}\,\overline{0}$.

In particular, if $\bal$ is regular then it is a limit of rational
extreme points of $\DF(\bn)$ from both sides (as in
Figure~\ref{fig:squares}), unless there is some $s\in\N$ such that
$n_{s+2i}=0$ for all $i\in\N$ (so that none of the points
$K_{n_0}^{-1}\circ\cdots\circ K_{n_{s+2i-1}}^{-1}(0,1,0)$ is an
extreme point). By Facts~\ref{facts:infifacts}i), such an~$s$ exists
if and only if $K^r(\bal)$ lies on the boundary of~$\Delta$ for some
$r\in\N$. In this case, by the convexity of $\DF(\sigma^r(\bn))$ and
the fact that $\cF\subset\DF(\sigma^r(\bn))$, the boundary of
$\DF(\sigma^r(\bn))$ contains a segment~$I$ with endpoints $K^r(\bal)$
and either $(0,1,0)$ or $(1,0,0)$. This segment is contained in a face
of~$\Delta$, and so has rational direction. By
Lemma~\ref{lem:lfs-recursive}, $K_{n_0}^{-1}\circ\cdots\circ
K_{n_{r-1}}^{-1}(I)$ is a segment of the boundary of $\DF(\bn)$ which
has one endpoint at $\bal$ and the other endpoint at a rational
boundary point of $\DF(\bn)$.

In summary, if $\bal$ is regular then there is a segment of the
boundary of $\DF(\bn)$ which contains $\bal$ if and only if every
other entry of $\Phi(\bal)$ is eventually zero; and in this case, the
segment is necessarily contained in a line of rational direction which passes
through rational vectors, so that the components of $\bal$ are
rationally dependent. See Figure~\ref{fig:oneside}, which
depicts $\DF(\bn)$ for $\bn = 1\,1\,\overline{1\,0}$. The irrational
extreme point~$\bal$ is a limit of rational extreme points from the odd side
only, since $n_s = 0$ for all odd~$s\ge3$. In fact $\bal =
((1+\sqrt5)/6, (3-\sqrt5)/6, 1/3)$, and the rational extreme point
adjacent to $\bal$ is $(2/3,0,1/3)$.
 
\end{enumerate}
\begin{figure}[htbp]
\begin{center}
\includegraphics[width=0.5\textwidth]{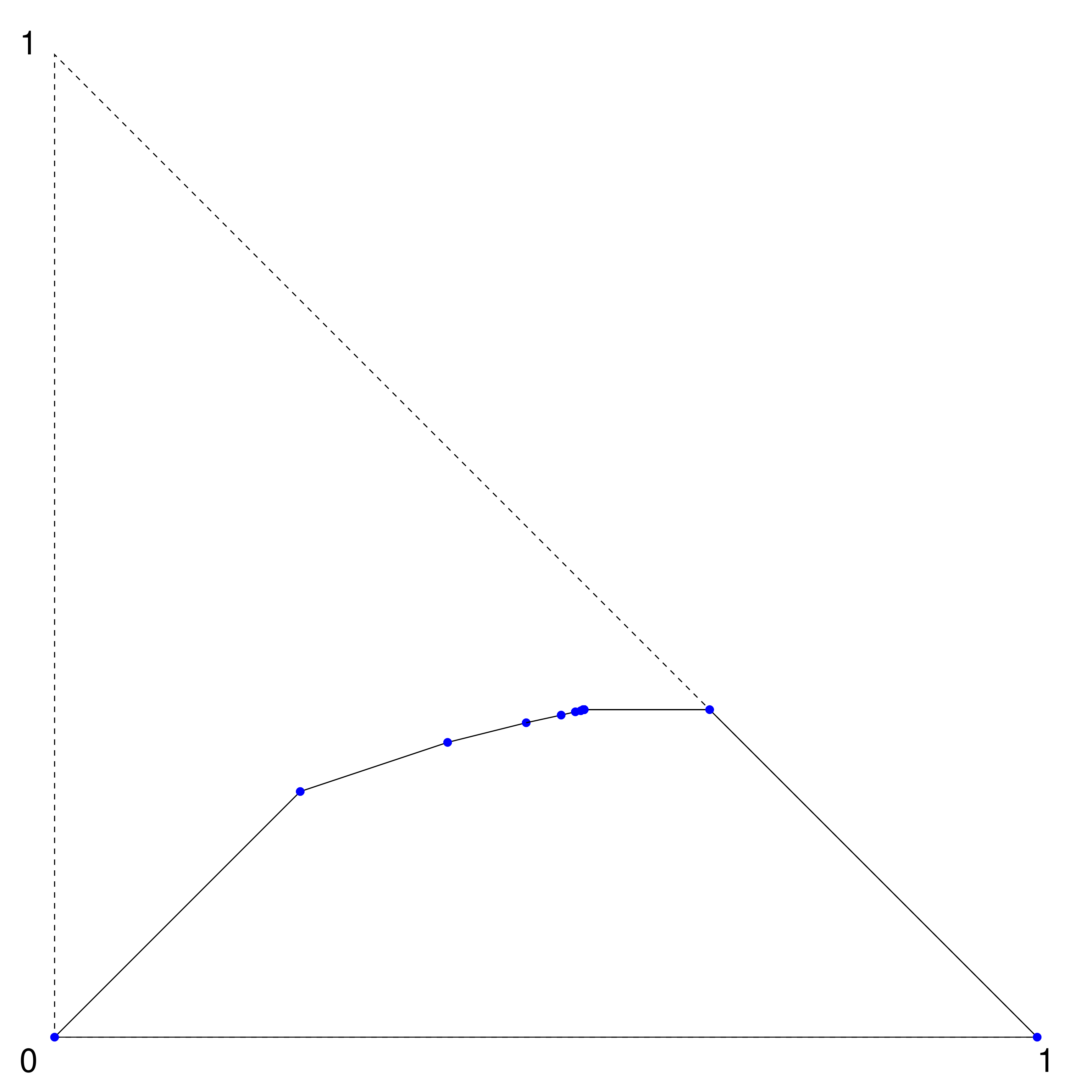}
\caption{An approximation to $\DF(\bn)$ where $\bn= 1\,1\,\overline{1\,0}$}
\label{fig:oneside}
\end{center}
\end{figure}
\end{examples}

\subsection{The digit frequency set varies continuously}

\begin{theorem}
\label{thm:continuous}
The function $\DF\colon\cN\to\cC(\Delta)$ is continuous.
\end{theorem}
\begin{proof}
Let $\bn\in\cN$. We shall show that $\DF$ is continuous at~$\bn$, and
we consider first the case where $\bn\in\N^\N$ is not of finite type. 

Let~$\epsilon>0$. By Lemma~\ref{lem:face-converge}, and since
$A_{\bn,r}\to\Phi^{-1}(\bn)$ as $r\to\infty$, there is some~$r$ with
\mbox{$d_H(\cF_{\bn,r}, A_{\bn,r}) < \epsilon$}. Then for any $\bm\in\cN$
which is close enough to~$\bn$ that $\Word{\bm}{r} = \Word{\bn}{r}$, we have,
by Lemma~\ref{lem:bounds},
\[\cF_{\bn,r} \subset L_{\bn,r} \subset \DF(\bm) \subset L_{\bn,r}\cup A_{\bn,r}
\quad\text{ and }\quad \cF_{\bn,r} \subset L_{\bn,r} \subset \DF(\bn)
\subset L_{\bn,r}\cup A_{\bn,r}.\] 
Therefore the symmetric difference
\[\DF(\bm) \bigtriangleup \DF(\bn) \subset A_{\bn,r} \subset
B_\epsilon(\cF_{\bn,r}) \subset B_\epsilon(\DF(\bm) \cap \DF(\bn)),\]
so that $d_H(\DF(\bm), \DF(\bn))<\epsilon$ as required.

\medskip\medskip

Next consider the case where $\bn = n_0\,\ldots\,n_{R-2}\,
n_{R-1}\,\infty\in\cN\setminus\N^\N$ (with $R\ge1$) is of finite
type. Now any $\bm\in\cN$ with $\bm\not=\bn$ and $d(\bm,\bn) <
2^{-\left(R+\sum_{s\le R-1}n_s\right)}$ is of the form $\bm =
n_0\,\ldots\,n_{R-2}\, n_{R-1}\,m_R\,\ldots$ for some $m_R\in\N$. It
therefore suffices to show that for all~$\epsilon>0$ there is some~$M$
such that every $\bm\in\cN$ of the form $\bm =
n_0\,\ldots\,n_{R-2}\, n_{R-1}\,m_R\,\ldots$ with $m_R\ge M$ has
$d_H(\DF(\bm), \DF(\bn)) < \epsilon$.

By~(\ref{eq:infinity-LF}) we have $\DF(\bn) =
\DF(n_0\,\ldots\,n_{R-2}\, (n_{R-1}+1)\, \overline{0})$. Applying
Lemma~\ref{lem:lfs-recursive} $R-1$ times to $\DF(\bn)$ and $\DF(\bm)$
and using the continuity of the maps $K_n^{-1}$, we can suppose that
$R=1$. We therefore need to show that if $\bm=n_0\,m_1\,\ldots$ with $m_1$
sufficiently large, then $d_H(\DF(\bm), \DF((n_0+1)\,\overline{0})) <
\epsilon$.

Now $\DF((n_0+1)\,\overline{0}) = T_{n_0}$, while
Lemma~\ref{lem:bounds} gives
\[L_{\bm,0} = T_{n_0} \subset \DF(\bm) \subset T_{n_0} \cup
K_{n_0}^{-1}(T_{m_1-1}) = U_{\bm,1}.\]
Since  $T_{m_1-1}\to\cF$ as $m_1\to\infty$ and $K_{n_0}^{-1}(\cF)
\subset T_{n_0}$, the result follows.

\medskip \medskip

The remaining case $\bn=\infty$ is straightforward since $\DF(\infty)
= \cF$ and $\DF(\bm) \subset T_{m_0-1} \to \cF$ as $m_0\to\infty$.
\end{proof}

Although $\DF(\bn)$ varies continuously with~$\bn$, the same is not
true of the set $\EP(\bn)$ of extreme points of $\DF(\bn)$. See
Figures~\ref{fig:cubes} and~\ref{fig:cubesmod}, which depict
respectively $\DF(\bn)$ and $\DF(\bm)$, in the case~$k=3$, for elements $\bn$ and $\bm$
of~$\N^\N$ which agree on their first 26 entries and are therefore
very close to each other. The two digit frequency sets are also very
close to each other, but the sets of extreme points are far apart. In
these examples, $n_r=r^3$ for $0\le r\le 25$ and $n_r=0$ for $r>25$;
while $m_r=n_r$ for all~$r$ except $r=26$, for which $m_r=100$.

\begin{figure}[htbp]
\begin{center}
\includegraphics[width=0.5\textwidth]{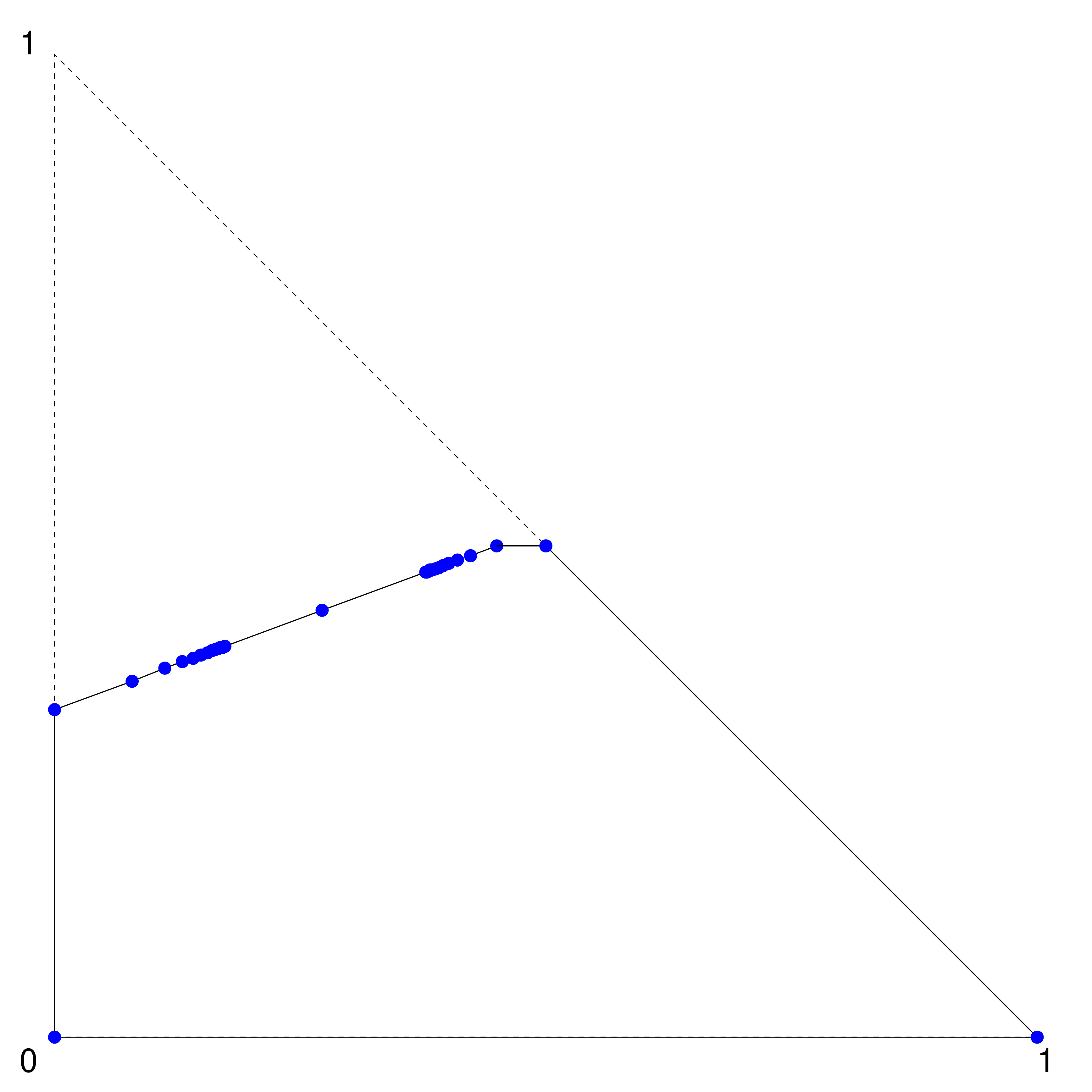}
\caption{A rational digit frequency set close to an exceptional one}
\label{fig:cubesmod}
\end{center}
\end{figure}

The proof of the following theorem shows how such examples can be
constructed formally in the case~$k=3$.

\begin{theorem}
\label{thm:EP-continuity}
Let~$k=3$. The function $\EP\colon\cN\to\cC(\Delta)$ is discontinuous
at $\bn\in\cN$ if and only if $\bn$ is either exceptional or of finite
type. 
\end{theorem}

\begin{proof}
Continuity in the case where $\bn\in\N^\N$ is regular is proved
similarly to continuity of $\DF$ (Theorem~\ref{thm:continuous}). If
$\bn\in\cN$ satisfies $\Word{\bm}{r}=\Word{\bn}{r}$ then
$\EP(\bm)\bigtriangleup\EP(\bn)\subset A_{\bn,r}$, and $A_{\bn,r}$
contains points of both $\EP(\bm)$ and $\EP(\bn)$. Since $A_{\bn,r}$
converges to the point $\Phi^{-1}(\bn)$ as $r\to\infty$, the result
follows. 

\medskip

Next, consider the case where $\bn = n_0\,\ldots\,n_r\,\infty$ is of
finite type, and set $\bn^{(i)} =
n_0\,\ldots\,n_r\,i\,i\,\overline{0}$ for each $i\ge 1$. Then
$\bn^{(i)}\to\bn$ as $i\to\infty$, and it suffices to show that 
$\EP(\bn^{(i)})$ does not converge to $\EP(\bn)$.

Recall that $\EP(\bn)=\EP(n_0\,\ldots\,n_{r-1}\,(n_r+1)\,\overline{0})$
by~(\ref{eq:infinity-LF}). One of the elements of $\EP(\bn^{(i)})$ is
\[
K_{n_0}^{-1}\circ\cdots\circ K_{n_r}^{-1}\circ K_i^{-1}\circ
K_i^{-1}(0,0,1) \to L:=K_{n_0}^{-1}\circ\cdots\circ
K_{n_{r-1}}^{-1}\left(
\frac{n_r+1}{n_r+3},\,\frac{1}{n_r+3},\,\frac{1}{n_r+3}
\right)\quad\text{ as }i\to\infty.
\] 
Since $\EP(\bn)$ has finitely many elements, it only remains to show
that none of them is equal to this limit~$L$. It is clearly impossible for
an extreme point $K_{n_0}^{-1}\circ\cdots\circ K_{n_s}^{-1}(0,1,0)$
(where $s\le r-1$) to coincide with $L$, since $(0,1,0)$ is not in the
$K$-orbit of any interior point of~$\Delta$. The only remaining
non-trivial element of $\EP(\bn)$ is
\[
K_{n_0}^{-1}\circ\cdots\circ K_{n_{r-1}}^{-1}\circ K_{n_r+1}^{-1}(0,0,1)
= K_{n_0}^{-1}\circ\cdots\circ K_{n_{r-1}}^{-1}\left(
\frac{n_r+1}{n_r+2},\,0,\,\frac{1}{n_r+2}
\right),
\]
which is also distinct from~$L$.

The case $\bn=\infty$ can be treated similarly by
considering~$\bn^{(i)}=i\,i\,\overline{0}$ and using
$\DF(\infty)=\cF$.

\medskip

Finally, then, consider the case where $\bn\in\N^\N$ is exceptional,
and let $\ell=\Phi^{-1}(\bn)$ be the exceptional interval in the
boundary of $\DF(\bn)$. Observe first that the points of $\EP(\bn)$
can only accumulate on the endpoints of~$\ell$. For the rational
elements of $\EP(\bn)$ are vertices of the 2-simplices~$A_{\bn,r}$,
which converge to~$\ell$, so all accumulation points must be
in~$\ell$; and an accumulation in the interior of~$\ell$ would
contradict the convexity of $\DF(\bn)$.

Let $\bal$ be the midpoint of $\ell$. The 2-simplices~$A_{\bn,r}$
contain~$\ell$ in their interior (since they have rational vertices,
exactly two of which lie in~$\DF(\bn)$), and every point of
$A_{\bn,r}$ has itinerary starting $n_0\,\ldots\,n_{r-3}$
(Section~\ref{sec:converge}). For each~$r\ge 0$, pick a rational point
$\bal^{(r)}$ of $A_{\bn,r}$ with $d(\bal, \bal^{(r)})<2^{-r}$, and
let~$\bn^{(r)} = \Phi(\bal^{(r)})$. Then $\bn^{(r)}\to\bn$,
$\bal^{(r)}\in \EP(\bn^{(r)})$, and $\bal^{(r)}\to\bal$. Since $\bal$
  is bounded away from $\EP(\bn)$, this establishes the discontinuity
  of $\EP$ at $\bn$, as required.
\end{proof}

\begin{remarks}\mbox{}
\begin{enumerate}[a)]
\item In the final paragraph of the proof, the rational points
  $\bal^{(r)}$ can be chosen all to be contained in, or all to be
  disjoint from, $\DF(\bn)$. Therefore $\EP$ is discontinuous from
  both sides at exceptional itineraries.
\item The proofs of continuity of $\EP$ in the regular non-finite case, and
of discontinuity in the finite case, work for all $k\ge 3$ (with minor
modifications in the finite case). The proof of discontinuity in the
exceptional case does not generalise so easily, principally because of
our ignorance of the set of accumulation points of $\E(\bn)$ in higher
dimensions (Remark~\ref{rmk:extreme-pt-thm}c)).
\end{enumerate}
\end{remarks}


\section{Application to $\beta$-expansions}
\label{sec:beta-expansion}

\subsection{Reinterpretation of results on symbolic $\beta$-shifts}

The results of Section~\ref{sec:beta-shift} will now be applied to
digit frequencies of $\beta$-expansions. We start by recalling some
notation and a key fact from Section~\ref{sec:intro}.

Fix throughout an integer $k\ge 2$ and work with $\beta\in(k-1,k)$, so
that digit sequences of $\beta$-expansions lie in $\Sigma =
\{0,\ldots,k-1\}^\N$, which we order lexicographically. For each~$\beta$, write $\DF(\beta)$ for
the set of all digit frequencies of (greedy)
$\beta$-expansions of numbers $x\in[0,1]$, a subset of the standard
$(k-1)$-simplex~$\Delta$. While it is convenient to restrict to
non-integer $\beta$, we note that $\DF(k)=\Delta$, and $\DF(k-1)=\cF$,
the face $\alpha_{k-1}=0$ of $\Delta$.

The set $Z_\beta$ of all digit
sequences~$d_\beta(x)$ of $\beta$-expansions of~$x\in[0,1]$ is given
by
\[Z_\beta = \{v\in\Sigma\,:\,\sigma^r(v) < w_\beta\text{ for all
}r\in\N\} \cup \{d_\beta(1)\},\]
where 
\[
w_\beta = 
\begin{cases}
\overline{d_1\ldots d_{r-1}(d_r-1)} & \text{ if }d_\beta(1) =
d_1\ldots d_{r-1}d_r\,\overline{0}\,\, \text{ for some~$r\ge1$ with }d_r>0,\\
d_\beta(1) & \text{ otherwise.}
\end{cases}
\]

We need the following elementary facts about the sequences
$d_\beta(1)$ and $w_\beta$:
\begin{lemma}\mbox{}
\label{lem:beta-facts}
\begin{enumerate}[a)]
\item The functions $\beta\mapsto d_\beta(1)$ and $\beta \mapsto
  w_\beta$ from $(k-1,k)$ to $\Sigma$ are strictly increasing.
\item $\{w_\beta\,:\,\beta\in(k-1,k)\}$ is equal to the set of
  elements~$w$ of~$\cM^*$ which are not equal to~$\overline{k-1}$ and
  are not of the form $w=v\,\overline{0}$ for any word~$v$. In
  particular, for every $\bn\in\cN\setminus\{\infty, \overline{0}\}$
  there is a unique $\beta\in(k-1,k)$ with $w_\beta = S(\bn)$.
\end{enumerate}
\end{lemma}

\begin{proof} 
Let $k-1<\beta<\gamma<k$. We have $d_\beta(1)_0 =
d_\gamma(1)_0=k-1$. If $d_\beta(1)_r = d_\gamma(1)_r$ for $0\le r\le
R$, then $f^R_\gamma(1) - f^R_\beta(1) \ge
\beta^R(\gamma-\beta)$. There is therefore a greatest~$R$ with
$d_\beta(1)_r = d_\gamma(1)_r$ for $0\le r\le R$, and since
$f_\gamma^R(1) > f_\beta^R(1)$ it follows that $d_\beta(1)_{R+1} <
d_\gamma(1)_{R+1}$. That is, $\beta\mapsto d_\beta(1)$ is strictly
increasing. The same is true for $\beta\mapsto w_\beta$ since
$w_\beta$ and $d_\beta(1)$ are either equal or consecutive maximal
sequences.

b) is a translation of the well known result (see for example
corollary~7.2.10 of~\cite{ACW}) that an element~$w$ of~$\cM^*$ is
equal to $d_\beta(1)$ for some~$\beta\in(k-1,k)$ if and only if it is
not periodic and not equal to $(k-1)\,\overline{0}$. If
$\bn\not=\infty$ and $\bn\not=\overline{0}$ then $S(\bn)$ is an
element of~$\cM^*$, not equal to $\overline{k-1}$, which is in the
image of $\Lambda_{n_0}$ and hence is not of the form
$v\,\overline{0}$: therefore $S(\bn) = w_\beta$ for some~$\beta$,
which is unique by~a).
\end{proof}

\begin{defn}[$\beta\colon\cN\setminus\{\infty,\,\overline{0}\}\to(k-1,k)$]
\label{defn:beta_n}
Define $\beta\colon \cN\setminus\{\infty,\,\overline{0}\}\to (k-1,k)$ by
\[
\beta(\bn) = \text{the unique $\beta$ with $w_\beta=S(\bn)$}.
\]
\end{defn}
This is a strictly increasing function by Lemmas~\ref{lem:S-op-homeo} and~\ref{lem:beta-facts}a).

The following definition and lemma make the connection between
$\DF(\beta)$ for \mbox{$\beta\in(k-1,k)$}, and $\DF(\bn)$ for
$\bn\in\cN$. The condition $w_0=k-1$ in the definition is to ensure that
\mbox{$S(\infty) = (k-1)\,\overline{0} \le w$}, so that the maximum is
defined.

\begin{defn}
\label{defn:n_w}
Let $\bn\colon\{w\in\Sigma\,:\,w_0 = k-1\} \to \cN$ be the function
defined by
\[
\bn(w) = \max\,\{\bm\in\cN\,:\,S(\bm)\le w\}.
\]
\end{defn}

\begin{lemma}
\label{lem:expansion-translate}
Let $\beta\in(k-1,k)$. Then $\DF(\beta) = \DF(\bn(w_\beta)) =
\DF(\bn(d_\beta(1)))$.
\end{lemma}

\begin{proof}
Suppose first that $f_\beta^r(1)\not=0$ for all~$r\in\N$, so that
$w_\beta = d_\beta(1)$. Then $\DF(\beta)$ is the set of
digit frequencies of elements of 
\[Z_\beta = \{v\in\Sigma\,:\,\sigma^r(v)< d_\beta(1)\text{ for all
}r\in\N\} \cup \{d_\beta(1)\},\] while, by Lemma~\ref{lem:w-in-J},
$\DF(\bn(w_\beta))$ is the set of digit frequencies of
elements of
\[X(w_\beta) = \{v\in\Sigma\,:\,\sigma^r(v) \le d_\beta(1)\text{ for
  all }r\in\N\}.\]
Now $Z_\beta\subset X(w_\beta)$, since $d_\beta(1)\in\cM$ and hence
$d_\beta(1)\in X(w_\beta)$. On the other hand, any element~$v$ of
$X(w_\beta)\setminus Z_\beta$ has $\sigma^r(v) = d_\beta(1)$ for
some~$r\ge 0$, and hence the digit frequency of~$v$, if it exists, is
equal to that of~$d_\beta(1)$. Therefore $\DF(\beta) = \DF(\bn(w_\beta))$.

Suppose instead that $d_\beta(1) = d_1\ldots
d_{r-1}d_r\overline{0}$, so that $w_\beta = \overline{d_1\ldots d_{r-1}(d_r -
1)}$.  Then $\DF(\beta)$ and $\DF(\bn(w_\beta))$ are the
sets of digit frequencies of elements of
\[Z_\beta = \{v\in\Sigma\,:\, \sigma^r(v) < \overline{d_1\ldots
d_{r-1}(d_r-1)} \text{ for all }r\in\N\} \cup \{d_\beta(1)\}\]
and
\[
X(w_\beta) = \{v\in\Sigma\,:\,\sigma^r(v) \le \overline{d_1\ldots
d_{r-1}(d_r-1)} \text{ for all }r\in\N\}
\]
respectively. Now $Z_\beta \setminus X(w_\beta) = \{d_\beta(1)\}$, which has digit
frequency $(1,0,\ldots,0)$, the same as the digit frequency of
$\overline{0}\in X(w_\beta)$. On the other hand, any element~$v$ of
$X(w_\beta)\setminus Z_\beta$ satisfies $\sigma^r(v) = w_\beta$ for
some~$r\in\N$, and hence has the same (rational) digit frequency as
$w_\beta$. 

Therefore the set of digit frequencies of $Z_\beta$ is a subset of the
set of digit frequencies of $X(w_\beta)$, and the difference between
the two is at most one point. However, since the $f_\beta$-orbit of
$1$ is finite, $\DF(\beta)$ is a polytope (see
Section~\ref{sec:markov-example}), and therefore cannot be obtained
from the compact convex set $\DF(\bn(w_\beta))$ by removal of a single
point. The two digit frequency sets are therefore equal. 

It remains to show that $\DF(\bn(w_\beta)) = \DF(\bn(d_\beta(1)))$. In
fact, $\bn(w_\beta) = \bn(d_\beta(1))$ for all~$\beta$. To see this,
observe that $S(\bn)$ is maximal for all $\bn\in\cN$, and $w_\beta$
and $d_\beta(1)$ are consecutive maximal elements if they are not
equal. Therefore if $\bn(w_\beta)\not=\bn(d_\beta(1))$ then
$d_\beta(1) = S(\bn)$ for some~$\bn\in\cN$. Since $d_\beta(1)$ has
digit frequency $(1,0,\ldots,0)$, this can only happen if
$\bn=\infty$ and $d_\beta(1) = (k-1)\,\overline{0}$. However if
$d_\beta(1)= (k-1)\,\overline{0}$ then $\beta = k-1$, a contradiction.
\end{proof}

Using this lemma we can interpret the results of
Section~\ref{sec:beta-shift} in terms of $\beta$-expansions. Before
doing so, we define intervals $I_{n_0\ldots n_R}
\subset (k-1,k)$ associated to each rational-finite pair. 
\begin{defn}[$I_{n_0\ldots n_R}$,\,\,$\cX$]
Given $R\ge 0$ and $n_0,\ldots,n_R\in\N$ write
\[I_{n_0\ldots n_R} = \left[
\beta(n_0\,\ldots\,n_{R-1}\,(n_R+1)\,\overline{0}),\,\, 
\beta(n_0\,\ldots\,n_{R-1}\,n_R\,\infty)
\right] \subset (k-1,k).
\]
Let $\cX$ denote the complement in~$(k-1,k)$ of the union of these
intervals. 
\end{defn}

\begin{theorem}
\label{thm:beta-exp-props}
Let~$k\ge 3$. Then
\begin{enumerate}[a)]
\item $\DF(\beta)$ is a compact convex set of dimension $k-1$ for all $\beta\in (k-1,k)$.
\item $\DF(\beta)$ has countably many extreme points, of which all but
  at most~$k-1$ are rational. There exist~$\beta$ for which the set of
  extreme points accumulates on $k-1$ non-rational points.
\item The extension $\DF\colon [k-1,k] \to \cC(\Delta)$ is continuous
  and increasing.
\item The $I_{n_0\ldots n_R}$ are mutually disjoint non-trivial closed
  subintervals of~$(k-1,k)$ whose union has full Lebesgue measure, on
  each of which the digit frequency set is a constant polytope with
  rational vertices.
\item The function $\beta\mapsto w_\beta$ restricts to a bijection
\[
  \cX \to \{S(\bn)\,:\,\bn \text{ is not of rational or finite type}\}.
\] 
In particular, $\DF$ is injective on~$\cX$, and $\DF(\cX)$
does not contain any polytopes.
\item The set $\DF([k-1,k]) \subset \cC(\Delta)$ is homeomorphic to a
  compact interval.
\end{enumerate}
\end{theorem}

\begin{proof} 
a) is a restatement of Corollary~\ref{cor:dfs-compact} and Theorem~\ref{thm:convex} (the digit frequency
set having dimension $k-1$ since it strictly contains $\cF$), and b)
is immediate from Theorem~\ref{thm:extreme-points} and
Remark~\ref{rmk:extreme-pt-thm}b), in each case using
Lemma~\ref{lem:expansion-translate}.

\medskip

For~c), consider first $\DF\colon(k-1,k)\to\cC(\Delta)$. The
functions $\beta\mapsto w_\beta$, $w\mapsto\bn(w)$, and
$\bn\mapsto\DF(\bn)$ are all increasing, the first by
Lemma~\ref{lem:beta-facts} and the other two by definition. Therefore
$\beta \mapsto \DF(\beta) = \DF(\bn(w_\beta))$ is also increasing. To
show that it is continuous, fix $\beta\in(k-1,k)$ and $\epsilon>0$,
and let $\bn = \bn(w_\beta)$. Since $\bn\mapsto\DF(\bn)$ is continuous
by Theorem~\ref{thm:continuous}, we can find $\bm,\bp\in\cN$ with
$\bm<\bn<\bp$ and with $d_H(\DF(\bm),\DF(\bp)) < \epsilon$. (If $\bn$
is of rational type then we take $\bp$ to be the corresponding element
of finite type, with $\DF(\bp) = \DF(\bn)$; and if not, there are
$\bp>\bn$ arbitrarily close to~$\bn$. Similarly if $\bn$ is of finite
type then we take $\bm$ to be the corresponding element of rational
type; and if not, there are $\bm<\bn$ arbitrarily close to $\bn$.)
Then $d_H(\DF(\beta), \DF(\gamma)) < \epsilon$ for all
$\gamma\in(\beta(\bm), \beta(\bp))$.

Since $\DF(k-1)=\cF$ and $\DF(k)=\Delta$, the extension to $[k-1,k]$
is clearly increasing.  That $\DF(\beta)\to\cF$ as $\beta\searrow k-1$
is a consequence of the fact that $\DF(\beta)\subset T_{n_0-1}$ if
$\bn(\beta)$ begins with $n_0$; and that $\DF(\beta)\to\Delta$ as
$\beta\nearrow k$ follows from the observation, using
Theorem~\ref{thm:extreme-points}, that if $\bn(\beta)$ begins with
$0^R$, where $R\ge k-2$, then every non-trivial extreme point of
$\DF(\beta)$ lies in $K_0^{-(R+2-k)}(\Delta)$, which converges
Hausdorff to $\{\be_{k-1}\}$ as $R\to\infty$.

\medskip

For d), the intervals~$I_{n_0\ldots n_R}$ are clearly closed and
non-trivial since $\bn\mapsto\beta(\bn)$ is strictly increasing. They
are mutually disjoint because
$n_0\,\ldots\,n_{R-1}\,(n_R+1)\,\overline{0}$ and
$n_0\,\ldots\,n_{R-1}\,n_R\,\infty$ are consecutive elements
of~$\cN$. By Theorem~\ref{thm:extreme-points}
and~(\ref{eq:infinity-LF}), $\DF(\beta) =
\DF(n_0\,\ldots\,n_{R-1}\,(n_R+1)\,\overline{0})$ is a constant
polytope on each interval. That the union of the intervals has
Lebesgue measure~1 is a consequence of
Theorem~\ref{thm:polytope-typical} below.

\medskip

Now suppose that $\beta$ is in the complement~$\cX$ of the union of
these intervals. Then, by Lemma~\ref{lem:gaps}, for every $n_0\ldots
n_R$, either we have $S(n_0\,\ldots\,n_{R-1}\,(n_R+1)\,\overline{0}) >
w_\beta$, or there is some $\bm\in\cN$ with
$S(n_0\,\ldots\,n_{R-1}\,n_R\,\infty) < S(\bm) < w_\beta$. Therefore
$\bn(w_\beta)$ is not of rational or finite type and, using
Lemma~\ref{lem:gaps} again, $w_\beta = S(\bn(w_\beta))$. Therefore the
image of~$\cX$ under $\beta\mapsto w_\beta$ is contained in the set of
infimax sequences which are not of rational or finite type. On the
other hand, every such sequence~$S(\bn)$ is equal to $w_{\beta(\bn)}$
where $\beta(\bn)\in\cX$. Since $\beta\mapsto w_\beta$ is strictly
increasing, it follows that it is a bijection from~$\cX$ to the set of
infimax sequences which are not of rational or finite type. Moreover,
by Theorem~\ref{thm:extreme-points}, $\DF(\beta)$ is not a polytope
for $\beta\in\cX$.

If $\beta,\gamma\in\cX$ with $\beta<\gamma$, then there is some $\bn$ of
rational type with \mbox{$w_\beta < S(\bn) < w_\gamma$}, and hence there is
some $\beta'$ between $\beta$ and $\gamma$ with $\DF(\beta')$
a polytope. This establishes the injectivity of $\DF$ on $\cX$.

\medskip

 By parts~d) and~e), collapsing each interval $I_{n_0\ldots n_R}$ to a
 point gives a compact interval on which $\DF$ descends to a
 continuous injection, so that the image of $\DF$ is a compact
 interval as required.
\end{proof}

\begin{remark}
\label{rmk:order}
One way to see the effect of exceptional elements on digit frequency
sets is to define a {\em forcing relation} $\le$ on~$\Delta$ by
\[\bal \le \bal' \iff \,\forall\beta\in(k-1,k),\,\,\, \bal'\in\DF(\beta)
\implies \bal\in\DF(\beta).\] By Lemma~\ref{lem:expansion-translate}
and~(\ref{eq:DFn}), we have
\[\bal\le\bal' \iff \Phi(\bal) \le \Phi(\bal')\]
(if $\Phi(\bal) > \Phi(\bal')$ then pick $\bn\in\cN$ with
$\Phi(\bal')<\bn<\Phi(\bal)$ and let $\beta = \beta(\bn)$: then
$\bal'\in \DF(\beta)$ but $\bal\not\in\DF(\beta)$).

The relation $\le$ is therefore reflexive, transitive, and total, but
is not antisymmetric when $k\ge 3$. In order to make it into a total
order, it is necessary to identify each exceptional
simplex in~$\Delta$ to a point.
\end{remark}

\begin{example} 
\label{ex:k=2}
Some parts of Theorem~\ref{thm:beta-exp-props} are not true in the
case~$k=2$, when $\DF(\beta)$ is a subset of the interval $\Delta =
\{(\alpha_0,\alpha_1)\in \R^2_{\ge0}\,:\,\alpha_0+\alpha_1=1\}$, which
we identify with $[0,1]$ using the homeomorphism
$(\alpha_0,\alpha_1)\mapsto \alpha_1$. Since $\DF(\beta)$ is compact
and convex, and $0 = \delta(d_\beta(0))\in\DF(\beta)$ for all
$\beta\in(1,2)$, the set $\DF(\beta) = [0,\rhe(\beta)]$ is determined
by its right hand endpoint $\rhe(\beta)$, which is the digit frequency
of the Sturmian sequence $S(\bn(d_\beta(1)))$.

Figure~\ref{fig:k=2} is a graph of $\rhe(\beta)$ against $\beta$,
showing how $\rhe(\beta)$ locks on each rational value. For
instance, the itinerary of the point $(1/2,1/2)\in\Delta$ is
$1\,\overline{0}$, so that $\rhe(\beta) = 1/2$ if and only if
$\beta\in I_0 = [\beta(1\,\overline{0}), \beta(0\,\infty)]$; that is,
if and only if $w_\beta\in[S(1\,\overline{0}), S(0\,\infty)] =
[\overline{10}, 1\overline{10}]$.

 Now $w_\beta = \overline{10}$ when $\beta^2-\beta-1 = 0$, and
 $w_\beta = 1\overline{10}$ when
 $\beta^3-\beta^2-2\beta+1=0$. This gives the endpoints (approximately
 1.618 and 1.802) of the interval $\{\beta\in(1,2)\,:\,\rhe(\beta) =
 1/2\}$.

\begin{figure}[htbp]
\begin{center}
\includegraphics[width=0.4\textwidth]{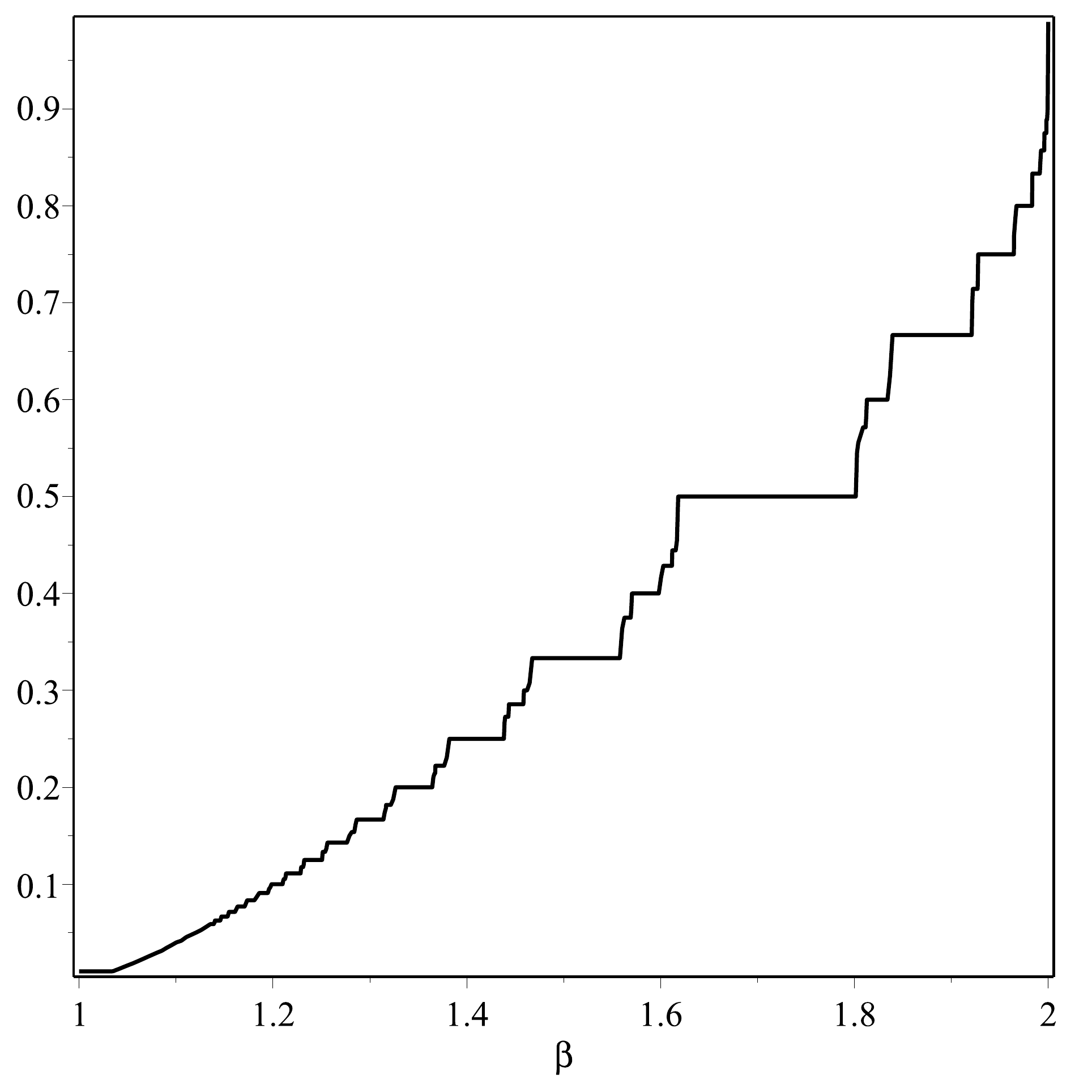}
\caption{The right hand endpoint of $\DF(\beta)$ when $k=2$}
\label{fig:k=2}
\end{center}
\end{figure}
\end{example}

\begin{example}
\label{ex:explicit-beta} 
Here we do analogous calculations to those of Example~\ref{ex:k=2} in
the cases $k=3$ and~$k=4$ (compare with Example~\ref{ex:ep} and
Figure~\ref{fig:short}). When $k=3$ we have $\DF(\beta) =
\DF(2\,1\,0\,1\,\overline{0})$ if and only if $\beta\in
I_{2\,1\,0\,0}$, i.e. if and only if
\[w_\beta \in [S(2\,1\,0\,1\,\overline{0}),\,\, S(2\,1\,0\,0\,\infty)] =
   [\overline{200120011},\,\, 2001\overline{20012000}].\]

 Now $w_\beta = \overline{200120011}$ when $\beta \simeq 2.190055$ (a
 root of $\beta^9-2\beta^8-\beta^5-2\beta^4-\beta-2$), and \mbox{$w_\beta =
 2001\overline{20012000}$} when $\beta\simeq 2.19019$ (a root of
 $\beta^{12}-2\beta^{11}-\beta^8-2\beta^7-2\beta^4+1$). Thus
 $\DF(\beta)$ locks on the pentagon of Example~\ref{ex:ep}a) and
 Figure~\ref{fig:short} for~$\beta$ between these two values.

On the other hand, when~$k=4$ we have $\DF(\beta) =
\DF(2\,1\,0\,1\,\overline{0})$ if and only if
\[
w_\beta \in [S(2\,1\,0\,1\,\overline{0}),\,\, S(2\,1\,0\,0\,\infty)] =
[\overline{30013000},\,\, 300\overline{13001}],
\]
 so that the polyhedron with 7 vertices of Example~\ref{ex:ep}b) is
 equal to $\DF(\beta)$ for $\beta$ between roots of
 $\beta^8-3\beta^7-\beta^4-3\beta^3-1$ and
 $\beta^8-3\beta^7-\beta^4-4\beta^3+3\beta^2-1$ (approximately 3.0688
 and 3.0690).
\end{example}

\subsection{Typical phenomena}
\label{sec:typical}
In this section we shall show that, from the point of view of the
parameter~$\beta$, the typical digit frequency set is of rational type
(that is, a polytope with rational vertices). By constrast, we then
show that from the point of view of the digit frequency sets
themselves, the generic example is non-rational and
regular (that is, having a single accumulation of rational vertices);
and moreover, the non-rational extreme point is generically {\em totally
  irrational} (its components are independent over the rationals).

\begin{theorem}
\label{thm:polytope-typical}
Let~$k\ge 3$. Then $\DF(\beta)$ is a polytope with rational vertices
for Lebesgue a.e. \mbox{$\beta\in(k-1,k)$}. 
\end{theorem}

\begin{proof} 
For each $\beta\in(k-1, k)$, let $\bp(\beta)\in\Delta$ be the {\em
  normal} digit frequency for $\beta$-expansions, which is realised by
$d_\beta(x)$ for Lebesgue a.e. $x\in[0,1]$.  It is given by
\[p_i(\beta) = \int_{i/\beta}^{(i+1)/\beta} h_\beta 
\qquad \text{for $0\le i < k-1$, \quad and} \qquad p_{k-1}(\beta) =
\int_{(k-1)/\beta}^{1} h_\beta 
,
\]
where $h_\beta\colon[0,1]\to\R^+$ is the density of Parry's
measure of maximal entropy~\cite{Parry}. Now $h_\beta$ is a decreasing
function, from which it follows that
\begin{equation}
\label{eq}
\frac{p_{k-1}(\beta)}{p_0(\beta)} \le \frac{1-(k-1)/\beta}{1/\beta} =
\beta - (k-1).
\end{equation}

By a theorem of Schmeling~\cite{Schmeling}, the sequence $w_\beta$ has
digit frequency $\bp(\beta)$ for Lebesgue a.e.\ $\beta$ in $(k-1,k)$. It
therefore suffices to prove that $w_\beta$ does {\em not} have digit
frequency $\bp(\beta)$ whenever $\DF(\beta)$ is not a polytope with
rational vertices.

Suppose therefore that $\DF(\beta)$ is not a polytope with rational
vertices, so that, by Theorem~\ref{thm:beta-exp-props}, $w_\beta =
\cI(\bal)$ for some~$\bal$ which is not of rational type. If the digit
frequency of $w_\beta$ exists then it is equal to $\bal$ by
Facts~\ref{facts:infifacts}f), so that it is only necessary to show
that $\bal\not=\bp(\beta)$.

Let $\bn = \Phi(\bal)$. Then $\bal\in \Delta_{n_0}$, which means by
definition that $\alpha_{k-1}/\alpha_0 > 1/(n_0+1)$ (we can assume
that $\alpha_0>0$, since otherwise it is immediate that
$\bal\not=\bp(\beta)$). We shall show that $1/(n_0+1) > \beta - (k-1)$,
which will establish the result by comparison with~(\ref{eq}). This
statement is immediate if~$n_0=0$, so we suppose $n_0\ge 1$.

Since $\bn < n_0\,\overline{0}$, we have $w_\beta = \cI(\bal) <
S(n_0\,\overline{0}) = \overline{(k-1)\,0^{n_0}}$, and hence
$\beta<\beta'$, where $w_{\beta'} = \overline{(k-1)\,0^{n_0}}$. Now $\beta'$
is the unique root in $(k-1,k)$ of the function $f(x) =
x^{n_0+1}-(k-1)x^{n_0}-1$. Since this function is increasing
in~$(k-1,k)$, showing that $f((k-1) + 1/(n_0+1)) \ge 0$ will establish
that $\beta < \beta' \le (k-1) + 1/(n_0+1)$.

Now
\begin{eqnarray*}
f\left(k-1 + \frac{1}{n_0+1}\right) &=& \left(k-1 +
\frac{1}{n_0+1}\right)^{n_0+1} - (k-1)\left(k-1 +
\frac{1}{n_0+1}\right)^{n_0}-1 \\
&=& \frac{1}{n_0+1}\left(k-1 + \frac{1}{n_0+1}\right)^{n_0} - 1\\
&>&\,\, \frac{(k-1)^{n_0}}{n_0+1} - 1 \,\,\ge\,\, \frac{2^{n_0}}{n_0+1}-1\,\, \ge\,\, 0
\end{eqnarray*}
as required.
\end{proof}

Before embarking on the proof that the generic digit frequency set is
non-rational and regular, we introduce some notation for the various
spaces which will be involved. We fix throughout the integer $k\ge 3$.

An element~$\bal$ of~$\Delta'$ is
said to be {\em completely irrational} if there is no non-trivial
relationship of the form $\sum_{i=0}^{k-1}m_i\alpha_i = 0$ for
integers~$m_i$. It is said to be of {\em infinite type} if every
component of $K^r(\bal)$ is strictly positive for all $r\ge 0$.

We define the following subsets of $\Delta'$:
\begin{eqnarray*}
\Rat_\Delta &=& \mbox{the set $\Delta'\cap\Q^k$ of rational elements},\\
\CI_\Delta &=& \mbox{the set of completely irrational elements},\\
\In_\Delta &=&
\mbox{the set of infinite type elements},\\
\Reg_\Delta &=& \mbox{the set of regular elements, and}\\
\cO_\Delta &=& \mbox{the set of elements whose itinerary contains
  infinitely many words $1^{2k-3}$.}
\end{eqnarray*}

The images of these sets under the itinerary map
$\Phi\colon\Delta'\to\N^\N$ are denoted with subscripts~$S$ (for
``sequence''). Thus, for example, $\Rat_S = \Phi(\Rat_\Delta)$ is the
set of elements of $\N^\N$ which end~$\overline{0}$; $\cO_S =
\Phi(\cO_\Delta)$ is the set of elements of $\N^\N$ which contain
infinitely many distinct words $1^{2k-3}$; and, by
Facts~\ref{facts:infifacts}i), $\In_S = \Phi(\In_\Delta)$ is the set
of elements~$\bn$ of $\N^\N$ which have the property that, for
all~$r\ge0$, there is some~$s\ge0$ with
$n_{r+s(k-1)}\not=0$. Facts~\ref{facts:infifacts}h) states that
$\cO_S\subset\Reg_S$ is a dense $G_\delta$ subset of $\N^\N$.

Let $\cD = \DF(\N^\N\setminus\{\overline{0}\})$ be the set of all digit frequency sets for
$\beta\in(k-1,k)$ with the Hausdorff topology, which is homeomorphic
to an open interval. The images of the above subsets of $\N^\N$ under
$\DF$ will be denoted with a subscript~$D$ (for ``digit''). Thus, for
example, $\Rat_D$ is the set of digit frequency sets which are
polytopes, and $\Reg_D$ is the set of digit frequency sets which are
either polytopes or have a single non-rational extreme point.

The complements of these sets in $\Delta'$, in $\N^\N$, or in $\cD$,
as appropriate, are denoted with a superscript~$c$. Thus, for example,
$\Rat^c_D$ is the set of non-polytope digit frequency sets.

Recall that a function $f\colon X\to Y$ is called {\em quasi-open} if,
for every open set $U\subset X$, the image $f(U)\subset Y$ has
interior.

\begin{lemma} \mbox{}
\label{lem:typical-props}
\begin{enumerate}[a)]
\item $\CI_\Delta$ and $\In_\Delta$ are dense $G_\delta$ subsets of $\Delta'$, with
  $\CI_\Delta \subset \In_\Delta$.
\item The itinerary map $\Phi\colon\Delta'\to\N^\N$ is
  quasi-open; its restriction
  $\Phi\colon\In_\Delta\to\In_S$ is continuous; and its restriction
  $\Phi\colon\In_\Delta\cap\,\cO_\Delta\to\In_S\cap\,\cO_S$ is a
  homeomorphism.
\item $\In_S$ is a dense $G_\delta$ subset of $\N^\N$.
\item The restriction $\DF\colon \Rat^c_S \to Rat^c_D$ is a homeomorphism.
\end{enumerate}
\end{lemma}

\begin{proof}
\mbox{}
\begin{enumerate}[a)]
\item $\CI_\Delta$ is the countable intersection of the open dense
  subsets $\sum_{i=0}^{k-1}m_i\alpha_i\not=0$ of the Baire
  space~$\Delta'$, and is therefore dense $G_\delta$. Similarly for
\[\In_\Delta = \bigcap_{R\ge 0}\,\,\bigcap_{n_0,\ldots,n_R\in\N} \left(
\Delta' \setminus K_{n_0}^{-1}\circ \cdots \circ K_{n_R}^{-1}(\partial
\Delta)\right),\] where $\partial\Delta$ is the union of the faces of
$\Delta$.

The image of a
  completely irrational vector under a projective homeomorphism with
  integer coefficients is again completely irrational, and in
  particular has no zero components, from which it
  follows that $\CI_\Delta\subset \In_\Delta$. 

\item Let~$U\subset\Delta'$ be open, pick $\bal\in U\cap\Rat_\Delta$,
  and write $\bn=\Phi(\bal)$. Because
  \mbox{$\Rat_\Delta\subset\Reg_\Delta$}, we have $A_{\bn, r}\subset
  U$ for sufficiently large~$r$. Therefore $\Phi(U)$ contains the
  (open) cylinder set determined by the block $n_0\,\ldots\,n_r$ for
  sufficiently large~$r$, establishing that $\Phi$ is quasi-open as
  required.

To show that the restriction $\Phi\colon\In_\Delta\to\In_S$ is
continuous, observe that the hyperplanes \mbox{$\alpha_0 =
  n\alpha_{k-1}$} on which $K\colon\Delta'\to\Delta'$ is discontinuous
are contained in $\In^c_\Delta$. It follows that if
$\bal\in\In_\Delta$ and $r>0$, then there is a neighbourhood~$U$ of
$\bal$ in $\Delta'$ such that $K^s(U)$ is disjoint from these 
hyperplanes for $0\le s\le r$. Then $\Phi(U)$ is contained in the
cylinder set determined by the first~$r$ symbols of $\Phi(\bal)$,
which establishes continuity.

In particular, the restriction $\Phi\colon\In_\Delta\cap\,\cO_\Delta \to
\In_S\cap\,\cO_S$ is continuous. It is also bijective because points of
$\Phi(\cO_\Delta) = \cO_S$ have unique preimages under~$\Phi$.  It
therefore only remains to show that it is open. For this it is
required to show that if $U\subset\Delta'$ is open and $\bn\in\Phi(U
\cap \In_\Delta\cap\,\cO_\Delta)$, then there is an open subset~$V$ of
$\N^\N$ with $\bn\in V\cap\In_S\cap\,\cO_S \subset \Phi(U\cap
\In_\Delta\cap\, \cO_\Delta)$. As above, the fact that $\bn$ is regular
means that $A_{\bn,r}\subset U$ for sufficiently large~$r$, so that
the cylinder set~$V$ determined by $n_0\,\ldots\,n_r$ satisfies
$V\subset \Phi(U)$. Therefore
\[V\cap\In_S\cap\, \cO_S \,\subset\, \Phi(U) \cap \Phi(\In_\Delta) \cap
\Phi(\cO_\Delta) \,=\, \Phi(U\cap\In_\Delta\cap\,\cO_\Delta),\] with the
final equality holding since points of $\Phi(\cO_\Delta)$ have a
unique $\Phi$-preimage in~$\Delta'$. This completes the proof, since
clearly $\bn\in V\cap\In_S\cap \,\cO_S$.

\item  
\[\In_S = \bigcap_{r=0}^{\infty}
\{\bn\in\N^\N\,:\,n_{r+s(k-1)} \not = 0 \text{ for some }s\ge 0\},\]
a countable intersection of open dense subsets of the Baire
space~$\N^\N$. 

\item $\DF\colon\Rat^c_S \to \Rat^c_D$ is continuous by
  Theorem~\ref{thm:continuous}, injective by
  Corollary~\ref{cor:almost-injective}, and surjective by
  definition. It therefore only remains to show that it is open.
Since the cylinder sets form a basis for the topology of
  $\N^\N$, it suffices to show that for each cylinder set~$C$, the set
  $\DF(C\setminus \Rat_S)$ is open in $\Rat^c_D$. Suppose, then,
  that~$C$ is determined by the block $n_0\,\ldots\,n_r$. Write $\ell
  = n_0\,\ldots\,n_r\,\infty$ and $r =
  n_0\,\ldots\,n_r\,\overline{0}$, so that $C = [\ell,r]_\cN \cap
  \N^\N$, where $[a,b]_\cN := \{\bn\in\cN\,:\,a\le \bn\le b\}$. Since
  $\DF$ is continuous and order-preserving on~$\cN$ and $\DF(\ell),
  \DF(r)\in \Rat_D$, we have that
\[\DF(C\setminus\Rat_S) = [\DF(\ell), \DF(r)]_D \setminus \Rat_D =
(\DF(\ell), \DF(r))_D \setminus \Rat_D\] 
is open in $\Rat^c_D$ as required.
\end{enumerate}

\end{proof}

\medskip\medskip\medskip\medskip

We can now prove that a generic digit frequency set has a
single limiting extreme point which is completely irrational.

\medskip\medskip

\begin{theorem}
\label{thm:typical-regular}
The set $\Reg_D\cap\CI_D$ contains a dense $G_\delta$ subset
of~$\cD$.
\end{theorem}

\begin{proof}
In this proof references a), b), c), and d) are to the parts of
Lemma~\ref{lem:typical-props}, while numerical references 1) and 2) are
to the following straightforward facts about subspaces $A\subset
B\subset X$ of a metric space~$X$:
\begin{enumerate}[1)]
\item If $A$ is a dense (respectively $G_\delta$) subset of
  $X$, then it is a dense (respectively $G_\delta$) subset of $B$.
\item If $B$ is a dense (respectively $G_\delta$) subset of
  $X$, and $A$ is a dense (respectively $G_\delta$) subset of $B$, then $A$
  is a dense (respectively $G_\delta$) subset of $X$.
\end{enumerate}

\medskip

By c) and Facts~\ref{facts:infifacts}h), $\cO_S\cap\In_S$ is a dense
$G_\delta$ subset of $\N^\N$. Because $\cO_S\subset\Reg_S$, its
preimage $\Phi^{-1}(\cO_S\cap\In_S)$ is equal to
$\cO_\Delta\cap\In_\Delta$.

Now $\cO_S\cap\In_S$ is a $G_\delta$ subset of $\In_S$ by 1), so by
the continuity of $\Phi\colon \In_\Delta\to\In_S$, its preimage
$\cO_\Delta\cap\In_\Delta$ is a $G_\delta$ subset of $\In_\Delta$, and
hence, by a) and 2), of $\Delta'$. On the other hand, it follows from
the quasi-openness of $\Phi\colon\Delta'\to\N^\N$ and the denseness of
$\cO_S\cap\In_S$ in $\N^\N$ that $\cO_\Delta\cap\In_\Delta$ is dense
in~$\Delta'$. That is, $\cO_\Delta\cap\In_\Delta$ is a dense
$G_\delta$ subset of $\Delta'$. By a),
$\cO_\Delta\cap\In_\Delta\cap\CI_\Delta = \cO_\Delta\cap\CI_\Delta$ is
also a dense $G_\delta$ subset of $\Delta'$.

In particular, by 1), $\cO_\Delta\cap\CI_\Delta$ is a dense
$G_\delta$ subset of  $\cO_\Delta\cap\In_\Delta$. Therefore, by~b),
$\Phi(\cO_\Delta\cap\CI_\Delta)$ is a dense
$G_\delta$ subset of $\cO_S\cap\In_S$, and so also of $\N^\N$ by 2);
but $\Phi(\cO_\Delta\cap\CI_\Delta) =
\Phi(\cO_\Delta)\cap\Phi(\CI_\Delta) = \cO_S\cap\CI_S$ since
$\cO_S\subset\Reg_S$. 

Now $\cO_S\cap\CI_S$ is a dense $G_\delta$ subset of $\Rat_S^c$ by 1),
and hence $\DF(\cO_S\cap\CI_S)$ is a dense $G_\delta$ subset of
$\Rat_D^c$ by d). Since $\Rat_D^c$ is a dense $G_\delta$ subset of
$\cD$ (it has countable complement), it follows by 2) that
the subset $\DF(\cO_S\cap \CI_S)$ of $\Reg_D\cap\CI_D$ is dense
$G_\delta$ in $\cD$ as required.
\end{proof}

The following lemma provides explicit elements of
$\Reg_S\cap\CI_S$, and hence of $\Reg_D\cap\CI_D$.

\begin{lemma}
\label{lem:explicit-CI}
Suppose that~$k\ge 3$. For every $n>0$, the element $\overline{n}$ of
$\N^\N$ lies in $\Reg_S\cap\CI_S$.
\end{lemma}
\begin{proof} 
The itinerary $\overline{n}$ lies in $\Reg_S$ since it is
bounded and contains no zeroes. It is therefore only necessary to
prove that if $\Phi(\bal)=\overline{n}$, then $\bal$ is completely
irrational.

Let~$A$ be the $k$ by $k$ matrix which is the Abelianization of the
substitution~$\Lambda_n$: that is, $A_{0,k-2}=n+1$,
$A_{0,k-1}=n$, $A_{i,i-1}=1$ for $1\le i\le k-1$, $A_{k-1,k-1}=1$, and all
other entries are zero. Then $\bal=\bv/||\bv||_1$, where $\bv$ is the
Perron-Frobenius eigenvector for~$A$ normalized so
that $v_{k-1}=1$.

The eigenvector equation $A\bv = \lambda\bv$ gives $v_i =
\lambda^{k-2-i}(\lambda-1)$ for $0\le i\le k-2$. Therefore if
$m_i\in\Z$ for $0\le i\le k-1$, then $\sum_{i=0}^{k-1} m_iv_i$ is a
polynomial in $\lambda$ of degree at most $k-1$ with integer
coefficients. Hence if $\bal$ is not completely irrational, then the
degree of~$\lambda$ is less than~$k$, and therefore the characteristic
polynomial \mbox{$p(x) = x^k-x^{k-1}-(n+1)x+1$} of~$A$ is reducible
over~$\Z$. Conversely, if $q(\lambda)=0$ for some non-zero integer
polynomial $q(x) = \sum_{i=0}^{k-1}q_ix^i$ of degree less than~$k$,
then we can construct integers $m_i = \sum_{j=k-1-i}^{k-1} q_j$, not all
zero, with $\sum_{i=0}^{k-1} m_iv_i = 0$. Therefore $\bal$ is completely
irrational if and only if $p(x)$ is irreducible over~$\Z$.

Let $r(x) = x^k p(1/x) = x^k - (n+1)x^{k-1}-x+1$ which is irreducible
if and only if~$p$ is. The Perron criterion for irreducibility applies
to $r$ to give the required result when $n+1>3$. A slight variation in
the first part of the usual proof of Perron's criterion (for example,
in the proof of Theorem~2.2.5(a) of~\cite{polycite}, take $g(x) =
-(n+1)x^{k-1} + 1$ which has all $k-1$ roots in the unit disk and
satisfies $|g|>|r-g|$ on the unit circle, implying by Rouch\'e's
theorem that~$r$ also has $k-1$ roots in the unit disk) gives the
result for $n=1$ and $n=2$ also.
\end{proof}

\subsection{Generic smoothness at non-rational extreme points}
Figures~\ref{fig:cubes} and~\ref{fig:squares} suggest that, in the
case~$k=3$, non-rational (i.e. limiting) extreme points~$\bal$ of
digit frequency sets $\DF(\bn)$ are {\em smooth}: that is, that there
is a unique line~$L$ through~$\bal$ in the plane of $\DF(\bn)$ such
that $\DF(\bn)\setminus L$ is connected. In this section we show that
this property holds for all $\bn$ in $\cP_S$, the set of itineraries
which contain infinitely many words $1\,1\,1\,1\,1$. By arguments
analogous to those of Section~\ref{sec:typical}, $\cP_S$ and its
counterparts $\cP_\Delta$ and $\cP_D$ are generic subsets of $\N^\N$,
$\Delta$, and $\cD$ respectively. Note that, by
Facts~\ref{facts:infifacts}h), we have $\cP_S \subset\Reg_S$.

We restrict to the case~$k=3$ throughout. Let $\bal\in\Delta$ have
itinerary $\Phi(\bal) = \bn$. Recall from Section~\ref{sec:converge}
that we write, for each $r\ge 0$,
\begin{eqnarray*}
\Lambda_{\bn,r} &=&
\Lambda_{n_0}\circ\Lambda_{n_1}\circ\cdots\circ\Lambda_{n_r}, \\
\Upsilon_{\bn,r} &=& K_{n_0}^{-1}\circ K_{n_1}^{-1}\circ\cdots \circ
K_{n_r}^{-1} \colon \Delta\to\Delta, \qquad\text{ and}\\
A_{\bn,r} &=&\Upsilon_{\bn,r}(\Delta).
\end{eqnarray*}

Thus $A_{\bn,r}$ is a triangle, with rational vertices labelled
$\VV{r}{0}=\Upsilon_{\bn,r}(1,0,0)$,
$\VV{r}{1}=\Upsilon_{\bn,r}(0,1,0)$, and
$\VV{r}{2}=\Upsilon_{\bn,r}(0,0,1)$. The vertices $\VV{r}{0}$ and
$\VV{r}{1}$ are contained in $\DF(\bn)$; on the other hand, since
$\Phi^{-1}(\bn)\subset A_{\bn,r}$, we have $\VV{r}{2}\not\in\DF(\bn)$.

By~(\ref{eq:alpha-eqs-1}) and~(\ref{eq:alpha-eqs-2}), the triangles
$A_{\bn,r}$ evolve according to
\begin{eqnarray*}
\VV{r+1}{0} &=& \VV{r}{1},\\
\VV{r+1}{1} &=& \frac
{(n_{r+1}+1)\LL{r}{0}\VV{r}{0} + \LL{r}{2}\VV{r}{2}}
{(n_{r+1}+1)\LL{r}{0} + \LL{r}{2}}, \qquad\text{ and }\\
\VV{r+1}{2} &=& \frac
{n_{r+1}\LL{r}{0}\VV{r}{0} + \LL{r}{2}\VV{r}{2}}
{n_{r+1}\LL{r}{0} + \LL{r}{2}},
\end{eqnarray*}
where $\LL{r}{i} = |\Lambda_{\bn,r}(i)|$. Therefore both $\VV{r+1}{1}$
and $\VV{r+1}{2}$ lie on the edge of $A_{\bn,r}$ with endpoints
$\VV{r}{0}$ and~$\VV{r}{2}$, and cut this edge in the ratios
$(n_{r+1}+1)\LL{r}{0}\,:\,\LL{r}{2}$ and
$n_{r+1}\LL{r}{0}\,:\,\LL{r}{2}$ respectively. See
Figure~\ref{fig:evolve}, in which $A_{\bn,r}$ is shown with dashed
edges and $A_{\bn,r+1}$ with solid edges. 

\begin{figure}[htbp]
\begin{center}
\includegraphics[width=0.85\textwidth]{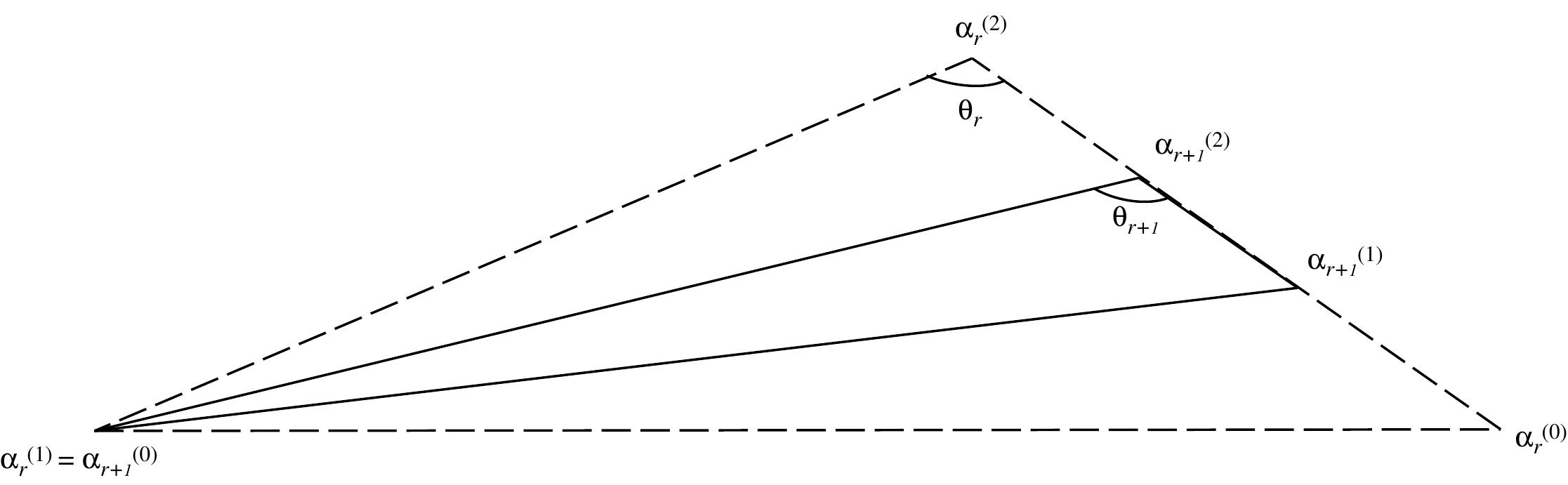}
\caption{Evolution of the triangles $A_{\bn,r}$}
\label{fig:evolve}
\end{center}
\end{figure}

As in the figure, let $\theta_r$ denote the angle of the triangle
$A_{\bn,r}$ at the vertex $\VV{r}{2}$. Then $\theta_{r+1}\ge \theta_r$
for all~$r$, with equality if and only if $n_{r+1}=0$ (i.e. if and
only if $\VV{r+1}{2} = \VV{r}{2}$). Generic
smoothness at non-rational extreme points is a consequence of the
following lemma.

\begin{lemma}
\label{lem:angle-to-pi}
Suppose that $\bn\in\cP_S$, and moreover that $n_0=n_1=1$. Then
$\theta_r\to\pi$ as $r\to\infty$.
\end{lemma}
\begin{proof}
Suppose for a contradiction that $\theta_r\to\theta=\pi-2\epsilon$ as
$r\to\infty$, where $\epsilon>0$. A direct calculation using
$n_0=n_1=1$ gives $\VV{1}{0}= (2/3,0,1/3)$, $\VV{1}{1}=(1/4,1/2,1/4)$,
and $\VV{1}{2} = (1/3,1/3,1/3)$, so that $\cos\theta_1<0$, and hence
$\theta>\pi/2$. 

Now pick any $r$ with $n_{r+s}=1$ for $0\le s\le 4$. Using
$\Lambda_1(0)=1$, $\Lambda_1(1)=2\,0\,0$ and $\Lambda_1(2)=2\,0$, we
have
\begin{eqnarray*}
\LL{r+s}{0} &=& \LL{r+s-1}{1},\\
\LL{r+s}{1} &=& \LL{r+s-1}{2} + 2\LL{r+s-1}{0}, \qquad\text{ and}\\
\LL{r+s}{2} &=& \LL{r+s-1}{2} + \LL{r+s-1}{0}
\end{eqnarray*}
for $0\le s\le 4$. Writing $(a,b,c) = (\LL{r-1}{0}, \LL{r-1}{1},
\LL{r-1}{2})$, this gives 
\begin{eqnarray*}
(\LL{r+2}{0}, \LL{r+2}{1}, \LL{r+2}{2}) &=& (a+2b+c, 5a+b+3c, 3a+b+2c)
  \qquad\text{ and}\\
(\LL{r+3}{0}, \LL{r+3}{1}, \LL{r+3}{2}) &=& (5a+b+3c, 5a+5b+4c,
  4a+3b+3c).
\end{eqnarray*}
Therefore both $n_{r+3}\LL{r+2}{0}/\LL{r+2}{2} =
\LL{r+2}{0}/\LL{r+2}{2}$ and $n_{r+4}\LL{r+3}{0}/\LL{r+3}{2} =
\LL{r+3}{0}/\LL{r+3}{2}$ are greater than~$1/3$. That is,
$\VV{r+3}{2}$ is at least $1/3$ of the way along the edge of
$A_{\bn,r+2}$ from $\VV{r+2}{2}$ to $\VV{r+2}{0}$; and similarly
$\VV{r+4}{2}$ is at least $1/3$ of the way along the edge of
$A_{\bn,r+3}$ from $\VV{r+3}{2}$ to $\VV{r+3}{0}$.

Let $\varphi_r$ and $\psi_r$ denote the angles of the triangle
$A_{\bn,r}$ at the vertices $\VV{r}{0}$ and $\VV{r}{1}$
respectively. Then $\varphi_r+\psi_r=\pi-\theta_r>2\epsilon$ for
all~$r$; moreover (see Figure~\ref{fig:evolve}),
$\psi_{r+1}>\varphi_r$ for all~$r$, so that if $\psi_r<\epsilon$
then $\psi_{r+1}>\epsilon$.

We can therefore pick an index~$r$ within each block of symbols
$1\,1\,1\,1\,1$ in $\bn$ with the property that both $\psi_r>\epsilon$
and $n_{r+1}\LL{r}{0}/\LL{r}{2}>1/3$. Figure~\ref{fig:evolve-detail}
shows the corresponding triangle $A_{\bn,r}$ and the relevant edges of
$A_{\bn,r+1}$. Let $u$, $v$, $w$, and $\gamma_r$ denote the lengths
and angle indicated in the figure. The condition
$n_{r+1}\LL{r}{0}/\LL{r}{2}>1/3$ gives that $(u+v)/v>4/3$.

\begin{figure}[htbp]
\begin{center}
\includegraphics[width=0.85\textwidth]{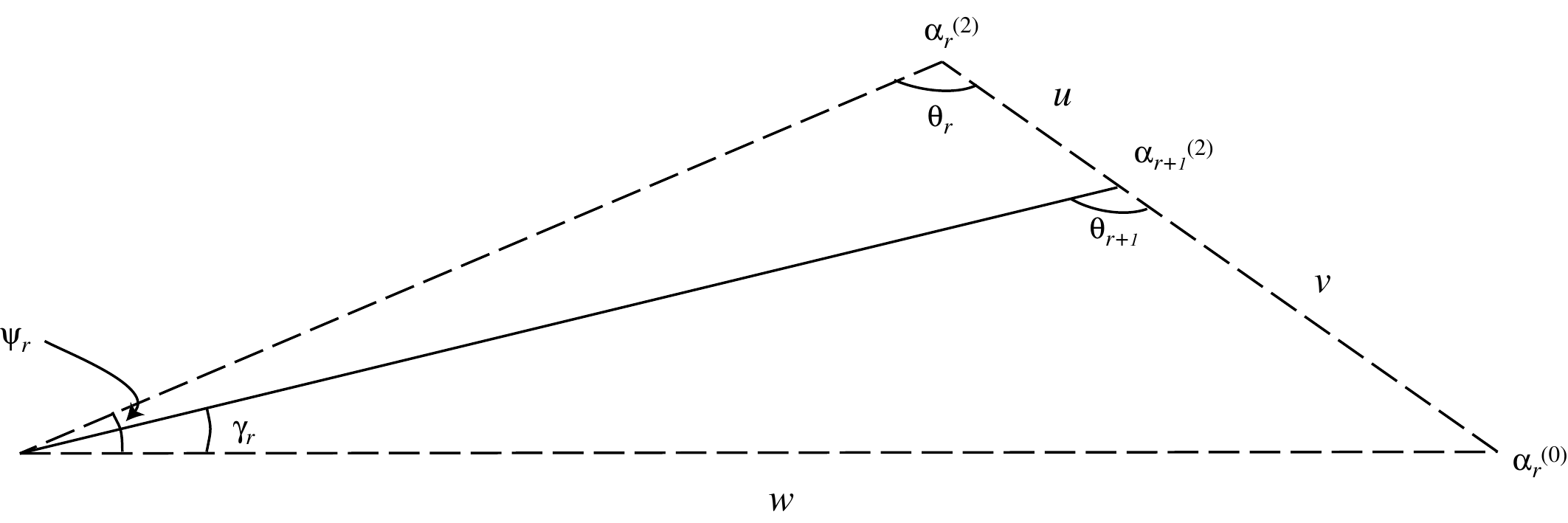}
\caption{Angles in the triangles $A_{\bn,r}$ and $A_{\bn,r+1}$}
\label{fig:evolve-detail}
\end{center}
\end{figure}

Applying the sine rule to the two triangles with base~$w$ gives
\[
\frac{\sin\psi_r}{\sin\gamma_r} =
\frac{u+v}{v}\,\frac{\sin\theta_r}{\sin\theta_{r+1}} > \frac43,
\]
since $\pi/2 < \theta_r < \theta_{r+1} <\pi$. Therefore
$\sin\psi_r-\sin\gamma_r > \frac{1}{4}\sin\psi_r >
\frac{1}{4}\sin\epsilon$. Since $0<\gamma_r<\psi_r<\pi/2$, this gives
\[
\theta_{r+1}-\theta_r = \psi_r - \gamma_r > \sin\psi_r-\sin\gamma_r > \frac{1}{4}\,\sin\epsilon.
\]
Therefore $\theta_r$ increases by at least $\frac{1}{4}\sin\epsilon$
as the index~$r$ passes through each block $1\,1\,1\,1\,1$ in $\bn$;
since there are infinitely many such blocks, this gives the required contradiction.
\end{proof}

\begin{theorem}
\label{thm:smooth}
Let $\bn\in\cP_S$. Then $\bal=\Phi^{-1}(\bn)$ is a smooth extreme
point of $\DF(\bn)$.
\end{theorem}
\begin{proof}
We can suppose without loss of generality that $n_0=n_1=1$, so that
the hypotheses of Lemma~\ref{lem:angle-to-pi} are satisfied. For if
not, let $\bm = \sigma^r(\bn)$ for some~$r$ with
$n_r=n_{r+1}=1$. Then, by Lemma~\ref{lem:lfs-recursive}, $\DF(\bn)$ is
the union of a polygon with the image of $\DF(\bm)$ under a projective
homeomorphism, and hence $\bal$ is a smooth extreme point of
$\DF(\bn)$ if and only if $\Phi^{-1}(\bm)$ is a smooth extreme point
of $\DF(\bm)$.

Suppose for a contradiction that $\bal$ is not a smooth extreme point,
so that there are distinct lines $L_1$ and $L_2$ through~$\bal$ which
do not disconnect $\DF(\bn)$. Let $\Theta<\pi$ be the angle between
$L_1$ and $L_2$ in the sector which contains $\DF(\bn)$. Then for
all~$r$, we have that $\VV{r}{0}$ and $\VV{r}{1}$ are contained in
this sector, while $\VV{r}{2}$ is contained in the opposite sector (in
order that $A_{\bn,r}$ contains $\bal$). It follows that
$\theta_r<\Theta$ for all~$r$, contradicting Lemma~\ref{lem:angle-to-pi}.
\end{proof}

\begin{remark}
Since $\cP_S\subset\Reg_S$, Theorem~\ref{thm:smooth}
says nothing about smoothness of extreme points in the exceptional
case. In fact, a similar but simpler argument can be used to show that
the endpoints of an exceptional interval are always smooth extreme
points, provided that the itinerary~$\bn$ has only finitely many
zeroes. 
\end{remark}

\subsection{Subsequential digit frequencies of $d_\beta(1)$}
Let $\beta>1$ be such that the digit frequency set~$\DF(\beta)$ is not
a polytope. Then $\DF(\beta) = \DF(\bn)$, where $\bn =
\bn(d_\beta(1))$, by Lemma~\ref{lem:expansion-translate}. Since
$\DF(\beta)$ is not a polytope, $\bn$ is not of rational or finite
type, so that $S(\bn)=d_\beta(1)$ by Definition~\ref{defn:n_w} and
Lemma~\ref{lem:gaps}.

In the regular case, when $\DF(\beta)$ has exactly one non-rational
extreme point $\bal$, the sequence $S(\bn)$ has digit
frequency $\bal$ by Theorem~\ref{thm:extreme-points} and
Facts~\ref{facts:infifacts}f). That is, $d_\beta(1)$ has well-defined
digit frequency $\delta_\beta(1) = \bal$.

In the exceptional case the digit frequency of $d_\beta(1) =
S(\bn)$ is not well defined, again by
Facts~\ref{facts:infifacts}f). In this case the interesting object is
the set of subsequential digit frequencies of $d_\beta(1)$. Let
$\bal^{(\beta,s)}$ be the rational element of $\Delta$ giving the
digit frequency of the initial subword $\Word{d_\beta(1)}{s}$ of
$d_\beta(1)$, and define~$F_\beta$ to be the set of limits of
convergent subsequences of $\left(\bal^{(\beta,s)}\right)_{s\ge 1}$.

$F_\beta$ is necessarily contained in the exceptional
simplex~$\Phi^{-1}(\bn)$. To see this, observe that for each $r\ge 0$, the vertices of the
simplex~$A_{\bn,r}$ are the digit frequencies of the words
$\Lambda_{\bn,r}(i)$ for $0\le i\le k-1$, and $S(\bn)$ is a
concatenation of these words. Therefore any subword of $S(\bn)$ which
is a concatenation of the $\Lambda_{\bn,r}(i)$ has digit frequency
contained in~$A_{\bn,r}$, and hence an arbitrary initial subword of
length~$s$ has digit frequency within distance $L/s$ of $A_{\bn,r}$,
where~$L$ is the maximum of the lengths of the
words~$\Lambda_{\bn,r}(i)$. It follows that
$F_\beta \subset A_{\bn,r}$ for all $r$, and so
$F_\beta\subset\Phi^{-1}(\bn) = \bigcap_{r\ge 0} A_{\bn,r}$.

\medskip

A natural and seemingly difficult question is whether or not it is
always the case that $F_\beta = \Phi^{-1}(\bn)$. The proof of
Theorem~\ref{thm:exceptional-digits} below, which treats the
case~$k=3$, depends strongly on the fact that the exceptional simplex
is one-dimensional, and so does not generalise to higher values
of~$k$. 

We will need a preliminary result, that bounded
itineraries are regular when~$k=3$. Since whether itineraries are
regular or exceptional is connected with their rate of growth, this
result appears obvious at first sight; but care has to be taken when
there are many zeroes in the itinerary.

\begin{lemma}
\label{lem:bounded-regular}
Let $k=3$ and $\bn\in\N^\N$. Then $\bn$ is regular in each of the
following two cases:
\begin{enumerate}[a)]
\item there are only finitely many~$r$ for which both $n_r>0$ and
  $n_{r+1}>0$;
\item $\bn$ is bounded.
\end{enumerate}
\end{lemma}

\begin{proof}
Given a non-negative $3\times 3$ matrix $A$, let
$f_A$ denote its projective action on~$\Delta$:
in other words, $f_A(\bal) = A\bal / ||A\bal||_1$.  For each $n\ge 0$, let
\[
A_n = 
\left(
\begin{array}{ccc}
0 & n+1 & n\\
1 & 0 & 0\\
0 & 1 & 1
\end{array}
\right)
\]
be the abelianization of the substitution~$\Lambda_n$, so that
$f_{A_n} = K_n^{-1}\colon\Delta\to\Delta$.

By a theorem of Birkhoff~\cite{Birkhoff}, if $A$ is strictly positive
then $f_A$ contracts the Hilbert metric~$\delta$ on the
interior~$\mathring{\Delta}$ of~$\Delta$ by a factor
$(\sqrt{d(A)}-1)/(\sqrt{d(A)}+1)$, where
\[
d(A) = \max_{1\le i,j,l,m\le 3} \frac{a_{il}a_{jm}}{a_{im}a_{jl}}
\]
is the largest number that can be obtained by choosing four elements
of~$A$ arranged in a rectangle, and dividing the product of the two
elements on one diagonal by the product of the two elements on the
other.

Moreover (Lemma~30 of~\cite{lex}) the matrices~$A_n$, while not
  strictly positive, have the property that they do not expand the
  Hilbert metric: $\delta(f_{A_n}(\bal), f_{A_n}(\bbeta)) \le
  \delta(\bal, \bbeta)$ for all $\bal,\bbeta\in\mathring{\Delta}$.

Recall that 
\[
\Phi^{-1}(\bn) = \bigcap_{r\ge 0} K_{n_0}^{-1}\circ
K_{n_1}^{-1}\circ\cdots\circ K_{n_r}^{-1}(\Delta) = \bigcap_{r\ge 0} f_{A_{n_0}}\circ
f_{A_{n_1}}\circ\cdots\circ f_{A_{n_r}}(\Delta).
\]
In order to prove that an itinerary~$\bn$ is regular, it therefore
suffices to find a constant~$C$ and infinitely many disjoint subwords
$n_r\ldots n_{r+s}$ of~$\bn$, each having the property that the product
$A_{n_r\ldots n_{r+s}} = A_{n_r}\,A_{n_{r+1}}\,\cdots\,A_{n_{r+s}}$ is
strictly positive and satisfies $d(A_{n_r\ldots n_{r+s}})\le C$.
\begin{enumerate}[a)]
\item Suppose that there are only finitely many~$r$ for which both
  $n_r>0$ and $n_{r+1}>0$. Since rational itineraries are
  regular, we can assume that $\bn$ has infinitely many non-zero
  entries, so that it has a tail of the form
  $p_1\,0^{k_1}\,p_2\,0^{k_2}\,\ldots$, where the $p_i$ and the $k_i$
  are strictly positive. Moreover, we can assume that infinitely many
  of the integers~$k_i$ are even, since otherwise $\Phi^{-1}(\bn)$ is
  contained in the faces of~$\Delta$ by Facts~\ref{facts:infifacts}i),
  and hence~$\bn$ is regular.

There are therefore infinitely many disjoint subwords of~$\bn$ which
either have the form $p\,0^{2r}\,q\,0^{2s}$ or have the form
$p\,0^{2r}\,q\,0^{2s-1}$, where $p$, $q$, $r$, and $s$ are positive
integers. Now a straightforward induction gives that
\[
A_0^{2r} = \left(
\begin{array}{ccc}
1 & 0 & 0 \\ 0 & 1 & 0 \\ r & r & 1
\end{array}
\right)
\qquad\text{ and }\qquad
A_0^{2r-1} = \left(
\begin{array}{ccc}
0 & 1 & 0 \\ 1 & 0 & 0 \\ r-1 & r & 1
\end{array}
\right).
\]
Therefore
\[
A_p\,A_0^{2r}\,A_q\,A_0^{2s} = \left(
\begin{array}{ccc}
1+p+pr+ps+pqrs&p+pr+ps+pqr+pqrs&p+pqr
\\ \noalign{\medskip}qs&1+q+qs&q\\ \noalign{\medskip}1+r+s+qrs&1+r+s+qr+qrs&1+qr
\end{array}
\right)
,
\]
a strictly positive matrix, each of whose entries is bounded below by,
but no more than five times than, the corresponding entry in the
matrix
\[
B = \left(
\begin{array}{ccc}
pqrs & pqrs & pqr \\ qs & qs & q \\qrs & qrs & qr
\end{array}
\right).
\]
Since $d(B)=1$ for all $p$, $q$, $r$, and $s$, it follows that
$d(A_pA_0^{2r}A_qA_0^{2s})$ is bounded above by 25. By a similar
calculation the same is true of $A_pA_0^{2r}A_qA_0^{2s-1}$
for all positive $p$, $q$, $r$, and $s$, which establishes the result.

\medskip

\item Let~$\bn$ be bounded. By a direct calculation, if $n_r>0$ and
  $n_{r+1}>0$, then the matrix $A_{n_r}A_{n_{r+1}}A_{n_{r+2}}$ is
  strictly positive. Since~$\bn$ is bounded, there are only finitely
  many possible values for this matrix, and hence there is a
  constant~$C$ such that $d(A_{n_r}A_{n_{r+1}}A_{n_{r+2}})\le C$ whenever
  $n_r$ and $n_{r+1}$ are both positive. This establishes that~$\bn$
  is regular when there are infinitely many such values of~$r$; and if
  there are only finitely many, then regularity follows from~a).
\end{enumerate}

\end{proof}

\begin{theorem}
\label{thm:exceptional-digits}
Let~$\beta\in(2,3)$ be such that $\bn = \bn(d_\beta(1))$ is
exceptional. Then $F_\beta = \Phi^{-1}(\bn)$.
\end{theorem}

\begin{proof}
Since $\beta\in(2,3)$ we have $k=3$. Let $\ell=\Phi^{-1}(\bn)$ be the
exceptional interval, and~$L$ denote the length of $\ell$.

The distance between $\bal^{(\beta,s)}$ and $\bal^{(\beta,s+1)}$ is at
most $1/s$, and so it is enough to prove that the two endpoints
$\bv_1$ and $\bv_2$ of $\ell$ lie in $F_\beta$. We shall show that for
all $\epsilon>0$ there are natural numbers $r_1$ and~$r_2$ such that
$\bal_{r_i}^{(2)}\in B_\epsilon(\bv_i)$ for each~$i$, which
establishes the result since $\bal_r^{(2)}$ is the digit frequency of
the initial subword $\Lambda_{\bn,r}(2)$ of $S(\bn)=d_\beta(1)$.

Suppose without loss of generality that $\epsilon < L/4$, and let~$R$
be large enough that $d_H(A_{\bn,r},\ell)<\epsilon/8$ for all $r\ge
R$. In particular,
\begin{equation}
\label{eq:alpha-close}
\text{For all }r\ge R\text{ and for each }a\in\{1,2\}, \text{ there exists }j\in\{0,1,2\}\text{
  with }\bal_r^{(j)}\in B_{\epsilon/8}(\bv_a).
\end{equation}

Pick $r\ge R$ with $n_{r+1}\ge 4L/\epsilon$, which is possible since
the exceptional itinerary~$\bn$ is unbounded by
Lemma~\ref{lem:bounded-regular}. Using~(\ref{eq:alpha-eqs-1})
and~(\ref{eq:alpha-eqs-2}), we have that
\begin{eqnarray*}
\VV{r+1}{1} &=& \frac{(n_{r+1}+1)\LL{r}{0}\VV{r}{0}+\LL{r}{2}\VV{r}{2}}
{(n_{r+1}+1)\LL{r}{0}+\LL{r}{2}} \qquad\text{and}\\
\VV{r+1}{2} &=& \frac{n_{r+1}\LL{r}{0}\VV{r}{0}+\LL{r}{2}\VV{r}{2}}
{n_{r+1}\LL{r}{0}+\LL{r}{2}}
\end{eqnarray*}
lie on the line segment joining $\VV{r}{2}$ to $\VV{r}{0}$. Since
$d_H(A_{\bn,r},\ell)<\epsilon/8$, this segment has length $\Lambda <
L+\epsilon/4 < 2L$. The distances of $\VV{r+1}{1}$ and $\VV{r+1}{2}$
from $\VV{r}{0}$ are therefore given by
$\LL{r}{2}\Lambda/((n_{r+1}+1)\LL{r}{0} + \LL{r}{2})$ and 
$\LL{r}{2}\Lambda/(n_{r+1}\LL{r}{0} + \LL{r}{2})$. Subtracting these
gives $d(\VV{r+1}{1},\VV{r+1}{2})< \Lambda/n_{r+1} < 2L/(4L/\epsilon)
= \epsilon/2$. It follows from~(\ref{eq:alpha-close}), using
$L>4\epsilon$, that $\VV{r+1}{0}$ is within $\epsilon/8$ of one of the
endpoints of $\ell$, say $\bv_1$; while both $\VV{r+1}{1}$ and
$\VV{r+1}{2}$ lie in $B_\epsilon(\bv_2)$.

Since $\VV{r+1}{2}\in B_\epsilon(\bv_2)$, it remains to find~$r'$ with
$\VV{r'}{2}\in B_\epsilon(\bv_1)$. We shall show that, if
$\VV{s}{0}\in B_{\epsilon/8}(\bv_1)$ and $\VV{s}{1}, \VV{s}{2}\in
B_{\epsilon}(\bv_2)$ for some~$s$, then
\begin{enumerate}[a)]
\item if $n_{s+1}=0$ then
  the same conditions hold when $s$ is replaced with $s+2$; and
\item if $n_{s+1}>0$ then $\VV{s+1}{2}\in B_\epsilon(\bv_1)$.
\end{enumerate}
This will complete the proof, since there is some $p\ge1$ for which
$n_{r+2p}>0$; for otherwise, by Facts~\ref{facts:infifacts}i), the
exceptional interval~$\ell$ would be contained in one of the
(one-dimensional) faces of~$\Delta$, contradicting the fact that it
contains no rational points.

\medskip

For a), observe that $\VV{s+1}{0} = \VV{s}{1}\in B_\epsilon(\bv_2)$ and, since $n_{s+1}=0$,
we have $\VV{s+1}{2}=\VV{s}{2}\in
B_\epsilon(\bv_2)$. By~(\ref{eq:alpha-close}), we have $\VV{s+1}{1}\in
B_{\epsilon/8}(\bv_1)$.

Then $\VV{s+2}{0}=\VV{s+1}{1}\in B_{\epsilon/8}(\bv_1)$, and both
$\VV{s+2}{1}$ and $\VV{s+2}{2}$ lie on the line segment joining
$\VV{s+1}{2}$ to $\VV{s+1}{0}$, which is contained in
$B_\epsilon(\bv_2)$, as required.

\medskip

For b), we have as in~a) that $\VV{s+1}{0}=\VV{s}{1}\in
B_\epsilon(\bv_2)$. Both $\VV{s+1}{1}$ and $\VV{s+1}{2}$ lie on the
line segment joining $\VV{s}{2}\in B_\epsilon(\bv_2)$ to $\VV{s}{0}\in
B_{\epsilon/8}(\bv_1)$, and $\VV{s+1}{1}$ is closer than $\VV{s+1}{2}$
to $\VV{s}{0}$. Therefore $\VV{s+1}{1}\in B_{\epsilon/8}(\bv_1)$
by~(\ref{eq:alpha-close}). Calculating the ratio of the distances from
$\VV{s+1}{1}$ and $\VV{s+1}{2}$ to $\VV{s}{0}$ gives
\[
\frac{d(\VV{s}{0}, \VV{s+1}{2})}{d(\VV{s}{0}, \VV{s+1}{1})} = 
\frac{(n_{s+1}+1)\LL{s}{0}+\LL{s}{2}}{n_{s+1}\LL{s}{0}+\LL{s}{2}} 
< \frac{n_{s+1}+1}{n_{s+1}} \le 2
\]
since $n_{s+1}\ge 1$. Therefore $d(\VV{s}{0},
\VV{s+1}{2})<\epsilon/2$, so that
$d(\VV{s+1}{2},\bv_1)<\epsilon/2+\epsilon/8 < \epsilon$ as required.

\end{proof}

\subsection{Calculations for specific values of~$\beta$}

We finish by addressing the practical problem of computing the
function $\beta\mapsto\bn(w_\beta)$, so that digit frequency sets of
specific numbers~$\beta$ can be determined.

We first define a left inverse~$\Gamma_n$ of each
$\Lambda_n\colon\Sigma\to\Sigma$. Informally, to
determine~$\Gamma_n(w)$ we repeatedly remove $\Lambda_n$-images of digits from the
front of~$w$ until we are no longer able to do so: at that stage we
complete $\Gamma_n(w)$ with $\overline{0}$ if the remaining block is
smaller than anything in the image of~$\Lambda_n$, and with
$\overline{k-1}$ if the remaining block is larger than anything in the
image of~$\Lambda_n$.

\begin{defn}[$\Gamma_n\colon\Sigma\to\Sigma$]
For each $n\in\N$, define $\Gamma_n\colon\Sigma\to\Sigma$ as
follows. Let~$w\in\Sigma$. Then
\begin{itemize}
\item If $w_0 = 0$ then $\Gamma_n(w) = \overline{0}$.
\item If $1\le w_0 \le k-2$ then $\Gamma_n(w) = (w_0-1)\,
  \Gamma_n(\sigma(w))$. 
\item If $w = (k-1)\,0^{n+1}\,v$ then $\Gamma_n(w) =
  (k-2)\,\Gamma_n(v)$.
\item If $w = (k-1)\, 0^n\,v$ with $v_0>0$ then $\Gamma_n(w) =
  (k-1)\,\Gamma_n(v)$.
\item If $w_0 = k-1$ and there is some $1\le r\le n$ with $w_r \not=0$ then
  $\Gamma_n(w) = \overline{k-1}$.
\end{itemize}
\end{defn}

Notice that $\Gamma_n$ is increasing (the five cases in its
definition are listed in order of increasing~$w$, and the five
corresponding outputs are also in increasing order) and continuous (if
$w$ and $w'$ agree to $r(n+2)$ digits for any $r$, then $\Gamma_n(w)$
and $\Gamma_n(w')$ agree to $r$ digits).

The following lemma gives a recursive algorithm for
calculating~$\bn(w)$ in the rational or finite case, or for reading
off successive entries of $\bn(w)$ in the general case.
\begin{lemma}
\label{lem:calc-bn-w}
Let~$w\in\Sigma$ with $w_0 = k-1$.  

If $w = (k-1)\,\overline{0}$ then $\bn(w) = \infty$.  Otherwise, let
$n\ge 0$ be such that $w = (k-1)\, 0^n\, v$ where $v_0\not=0$. Then
\begin{itemize}
\item If $v$ starts $1^m\,0$ for some~$m>0$, then $\bn(w) =
  (n+1)\,\overline{0}$.
\item If $v = \overline{1}$ then $\bn(w) = n\,\infty$.
\item Otherwise $n_0(w) = n$, and $\sigma(\bn(w)) =
  \bn(\Gamma_n(w))$. 
\end{itemize}
\end{lemma}

\begin{proof} 
Recall that $\bn(w) = \max\,\{\bm\in\cN\,:\,S(\bm)\le w\}$. Since
$S(\infty) = (k-1)\,\overline{0}$ and $S(N\,\infty) =
(k-1)\,0^N\,\overline{1}$ for each~$N\in\N$, it is immediate that
$\bn(w) = \infty$ if and only if \mbox{$w = (k-1)\,\overline{0}$}.

Now suppose that $w = (k-1)\,0^n\,v$ where $v_0\not=0$. Observe that,
in increasing order,
\begin{eqnarray*}
S((n+1)\,\overline{0}) &=& \overline{(k-1)\,0^{n+1}},\\
S(n\,\infty) &=& (k-1)\,0^n\,\overline{1},\\
S(n\,\overline{0}) &=& \overline{(k-1)\,0^n}, \text{ \ and}\\
S((n-1)\,\infty) &=& (k-1)\,0^{n-1}\,\overline{1}.
\end{eqnarray*}
Moreover $(n+1)\,\overline{0}$ and $n\,\infty$
are consecutive elements of $\cN$, as are $n\,\overline{0}$ and
$(n-1)\,\infty$; and~$n\,\infty$ is the limit of a strictly
decreasing sequence in~$\cN$. It follows that $\bn(w)
= (n+1)\,\overline{0}$ if and only if $ v < \overline{1}$ -- that is, if and
only if $v$ starts $1^m\,0$ for some~$m$; that otherwise $n_0(w) = n$;
and that $\bn(w) = n\,\infty$ if and only if $v=\overline{1}$.

Suppose then that $v>\overline{1}$. Using $n_0(w) = n$, we have
\begin{eqnarray*}
\sigma(\bn(w)) &=& \max\,\{\bm\in\cN\,:\,S(n\,\bm) \le w\}\\
&=& \max\,\{\bm\in\cN\,:\,\Gamma_n(S(n\,\bm)) \le \Gamma_n(w)\}.
\end{eqnarray*}
The second equality uses that~$\Gamma_n$ is increasing. First, if
$S(n\,\bm)\le w$ then $\Gamma_n(S(n\,\bm)) \le \Gamma_n(w)$. Second,
if $\Gamma_n(S(n\,\bm))\le \Gamma_n(w)$ then either $S(n\,\bm) \le w$
or $\Gamma_n(S(n\,\bm)) = \Gamma_n(w)$. In the latter case, there are
two possibilities:
\begin{enumerate}[a)]
\item $w = \Lambda_n(v)$ is in the image of $\Lambda_n$. Then
  $\Gamma_n(S(n\,\bm)) = \Gamma_n(w)$ reads
  $\Gamma_n(\Lambda_n(S(\bm))) = \Gamma_n(\Lambda_n(v))$, so $S(\bm) =
  v$, and $S(n\,\bm) = \Lambda_n(v) = w$.
\item $w$ is not in the image of~$\Lambda_n$, so that $\Gamma_n(w)$
  ends either with $\overline{0}$ or with $\overline{k-1}$ by
  definition of $\Gamma_n$. Since $\Gamma_n(S(n\,\bm)) = S(\bm)$, the
  only possibility is that $\Gamma_n(w) = \overline{k-1}$ and $\bm =
  \overline{0}$. Then $S(n\,\bm) = \overline{(k-1)\,0^n}$, while
  $\Gamma_n(w) = \overline{k-1}$ means by definition of~$\Gamma_n$
  that either $w = \overline{(k-1)\,0^n}$ or $w =
  ((k-1)\,0^n)^s\,(k-1)\,0^r\,v$ for some $s\ge0$, $r<n$, and
  $v\in\Sigma$ with $v_0>0$. So $S(n\,\bm)\le w$ in this case
  also.
\end{enumerate}

Therefore
\[\sigma(\bn(w)) = \max\,\{\bm\in\cN\,:\,\Gamma_n(S(n\,\bm)) \le
\Gamma_n(w)\}  = \bn(\Gamma_n(w))\]
using $\Gamma_n(S(n\,\bm)) = S(\bm)$, as required.
\end{proof}

\begin{example}
\label{ex:calc-bn-w}
Let~$\beta = 2.1901$ so that $k=3$. Calculating the orbit entries
$f_\beta^r(1)$ for $0\le r\le 12$, we find (and could establish
rigorously) that $w_\beta$ starts with the digits
$2\,0\,0\,1\,2\,0\,0\,1\,2\,0\,0\,0\,0\,\ldots$.  Applying the
algorithm of Lemma~\ref{lem:calc-bn-w}, we have
\begin{itemize}
\item $w_\beta = 2\,0^n\,v$ where $n=2$ and $v$ does not start
  $1^m\,0$. Therefore $n_0(w_\beta) = 2$, and $\sigma(\bn(w_\beta)) =
  \bn(\Gamma_2(w_\beta))$.
\item $\Gamma_2(w_\beta) = 2\,0\,2\,0\,1\,\overline{0} = 2\,0^n\,v$
  where~$n=1$ and $v$ does not start~$1^m\,0$. Therefore $n_1(w_\beta)
  = 1$, and $\sigma^2(\bn(w_\beta)) = \bn(\Gamma_1\Gamma_2(w_\beta))$.
\item $\Gamma_1\Gamma_2(w_\beta) = 2\,2\,\overline{0} = 2\,0^n\,v$
  where~$n=0$ and $v$ does not start~$1^m\,0$. Therefore $n_2(w_\beta)
  = 0$, and $\sigma^3(\bn(w_\beta)) =
  \bn(\Gamma_0\Gamma_1\Gamma_2(w_\beta))$.
\item $\Gamma_0\Gamma_1\Gamma_2(w_\beta) = 2\,1\,\overline{0} = 2\,0^n\,v$
  where~$n=0$ and $v = 1\,\overline{0}$. Therefore\\
  $\bn(\Gamma_0\Gamma_1\Gamma_2(w_\beta)) = 1\,\overline{0}$.
\end{itemize}

\medskip

In summary, $\bn(w_\beta) = 2\,1\,0\,1\,\overline{0}$, from which we can
compute $\DF(\beta)$ as in Example~\ref{ex:ep}a).

\end{example}

\begin{example}
\label{ex:markov-again}
Consider the example with $d_\beta(1) = 2\,1\,2\,1\,\overline{0}$
which was treated in Section~\ref{sec:intro} using Markov partition
techniques. The algorithm of Lemma~\ref{lem:calc-bn-w} gives
$\bn(d_\beta(1))=0\,1\,1\,\overline{0}$, so that, by
Theorem~\ref{thm:extreme-points}, the non-trivial extreme points of
$\DF(\beta)$ are
\begin{eqnarray*}
 K_0^{-1}(0,1,0) &=& (1/2,\,0,\,1/2),\\
K_0^{-1}K_1^{-1}(0,1,0) &=& (0,\,2/3,\,1/3), \text{ and}\\
K_0^{-1}K_1^{-1}K_1^{-1}(0,0,1) &=& (1/4,\,1/4,\,1/2),
\end{eqnarray*}
in agreement with the Markov partition calculation.

\end{example}

\bibliographystyle{amsplain}
\bibliography{betarefs}

\end{document}